\documentclass[11pt, reqno]{article}
\usepackage{amsmath,amssymb,amsfonts,amsthm}
\usepackage{color}

\usepackage[colorlinks=true,linkcolor=blue,citecolor=blue,urlcolor=blue,pdfborder={0 0 0}]{hyperref}
\usepackage{cleveref}

\usepackage{caption}
\usepackage{thmtools}

\usepackage{mathrsfs}

\crefname{theorem}{Theorem}{Theorems}
\crefname{thm}{Theorem}{Theorems}
\crefname{lemma}{Lemma}{Lemmas}
\crefname{claim}{Claim}{Claims}
\crefname{lem}{Lemma}{Lemmas}
\crefname{remark}{Remark}{Remarks}
\crefname{prop}{Proposition}{Propositions}
\crefname{proposition}{Proposition}{Propositions}
\crefname{defn}{Definition}{Definitions}
\crefname{definition}{Definition}{Definitions}
\crefname{corollary}{Corollary}{Corollaries}
\crefname{conjecture}{Conjecture}{Conjectures}
\crefname{question}{Question}{Questions}
\crefname{chapter}{Chapter}{Chapters}
\crefname{section}{Section}{Sections}
\crefname{part}{Part}{Parts}
\crefname{figure}{Figure}{Figures}

\theoremstyle{plain}
\newtheorem{thm}{Theorem}[section]

\newtheorem{lemma}[thm]{Lemma}

\newtheorem{theorem}[thm]{Theorem}

\newtheorem{corollary}[thm]{Corollary}

\newtheorem{prop}[thm]{Proposition}
\newtheorem{proposition}[thm]{Proposition}
\newtheorem{conjecture}[thm]{Conjecture}

\theoremstyle{definition}

\newtheorem{definition}[thm]{Definition}

\theoremstyle{remark}
\newtheorem{remark}[thm]{Remark}
\newtheorem*{remark*}{Remark}

\numberwithin{equation}{section}

\renewcommand{\P}{\mathbb P}
\newcommand{\E}{\mathbb E}

\newcommand{\R}{\mathbb R}
\newcommand{\Z}{\mathbb Z}
\newcommand{\N}{\mathbb N}

\newcommand{\cG}{\mathcal G}

\newcommand{\sA}{\mathscr A}
\newcommand{\sB}{\mathscr B}

\newcommand{\sD}{\mathscr D}

\usepackage[margin=1.045in]{geometry}
\usepackage{extarrows}
\usepackage{titling}
\usepackage{commath}
\usepackage{mathtools}
\usepackage{titlesec}
\newcommand{\eps}{\varepsilon}
\usepackage{bbm}
\usepackage{setspace}
\usepackage{enumitem}
\usepackage{cite}

\usepackage{tikz}

\newcommand{\lra}{\leftrightarrow}

\newcommand{\Aut}{\text{\textup{Aut}}\,}
\newcommand*{\diam}{\mathop{\textup{diam}}\nolimits}

\newcommand{\ph}{\varphi}

\numberwithin{equation}{section}

\usepackage{mathrsfs}

\newcommand{\bG}{\mathbf{G}}
\newcommand{\bP}{\mathbf{P}}
\newcommand{\bp}{\mathbf{p}}

\newcommand{\Ceff}{\mathscr{C}_\mathrm{eff}}
\newcommand{\sign}{\operatorname{sgn}}
\newcommand{\Cay}{\operatorname{Cay}}

\makeatletter
\newcommand{\xRrightarrow}[2][]{\ext@arrow 0359\Rrightarrowfill@{#1}{#2}}
\newcommand{\Rrightarrowfill@}{\arrowfill@\equiv\equiv\Rrightarrow}
\newcommand{\xLleftarrow}[2][]{\ext@arrow 3095\Lleftarrowfill@{#1}{#2}}
\newcommand{\Lleftarrowfill@}{\arrowfill@\Lleftarrow\equiv\equiv}
\newcommand{\xLleftRrightarrow}[2][]{\ext@arrow 3399\LleftRrightarrowfill@{#1}{#2}}
\newcommand{\LleftRrightarrowfill@}{\arrowfill@\Lleftarrow\equiv\Rrightarrow}
\makeatother

\pretitle{\begin{flushleft}\Large}
\posttitle{\par\end{flushleft}\vskip 0.5em
}
\preauthor{\begin{flushleft}}
\postauthor{
\\
\vspace{0.5em}
\footnotesize{Division of Physics, Mathematics and Astronomy, California Institute of Technology\\ Email: \href{mailto:t.hutchcroft@caltech.edu}{t.hutchcroft@caltech.edu}}\\
\footnotesize{School of Mathematics, University of Bristol\\ Email: \href{mailto:m.tointon@bristol.ac.uk}{m.tointon@bristol.ac.uk}} 
\end{flushleft}}
\predate{\begin{flushleft}}
\postdate{\par\end{flushleft}}
\title{\bf Non-triviality of the phase transition for percolation on finite transitive graphs}

\renewenvironment{abstract}
 {\par\noindent\textbf{\abstractname.}\ \ignorespaces}
 {\par\medskip}

\titleformat{\section}
  {\normalfont\large \bf}{\thesection}{0.4em}{}

\author{{\bf Tom Hutchcroft and Matthew Tointon}}

\begin{document}
\maketitle

\begin{abstract}
We prove that if $(G_n)_{n
\geq1}=((V_n,E_n))_{n\geq 1}$ is a sequence of finite, vertex-transitive graphs with bounded degrees and $|V_n|\to\infty$ that is at least $(1+\eps)$-dimensional for some $\eps>0$ in the sense that
\[\diam (G_n)=O\!\left(|V_n|^{1/(1+\eps)}\right) \qquad \text{ as $n\to\infty$}\]
then this sequence of graphs has a non-trivial phase transition for Bernoulli bond percolation.  More precisely, we prove under these conditions that for each $0<\alpha <1$ there exists $p_c(\alpha)<1$ such that for each $p\geq p_c(\alpha)$, Bernoulli-$p$ bond percolation on $G_n$ has a cluster of size at least $\alpha |V_n|$ with probability tending to $1$ as $n\to \infty$.
In fact, we prove more generally that there exists a universal constant $a$ such that the same conclusion holds whenever
\[\diam (G_n)=O\!\left(\frac{|V_n|}{(\log |V_n|)^a}\right) \qquad \text{ as $n\to\infty$.}\]
This verifies a conjecture of Benjamini~(2001) up to the value of the constant $a$, which he suggested should be $1$. We also prove a generalization of this result to quasitransitive graph sequences with a bounded number of vertex orbits.

A key step in our argument is a direct proof of our result when the graphs $G_n$ are all Cayley graphs of Abelian groups, in which case we show that one may indeed take $a=1$. This result relies crucially on a new theorem of independent interest stating roughly that balls in arbitrary Abelian Cayley graphs can always be approximated by boxes in  $\Z^d$ with the standard generating set. Another key step is to adapt the methods of Duminil-Copin, Goswami, Raoufi, Severo, and Yadin (\emph{Duke Math.\ J.}, 2020) from infinite graphs to finite graphs. This adaptation also leads to an isoperimetric criterion for infinite graphs to have a nontrivial uniqueness phase (i.e., to have $p_u<1$) which is of independent interest. We also prove that the set of possible values of the critical probability of an infinite quasitransitive graph has a \emph{gap} at $1$ in the sense that for every $k,n<\infty$ there exists $\eps>0$ such that every infinite graph $G$ of degree at most $k$ whose vertex set has at most $n$ orbits under $\Aut(G)$ either has $p_c=1$ or $p_c\leq 1-\eps$.
\end{abstract}

\newpage

\setcounter{tocdepth}{2}
\tableofcontents

\newpage

\setstretch{1.1}

\section{Introduction}

In \textbf{Bernoulli bond percolation}, the edges of a connected, locally finite graph $G=(V,E)$ are chosen to be either deleted (\textbf{closed}) or retained (\textbf{open}) independently at random with retention probability $p\in [0,1]$. We write $\P_p=\P_p^G$ for the law of the resulting random subgraph and refer to the connected components of this subgraph as \textbf{clusters}. When $G$ is infinite, the \textbf{critical probability} $p_c=p_c(G)$ is defined to be
\[
p_c = \sup\bigl\{p \in [0,1] : \text{every cluster is finite $\P_p^G$-almost surely} \bigr\}.
\]
It is a fact of fundamental importance that percolation undergoes a non-trivial phase transition in the sense that $0<p_c<1$ on most infinite graphs.
  Indeed, in the traditional setting of Euclidean lattices\footnote{As is standard in the area, we abuse notation by writing $\Z^d$ both for the free Abelian group of rank $d$ and its standard Cayley graph, the hypercubic lattice.}, it is a classical consequence of the \emph{Peierls argument} that $p_c(\Z^d)<1$ for every $d \geq 2$ \cite[Theorem 1.10]{grimmett2010percolation}, while the complementary bound $p_c \geq 1/(M-1)>0$ holds for every graph of maximum degree $M$ by elementary path-counting arguments \cite[Chapter 3]{grimmett2018probability}.  Besides the obvious importance of such  results to the study of percolation itself, non-triviality of the percolation phase transition also implies the non-triviality of the phase transition for many other important models in probability and mathematical physics, including the random cluster, Ising, and Potts models; see \cite[Section 3.4]{GrimFKbook} and \cref{rem:other_models} below.

Since the pioneering work of Benjamini and Schramm \cite{bperc96}, there has been substantial interest in understanding the behaviour of percolation beyond the traditional setting of Euclidean lattices. A natural level of generality is that of \textbf{(vertex-)transitive} graphs, i.e., graphs for which any vertex can be mapped to any other vertex by a graph automorphism. More generally, one can also consider  \textbf{quasitransitive} graphs, i.e., graphs $G=(V,E)$ for which the action of the automorphism group $\Aut(G)$ on $V$ has finitely many orbits. We refer the reader to \cite{grimmett2010percolation} for background on percolation in the Euclidean context and \cite{LP:book} for percolation on general transitive graphs.

Benjamini and Schramm conjectured in the same work that $p_c<1$ for every infinite, connected, quasitransitive graph that has superlinear volume growth (or, equivalently, is not rough-isometric to $\Z$). The final remaining cases of this conjecture were finally resolved in the recent breakthrough work of Duminil-Copin, Goswami, Raoufi, Severo, and Yadin \cite{1806.07733}. Several important cases of the conjecture have been known for much longer, including the cases that the graph in question has polynomial volume growth (see the discussion in \cite[Section 1]{1806.07733}), exponential volume growth \cite{lyons1995random}, or is a Cayley graph of a finitely presented group~\cite{MR1622785,MR2286517}.
Indeed, the precise result proven in \cite{1806.07733} is that $p_c<1$ for every infinite, connected, bounded degree graph satisfying a $d$-dimensional isoperimetric inequality for some $d > 4$ (we recall what this means later in the introduction). The general conjecture follows since every infinite transitive graph that does not satisfy this condition must have polynomial volume growth by a theorem of Coulhon and Saloff-Coste \cite{MR1232845}, and hence is covered by previous results.
Further works concerning this problem include \cite{MR1841989,MR3607808,MR1634419,MR3865659,MR3520023}; see the introduction of \cite{1806.07733} for a detailed guide to the relevant literature.

The purpose of this paper is to develop an analogous theory for \emph{finite} transitive and quasitransitive graphs. For such graphs there are multiple, potentially inequivalent ways to define the supercritical phase (see \cref{rem:pc_def}); we work with the most stringent such definition, which requires the existence of a \emph{giant cluster} whose volume is proportional to that of the entire graph. 
Given a graph $G=(V,E)$ and parameters $\alpha,q\in(0,1)$, define the \emph{critical probability} 
\[
p_c(G,\alpha,q)=\inf\Bigl\{p\in[0,1]:\P_p\bigl(\text{there exists an open cluster $K$ such that $|K|\ge\alpha|G|$}\bigr)\ge q\Bigr\}.
\]
Of course, we trivially have that $p_c(G,\alpha,q)<1$ whenever $G$ is a finite connected graph and $\alpha,q \in (0,1)$, so the relevant problem is instead to find conditions on \emph{sequences} of graphs $(G_n)_{n\geq 1}$ guaranteeing that $p_c(G_n,\alpha,q)$ is bounded away from $1$ as $n\to \infty$.

By analogy with the infinite case, we would ideally like to say that this holds whenever the graphs $G_n$ are ``not one-dimensional'' in some sense. However, a precise interpretation of what this should mean is much more delicate to determine in the finite case. In the infinite case, well-known results of Trofimov~\cite{MR811571}, Gromov~\cite{gromov81poly} and Bass and Guivarc'h~\cite{bass72poly-growth,guivarch73poly-growth} imply that for every transitive graph $G$ of at most polynomial growth there exists an integer $d$ such that for each $n\in\N$ the ball of radius $n$ in $G$ has cardinality bounded above and below by constants times $n^d$. In particular, any such graph with superlinear growth has at least quadratic growth. Sequences of finite transitive graphs, on the other hand, may have volume growth that is only very barely superlinear. Indeed, if one considers a highly asymmetric torus $(\Z / n \Z) \times ( \Z / m \Z)$ with $m=m(n) \ll n$, one can show 
that the percolation phase transition is non-trivial if and only if $m = \Omega( \log n)$ (see \cref{lem:mal-pak.d=2,rem:log.needed}). This example led Benjamini \cite[Conjecture 2.1]{benjamini2001percolation} to make the following conjecture.
\begin{conjecture}[Benjamini 2001]\label{conj:fin.perc}
For every $k \geq 1$, $\lambda>0$ and $\alpha,q\in(0,1)$ there exists $\eps=\eps(k,\lambda,\alpha,q)>0$ such that if $G=(V,E)$ is a finite, vertex-transitive graph of degree at most $k$ satisfying
\[
\diam(G)\le\frac{\lambda|V|}{\log|V|}
\]
then $p_c(G,\alpha,q)\le1-\eps$. 
\end{conjecture}

This conjecture has been verified for expander graphs (which automatically have diameter at most logarithmic in their volume) by Alon, Benjamini, and Stacey \cite{MR2073175}; see also \cite{MR2773031,sarkar2018note,MR4101348,MR3005730} for more refined results in this case. Malon and Pak \cite{MR2238046} verified the conjecture for \emph{Cayley graphs of Abelian groups generated by Hall bases} (a.k.a.\ hypercubic tori) and expressed a belief that the conjecture should be \emph{false} in general. Note that the case of the \emph{symmetric} hypercubic torus $(\Z/n\Z)^d$ is classical; indeed it follows from the work of Grimmett and Marstrand \cite{MR1068308} that such a torus has a giant cluster with high probability for every $p>p_c(\Z^d)$. 

The main result of the present paper verifies \cref{conj:fin.perc} for all but the very weakest instances of the hypotheses, applying in particular to graphs $G=(V,E)$ that are $(1+\eps)$-dimensional in the sense that $\diam(G) \leq \lambda |V|^{1/(1+\eps)}$ for some $\eps>0$ and $\lambda<\infty$. In fact, all of our results also apply more generally in the quasitransitive case. Of course, since every finite graph is trivially quasitransitive, we must define our quasitransitivity assumption quantitatively if it is to have any impact. Given $n\in\N$, we therefore define a graph $G=(V,E)$ to be \textbf{$n$-quasitransitive} if the action of $\Aut(G)$ on $V$ has at most $n$ distinct orbits, so that transitive graphs are $1$-quasitransitive.

\begin{theorem}\label{thm:fin.perc}
There exists an absolute constant $a\ge1$ such that for every $k,n \geq 1$, $\lambda>0$ and $\alpha,q\in(0,1)$ there exists $\eps=\eps(k,n,\lambda,\alpha,q)>0$ such that if $G=(V,E)$ is a finite, $n$-quasitransitive graph of degree at most $k$ satisfying
\[
\diam(G)\le\frac{\lambda|V|}{(\log|V|)^a}
\]
then $p_c(G,\alpha,q)\le1-\eps$. 
\end{theorem}

A crucial ingredient in our argument is a direct proof of \cref{thm:fin.perc} for an arbitrary Abelian Cayley graph. In fact, in this case we show that the exponent $a$ can be taken to be $1$, resolving \cref{conj:fin.perc} in full and significantly generalising the results of Malon and Pak \cite{MR2238046}.

\begin{theorem}\label{thm:fin.perc.ab}
Let $k \geq 1$, $\lambda>0$, and $\alpha,q\in(0,1)$. Then there exists $\eps=\eps(k,\lambda,\alpha,q)>0$ such that if $G=(V,E)$ is a Cayley graph of a finite Abelian group with degree at most $k$ satisfying 
\begin{equation}\label{eq:diam.ab}
\diam(G)\le\frac{\lambda|V|}{\log|V|}
\end{equation}
then $p_c(G,\alpha,q)\le1-\eps$.
\end{theorem}

The proof of \cref{thm:fin.perc.ab} relies on an interesting additive-combinatorial theorem, \cref{prop:boxes}, which roughly states that balls in Abelian Cayley graphs can always be approximated, in a certain sense, by boxes in $\Z^d$ with the standard generating set. This result is in turn proven using a notion of taking a \emph{quotient modulo a subset} of a group; we believe that both \cref{prop:boxes} and this notion are new, and are optimistic that both will have significant further applications in future work.

\subsection{About the proof} 
Our work relies crucially on the quantitative structure theory of finite transitive graphs as developed in a series of works by Tessera and the second author \cite{MR3877012,tt.Lie,tt.Trof,tt.resist}, building on Breuillard, Green and Tao's celebrated structure theorem for finite approximate groups \cite{bgt12}. Roughly speaking, this theory states that for each integer $d\geq 1$ and each locally finite, vertex-transitive graph $G$, there exists a scale $m$ such that $G$ looks at least $(d+1)$-dimensional on scales smaller than $m$ and looks like a nilpotent group of dimension\footnote{There are a number of natural notions of the dimension of a nilpotent group, not all of which agree. Here we use the term \emph{dimension} to refer to the degree of polynomial growth or to the isoperimetric dimension; these always agree by a result of Coulhon and Saloff-Coste~\cite{MR1232845} (see \cref{sec:structuretheory} for more details).} at most $d$ on scales larger than $m$; see \cref{sec:structuretheory} for detailed statements. 
Again, we stress that it is indeed possible for a vertex-transitive graph to look higher-dimensional on small scales than it does on large scales: Consider for example the torus $(\Z/n\Z)^{d_1} \times (\Z/m \Z)^{d_2-d_1}$ with $d_2>d_1$ and $m \ll n$, which looks $d_2$-dimensional on scales up to $m$ and $d_1$-dimensional on scales $k$ satisfying $m \ll k \leq n$. Given these structure-theoretic results, the proof of \cref{thm:iso.inf} has three main components: an analysis of percolation on Cayley graphs of finite nilpotent groups, an analysis of percolation under a high-dimensional isoperimetric condition using the techniques of \cite{1806.07733}, and finally an argument showing that we can patch together the outputs of these two analyses at the relevant crossover scale if necessary.

\medskip

Let us now outline this proof in a little more detail.
\begin{itemize}
    \item
	In \cref{subsec:quasitrans} we reduce \cref{thm:fin.perc} to the transitive case by noting every $n$-quasitransitive graph is rough-isometric to a transitive graph of comparable volume and diameter.
	\item 
  In \cref{subsec:boxes,subsec:Abelian}, we prove \cref{thm:fin.perc.ab} by using techniques from additive combinatorics to reduce from arbitrary Abelian Cayley graphs to 
   boxes in $\Z^d$ with the standard generating set as discussed above. We stress that this reduction from arbitrary Abelian Cayley graphs to hypercubic boxes, summarised by \cref{prop:boxes}, is one of the main technical innovations of the paper. Once this reduction has been carried out, \cref{thm:fin.perc.ab} can be handled by minor variations on the methods of Malon and Pak \cite{MR2238046}. 

   \item In \cref{subsec:NilpotentNilpotent}, we prove a version of \cref{thm:fin.perc} for Cayley graphs of nilpotent groups (\cref{thm:fin.perc.nilp}), but where the exponent $a$ is taken to be the step of the group (i.e., the length of the lower central series). This is done by induction on the step, with the base case being handled by \cref{thm:fin.perc.ab}, and is the first place in which we lose additional powers of log in our analysis. 
\item
   In \cref{sec:Isoperimetry}, which can be read independently of the rest of the paper, we adapt the methods of \cite{1806.07733} to analyze percolation on finite graphs satisfying a $12$-dimensional isoperimetric condition. As described in more detail just after the statement of \cref{thm:iso.inf} below, the proof of that paper adapts straightforwardly to show under a $(4+\eps)$-dimensional isoperimetric assumption that there is a non-trivial phase in which there exist \emph{large} clusters (i.e. clusters of size going to infinity with the volume of the graph), but an additional and completely new argument is needed to deduce the existence of a \emph{giant} cluster (i.e. a cluster of volume \emph{proportional} to the volume of the graph).
   \item
   In \cref{sec:structuretheory} we review the structure theory of vertex-transitive graphs as developed in \cite{tt.resist,tt.Trof}. We also prove an important supporting technical proposition, stating roughly that at the scale where the graph crosses over from being at least $(d+1)$-dimensional to at most $d$-dimensional, we can find a set that well approximates the ball and that induces a subgraph satisfying a $(d+1)$-dimensional isoperimetric inequality.
   \item Finally, in \cref{sec:mainproofs} we put all these ingredients together to deduce \cref{thm:fin.perc}. Note that a substantial argument is still required to complete this stage of the proof, which is the second and last place that additional powers of log are lost in our analysis.
\end{itemize}

\subsection{Further results}

We now state our other main results.

\medskip

\noindent\textbf{Isoperimetric criteria for percolation.} We now discuss our results concerning percolation on finite graphs under isoperimetric conditions, which build on the work of \cite{1806.07733} and play an important role in the proof of \cref{thm:fin.perc} as described above.
Let $d\geq 1$ and $c>0$. A locally finite graph $G=(V,E)$ is said to satisfy a \textbf{$d$-dimensional isoperimetric inequality} with constant $c$, abbreviated \eqref{Assumption:ID}, if
\begin{equation}
\label{Assumption:ID}
\tag{$\mathrm{ID}_{d,c}$} |\partial_E K| \geq c \min\bigl\{|K|,|V \setminus K|\bigr\}^{(d-1)/d}
\end{equation}
for every finite set of vertices $K \subseteq V$. Here, $\partial_E K$ denotes the edge boundary of $K$, i.e., the set of edges with endpoints in both $K$ and $V \setminus K$. It is a classical result of Coulhon and Saloff-Coste~\cite{MR1232845} that transitive graphs of at least $d$-dimensional volume growth always satisfy $d$-dimensional isoperimetric inequalities, and strong quantitative versions of this result for finite transitive graphs have recently been proven by Tessera and the second author in \cite{tt.resist}. 

The aforementioned theorem of Duminil-Copin, Goswami, Raoufi, Severo, and Yadin \cite{1806.07733}, which we review in detail in \cref{sec:Isoperimetry}, may be phrased quantitatively as follows. (While they did not phrase their results in this way, one may easily verify that all the constants appearing in their proof can be taken to depend only on the parameters $d$, $c$, and $k$; see \cref{rem:DGRSY} below.)

\begin{theorem}[DGRSY]
\label{thm:quantitative_DGRSY}
For each $k<\infty$, $d >4$, and $c>0$, there exists $\eps=\eps(d,c,k)>0$ such that if $G=(V,E)$ is an infinite, connected graph with degrees bounded by $k$ satisfying the $d$-dimensional isoperimetric inequality \eqref{Assumption:ID} then
$p_c(G)\leq 1-\eps$. Moreover, there exists a constant $\eta=\eta(d,c,k)>0$ such that
\[
\P_p\bigl(A \leftrightarrow \infty\bigr) \geq 1-  \exp\left[ -\eta |A|^{(d-2)/d} \right]
\]
for every $p \geq 1-\eps$ and every finite non-empty set $A \subseteq V$. 
\end{theorem}

Our second main theorem extends this result to finite graphs under a stronger assumption on the dimension. Note that neither \cref{thm:quantitative_DGRSY} nor \cref{thm:iso.inf} require transitivity. Given two sets of vertices $A$ and $B$, we write $\{A \leftrightarrow B\}$ for the event that there is an open path connecting $A$ to $B$. The dimensional threshold $6+2\sqrt{7}$ appearing here satisfies $6+2\sqrt{7} \approx 11.29 < 12$.

\begin{theorem}
\label{thm:iso.inf}
Let $G=(V,E)$ be a finite, connected graph with degrees bounded by $k$ that satisfies a $d$-dimensional isoperimetric inequality \eqref{Assumption:ID} for some $d > 6+2\sqrt{7}$ and $c>0$. There exists a positive constant $\eta=\eta(d,c,k)$ such that for every $\eps>0$ there exists $0<p_0=p_0(\eps,d,c,k)<1$
 such that $\P_p(|K_v| \geq (1-\eps)|V|) \geq 1-\eps$ for every $v\in V$ and $p\geq p_0$ and
\[
\P_p\bigl(A \leftrightarrow B\bigr) \geq 1- \eps \exp\left[ -\eta \min\{|A|,|B|\}^{(d-2)/d} \right]
\]
for every $p \geq p_0$ and every two non-empty sets $A,B \subseteq V$.
\end{theorem}

As we shall see in \cref{sec:Isoperimetry}, the proof of \cite{1806.07733} extends straightforwardly to show that if $G$ satisfies a $d$-dimensional isoperimetric assumption for some $d>4$ then there exists $\eps=\eps(k,d)>0$ such that if $p\geq 1-\eps$ then each vertex $v$ has a good probability to be connected to any set $A \subseteq V$ satisfying $|A| \geq |V|^{1-\delta}$, where $\delta=\delta(d)=(d-4)/4(d-1)$ is an explicit positive constant tending to $0$ as $d \downarrow 4$. An additional argument is required to deduce that a giant component exists, and we have been able to implement such an argument only under a stronger assumption on $d$.

\begin{remark}
We have not optimized the value of the dimensional threshold $6+2\sqrt{7}\approx 11.29$ appearing here. We have been able to extend the result to some lower values of the dimension, but not all the way down to $4+\eps$, and do not pursue these improvements here. We expect that \cref{thm:quantitative_DGRSY,thm:iso.inf} should hold for any $d>1$, but this appears to be beyond the scope of existing methods.
\end{remark}

\medskip

\noindent \textbf{Critical probability gap for infinite transitive graphs.}
Although the main focus of this paper is percolation on finite graphs, a number of the techniques apply equally well to infinite graphs. In particular, this allows us to make the results of \cite{1806.07733} more quantitative in the following sense. Recall that a transitive graph is said to have \textbf{superlinear volume growth} if $\limsup_{n\to\infty} \frac{1}{n} |B(o,n)| = \infty$, where $B(o,n)$ denotes the ball of radius $n$ around the vertex $o$.
\begin{theorem}[Critical probability gap]\label{thm:gap}
Let $k,n\in\N$. Then there exists $\eps=\eps(k,n)>0$ such that if $G$ is an infinite, connected $n$-quasitransitive graph of degree at most $k$ with superlinear volume growth then $p_c(G)\le1-\eps$. In particular, every infinite, connected, $n$-quasitransitive graph $G$ of degree at most $k$ has either $p_c(G)\le1-\eps$ or $p_c(G)=1$.
\end{theorem}

Since the main results of \cite{1806.07733} are already proven quantitatively as discussed above, the main novelty of the proof of \cref{thm:gap} comes from the application of quantitative forms of Gromov's theorem as developed in \cite{MR3877012,tt.resist,tt.Trof,tt.Lie} to handle the low-dimensional case in a quantitative way.

A critical probability gap of the form established by \cref{thm:gap} was first suggested to hold by G\'abor Pete \cite[p.\ 225]{Pete}, who noted that it would follow from Schramm's locality conjecture \cite{MR2773031} together with the (then conjectural) results of \cite{1806.07733}. Moreover, \cref{thm:gap} shows in particular that, in the formulation of the locality conjecture, one may harmlessly 
 replace the assumption  that $p_c(G_n)<1$ for all sufficiently large $n$ with the \emph{a priori} stronger assumption that $\limsup_{n\to\infty}p_c(G_n)<1$. See \cite{MR2773031,1808.08940} for overviews of this conjecture and the progress that has been made on it. [Added in proof: Schramm's locality conjecture has now been solved in a preprint of Easo and the first author \cite{easo2023critical}.]

\medskip

\noindent\textbf{Corollaries for the uniqueness threshold.}
 Recall that if $G=(V,E)$ is an infinite, connected, locally finite graph then the \textbf{uniqueness threshold} $p_u=p_u(G)$ for Bernoulli bond percolation on $G$ is defined by
\[
p_u = \inf\left\{p\in [0,1] : \text{there is a unique infinite cluster $\P_p$-almost surely}\right\}.
\]
It is a result originally due to H\"aggstrom, Peres, and Schonmann \cite{HaggPe99,MR1676831,HPS99} that if $G$ is quasitransitive then there is a unique infinite cluster almost surely for every $p>p_u$.
Benjamini and Schramm \cite[Question 3]{bperc96} asked whether the strict inequality $p_u<1$ holds for every transitive graph with one end. Of course, when the graph in question is amenable we have that $p_c=p_u$ by the classical results of Aizenman, Kesten, and Newman~\cite{MR901151} and Burton and Keane~\cite{burton1989density}, so that the question has a positive answer in this case by the results of \cite{1806.07733}.  In the nonamenable setting, the question has been resolved positively for Cayley graphs of finitely presented groups by Babson and Benjamini \cite{MR1622785} (see also \cite{MR2286517}), for graphs defined as direct products by Peres \cite{MR1770624}, and for Cayley graphs of Kazhdan groups and wreath products by Lyons and Schramm \cite{LS99} but remains open in general.

\cref{thm:iso.inf} leads to an interesting isoperimetric criterion for an infinite graph to have $p_u<1$. We define the \textbf{internal isoperimetric dimension} of an infinite, connected, locally finite graph $G=(V,E)$ to be the supremal value of $d$ for which there exists a positive constant $c$ and an exhaustion $V_1 \subseteq V_2 \subseteq \cdots$ of $V$ by finite connected sets such that the subgraph $G_n$ of $G$ induced by $V_n$ satisfies the $d$-dimensional isoperimetric inequality \eqref{Assumption:ID} for every $n\geq 1$. Notions closely related to the internal isoperimetric dimension have been studied systematically in the recent work of Hume, Mackay, and Tessera \cite{hume2017poincar}, whose methods implicitly lead to computations of the internal isoperimetric dimension in various examples: For example, one can prove via their methods that $\Z^d$ has internal isoperimetric dimension $d$ for every $d\geq 1$,  the $3$-regular tree has internal isoperimetric dimension $1$, and graphs rough-isometric to $d$-dimensional hyperbolic space $\mathbb{H}^d$ with $d \geq 2$ have internal isoperimetric dimension $d-1$. 

We are now ready to state our results on the uniqueness threshold.

\begin{thm}
\label{thm:pu}
Let $k\in\N$, $c>0$ and $d>6+2\sqrt{7}$, and suppose that $G=(V,E)$ is an infinite, connected graph with degree at most $k$ for which there exists an exhaustion $V_1 \subseteq V_2 \subseteq \cdots$ of $V$ by finite connected sets such that each subgraph $G_n$ of $G$ induced by $V_n$ satisfies the $d$-dimensional isoperimetric inequality \eqref{Assumption:ID}. Then there exists $\eps=\eps(d,c,k)>0$ such that Bernoulli-$p$ bond percolation on $G$ has a unique infinite cluster almost surely for every $p\ge1-\eps$. 
\end{thm}

\begin{corollary}
Every infinite, connected, bounded degree graph with internal isoperimetric dimension strictly greater than $6+2\sqrt{7}$ has $p_u<1$.
\end{corollary}

As above, the dimensional threshold appearing in \cref{thm:pu} has not been optimized, and we conjecture that the same conclusion holds for every infinite, connected, bounded degree graph with internal isoperimetric dimension strictly greater than $1$. (This should only be a sufficient condition for $p_c<1$, since graphs rough-isometric to the hyperbolic plane have $p_u<1$ but internal isoperimetric dimension $1$. One may be able to formulate a sharp condition for $p_u<1$ in terms of \emph{logarithmic} internal isoperimetric inequalities.) We remark that our notion of internal isoperimetric dimension is also closely related to the isoperimetric criteria for $p_u<1$ for graphs of polynomial growth developed in the work of Teixeira \cite{MR3520023}.

When $G$ is transitive, \cref{thm:pu} follows immediately from \cref{thm:iso.inf} applied to the sequence of finite graphs $G_n$ together with a result of Schonmann \cite{MR1676831} stating that if $G$ is transitive then
\[
\Bigl(\text{there is almost-sure uniqueness at $p$}\Bigr) \iff \Bigl(\lim_{n\to\infty} \inf_{x,y \in V} \P_p\left(B(x,n) \leftrightarrow B(y,n)\right) =1\Bigr),
\]
where $B(x,n)$ denotes the graph distance ball of radius $n$ around $x$.
See also \cite{LS99,tang2018heavy} for stronger versions of this theorem. In order to deduce \cref{thm:pu} from \cref{thm:iso.inf} without the assumption that $G$ is transitive, we prove the following variation on Schonmann's theorem in \cref{subsec:pu}.

\begin{theorem}
\label{thm:uniquenesscriterion}
Let $G=(V,E)$ be an infinite, connected, bounded degree graph. If $0 < p_0 < 1$ is such that 
\[
\lim_{n\to \infty} \max\Bigl\{ \P_{p_0}(A \nleftrightarrow B) : A,B \subseteq V \text{ with } |A|,|B|\geq n\Bigr\}  =0
\]
then Bernoulli-$p$ bond percolation on $G$ has a unique infinite cluster almost surely for each $p>p_0$.
\end{theorem}

\subsection{Further discussion and remarks}

\begin{remark}[Other models]
\label{rem:other_models}
It follows by standard stochastic domination arguments that each of our main theorems implies analogous results for several other percolation-type models, including site percolation \cite{grimmett1998critical}, finitely dependent models \cite{MR1428500}, and the Fortuin-Kastelyn random cluster model~\cite{GrimFKbook}. 
 Using the Edwards-Sokal coupling \cite{edwards1988generalization}, it follows moreover that the ferromagnetic Ising and Potts models have uniformly non-trivial ordered phases on the classes of graphs treated by these theorems (i.e., that the pair correlations of these models are uniformly bounded away from zero at sufficiently low temperatures).
\end{remark}

\begin{remark}[Sharp thresholds]
It is a standard consequence of the abstract sharp-threshold theorem of Kahn, Kalai, and Linial \cite{kahn1989influence} that if $G_n$ is a sequence of vertex-transitive graphs with volume tending to infinity then $|p_c(G_n,\alpha,1-\eps)-p_c(G_n,\alpha,\eps)|\to 0$ as $n\to \infty$ for each fixed $0<\alpha,\eps<1$. See e.g.\ \cite[Section 4.7]{grimmett2018probability} for background on the use of such sharp-threshold theorems in percolation. This allows us to immediately deduce the statement given in the abstract from that of \cref{thm:fin.perc}: There exists a universal constant $a$ such that if $(G_n)_{n
\geq1}=((V_n,E_n))_{n\geq 1}$ is a sequence of finite, vertex-transitive graphs with bounded degrees and $|V_n|\to\infty$ such that
\[\diam (G_n)=O\left(\frac{|V_n|}{(\log |V_n|)^a}\right) \qquad \text{ as $n\to\infty$}\]
 then for each $0<\alpha <1$ there exists $p_c(\alpha)<1$ such that for each $p\geq p_c(\alpha)$, Bernoulli-$p$ bond percolation on $G_n$ has a cluster of size at least $\alpha |V_n|$ with probability tending to $1$ as $n\to \infty$. 
\end{remark}

\begin{remark}[Large clusters vs.\ giant clusters]
\label{rem:pc_def}
We now discuss a key issue underlying many of the additional difficulties arising in the finite volume that do not arise in infinite volume.
 Let $(G_n)_{n\geq 1}=((V_n,E_n))_{n\geq 1}$ be a sequence of finite, vertex-transitive graphs of bounded degree and with $|V_n|\to\infty$. Our main theorems give criteria under which the critical probability
\begin{equation}
\label{eq:p_c_sequence_definition}
p_c=p_c((G_n)_{n\geq 1}):= \lim_{\alpha \downarrow 0} \lim_{q \downarrow 0} \limsup_{n\to\infty} p_c(G_n,\alpha,q)
\end{equation}
is strictly less than $1$. In the case that $(G_n)_{n\geq 1}$ converges locally to some infinite, connected, locally finite, vertex-transitive graph $G$, it is natural to wonder whether the two critical probabilities $p_c(G)$ and $p_c((G_n)_{n\geq 1})$ are necessarily equal. Indeed, if this were the case, one would be able to deduce our main results (and more!) rather easily from the results of \cite{1806.07733,tt.Trof} by a compactness argument. Alon, Benjamini, and Stacey \cite{MR2073175} proved that the equality $p_c(G)=p_c((G_n)_{n\geq 1})$ holds when $(G_n)_{n\geq 1}$ is an expander sequence. This equality is not true in general, however; indeed, the elongated torus $G_{n,k}=(\Z/n\Z)\times (\Z / k^n \Z)$ (with its standard generating set) Benjamini-Schramm converges to the square lattice $\Z^2$, which has $p_c(\Z^2)=1/2$, but has $p_c((G_{n,k})_{n\geq 1}) \to 1$ as $k\to\infty$. 

On the other hand, there are several alternative notions of critical probability for the sequence $(G_n)_{n\geq 1}$ that do always coincide with $p_c(G)$ when $G_n$ converges locally to $G$. For example, if we let $o_n$ be a vertex of $G_n$ for each $n\geq 1$, let $K_{o_n}$ be the cluster of $o_n$, and define 
\[
p_T=p_T((G_n)_{n\geq 1}) := \sup\left\{p \in [0,1] : \limsup_{n\to\infty} \E_p^{G_n} |K_{o_n}| < \infty \right\}
\]
then it follows from the sharpness of the phase transition \cite{duminil2015new,aizenman1987sharpness} that $p_T((G_n)_{n\geq 0})=p_c(G)$ whenever $G_n$ is a sequence of transitive graphs converging locally to a transitive, locally finite graph $G$.  It follows by similar sharpness arguments that the critical probability $p_T$ can also be characterised as the supremal value of $p$ for which the cluster volumes $|K_{o_n}|$ are tight as $n\to\infty$. 
Together with the structure theoretic results of \cite{tt.Trof}, one can deduce rather easily from these considerations and the results of \cite{1806.07733} that if $(G_n)_{n\geq 1}=((V_n,E_n))_{n\geq 1}$ is any sequence of finite, vertex-transitive, bounded degree graphs with $\diam(G_n) = o(|V_n|)$ as $n\to\infty$ then there exists $p < 1$ and $f:\N\to\N$ with $\lim_{n\to\infty} f(n)=\infty$ such that
\[
\P_p^{G_n}\!\left(|K_{v}| \geq f(n) \right) \geq \frac{1}{2}
\]
for every $n\geq 1$ and $v\in V_n$. In other words, \emph{a sublinear diameter suffices for there to be a non-trivial phase in which the cluster of the origin is unboundedly large with good probability.} However, as we see in the example of the elongated torus, it is possible in finite graphs to have a non-trivial phase in which there are many large clusters but no giant cluster. As mentioned above, this issue underlies many of the additional technical difficulties that arise for finite transitive graphs but not for infinite transitive graphs.
\end{remark}

\begin{remark}
In the recent works of Easo and the first author \cite{easo2021supercritical2,easo2021supercritical}, which were written after this paper first appeared, it is proven that the giant cluster is unique and has concentrated volume in supercritical percolation on any finite vertex-transitive graph. See also Easo's work \cite{easo2023existence} on the problem of when there is a well-defined \emph{critical probability} for large finite transitive graphs.
\end{remark}

\subsection{Notation}

Throughout the paper we assume without loss of generality that graphs do not have loops or multiple edges. This can indeed be done without loss of generality since adding loops has no effect on percolation, while adding multiple edges only makes it easier for a giant cluster to exist.

\medskip

Given a subset $A$ of a group $\Gamma$, we write $A^{-1}=\{a^{-1}:a\in A\}$ and $\hat A=A\cup\{\mathrm{id}\}\cup A^{-1}$. Given in addition $m\in\N$, we define $A^m = \{a_1 \cdots a_m : a_1,\ldots,a_m \in A\}$, write $\hat A^m=(\hat A)^m$, and define $\hat A^0=\{\mathrm{id}\}$. Given two sets $A$ and $B$ we also write $AB = \{ab : a\in A, b\in B\}$. We write $\langle A \rangle = \bigcup_{m\geq 0} \hat A^m$ for the subgroup of $\Gamma$ generated by $A$, and say that $A$ \textbf{generates} $\Gamma$ if $\langle A \rangle = \Gamma$. Given a group $\Gamma$ and a finite generating set $S$ of $\Gamma$, the (undirected) \textbf{Cayley graph} $\Cay(\Gamma,S)$ is the graph with vertex set $\Gamma$ and edge set $\{\{x,y\}\in \Gamma \times \Gamma :x = ys $ for some $s\in  S \cup S^{-1}\setminus\{\mathrm{id}\}\}$. Cayley graphs are always transitive since left multiplication defines a transitive action of $\Gamma$ on $\Cay(\Gamma,S)$ by graph automorphisms. We write $\diam_S(\Gamma)$ for the diameter of $\Cay(\Gamma,S)$, which is equal to the infimal $m$ such that $\Gamma = \hat S^m$.
 When $\Gamma$ is Abelian we will often denote the same objects using additive notation, so that e.g.\ $-A = \{-a : a\in A\}$, $\hat A = A\cup\{0\}\cup(-A)$, and  $m A= \{a_1 + \cdots + a_m : a_1,\ldots,a_m \in A\}$.

\medskip

As above, we write $\P_p=\P_p^G$ for the law of Bernoulli-$p$ bond percolation on a graph $G$, including the superscript only when the choice of graph is ambiguous.
We write $\{x \leftrightarrow y\}$ for the event that $x$ and $y$ belong to the same open cluster, and write $\{A \leftrightarrow B\}$ for the event that there is an open cluster intersecting both $A$ and $B$. We write $K_x$ for the cluster containing the vertex $x$. 

\medskip

Given a graph $G=(V,E)$ and a subgroup $H < \Aut(G)$, we write $Hv = \{hv:h\in H\}$ for the orbit of a vertex $v$ in $H$ and write $H_v = \{h\in H: hv=v\}$ for the stabilizer of $v$ in $H$. We also use similar notation for orbits and stabilizers of edges. We will often take $o$ to be a fixed root vertex of a vertex-transitive graph and write $B(o,n)$ for the graph-distance ball of radius $n$ around $o$.

\section{Background on percolation and vertex-transitive graphs}

In this section we present some basic tools for use in the rest of the paper.

\medskip

\noindent \textbf{The Harris-FKG inequality.} Let $G=(V,E)$ be a countable graph. A set $A \subseteq \{0,1\}^E$ is said to be \emph{increasing} if whenever $\omega\subset\omega'$ and $\omega\in A$ we have $\omega'\in A$. The Harris-FKG inequality \cite[Chapter 2.2]{grimmett2010percolation} states that increasing events are positively correlated under product measures, so that $\P_p(A \cap B) \geq \P_p(A)\P_p(B)$ for every two increasing measurable sets $A$ and $B$.
Since we will usually be free to increase $p$ whenever needed, we may also use the following trick to handle some situations in which Harris-FKG does not apply: If $G$ is a countable graph and $\omega$ and $\omega'$ are independent Bernoulli percolation configurations with retention probabilities $1-\bar p_1$ and $1-\bar p_2$ respectively, then $\omega \cup \omega'$ is distributed as Bernoulli percolation with retention probability $1-\bar p_1 \bar p_2$.

\medskip

\noindent \textbf{Criteria for a giant component.} 
The next lemma gives necessary and sufficient conditions for a giant component in terms of point-to-point connection probabilities. 

\begin{lemma}\label{lem:cluster/connect}
Let $G=(V,E)$ be a finite graph, let $p\in [0,1]$, and let $\alpha \in (0,1)$.
\begin{enumerate}
\item If $u\in V$ and $\beta>\alpha$ are such that $\E_p |K_u| \geq \beta |V|$, then $\P_p(|K_u| \geq \alpha |V|) \geq (\beta-\alpha)/(1-\alpha)$.
\item
 If $G$ is vertex-transitive and $0< \alpha,q \leq 1$ are such that $\P_p(|K_u|\ge\alpha|V|)\ge q$ for every $u\in V$ then $\E_p|K_u|\ge\alpha q|V|$ and
 \[\P_p(u\leftrightarrow v)\geq \eta^{3/\eta} \qquad \text{ where } 
 \qquad\eta=\frac{\alpha q}{2} \wedge p\] for every $u,v\in V$.
\end{enumerate} 
\end{lemma}

\begin{remark}
Note that, for any given $u\in V$, an assumption that $\P_p(u\leftrightarrow v)\ge\beta$ for every $v\in V$ is enough to be able to apply the first part of \cref{lem:cluster/connect}, since $\E_p|K_u|\ge\beta|V|$ in that case.
\end{remark}

\begin{remark}
The second part of \cref{lem:cluster/connect} is not true without the transitivity assumption: consider a path of length $n$ connected at one end to a complete graph with $n$ vertices. 
\end{remark}

\begin{remark}\label{rem:increase.alpha.q}
It is an important and useful fact \cite[Theorem 2.38]{grimmett2010percolation} that $\P_{p^\theta}(A) \geq \P_p(A)^\theta$ for every increasing event $A$ and $p,\theta \in (0,1)$. Using this together with \cref{lem:cluster/connect}, one may easily deduce that for each $0<\alpha_1 \leq \alpha_2<1$ and $0<q_1 \leq q_2 <1$ there exists $\theta=\theta(\alpha_1,\alpha_2,q_1,q_2) \in (0,1)$ such that $p_c(G,\alpha_2,q_2) \leq p_c(G,\alpha_1,q_1)^\theta$ for every finite vertex-transitive graph $G$. As such, in order to prove all our main theorems it would suffice to prove them for a specific value of $\alpha$ and $q$.
\end{remark}

The first item of the lemma follows trivially by applying Markov's inequality to $|V|-|K_u|$. The second item appears in a slightly less general form in the work of Benjamini \cite[Proposition 1.3]{benjamini2001percolation}, who attributed the argument to Schramm. 
For completeness, we now give a quick proof of this lemma at the full level of generality that we require.

\begin{lemma}\label{lem:half.doubles.to.whole}
 Let $G=(V,E)$ be a finite vertex-transitive graph, let $o\in V$, and let $p \in [0,1]$. If $0<\alpha \leq 1$ and $m\in \N$ are such that $|\{v\in V:\P_p(o\leftrightarrow v)\ge\alpha\}|>|V|/(m+1)$ then $\P_p(o\leftrightarrow v)\ge (\alpha \wedge p)^{3m}$ for every $v\in V$. 
\end{lemma}

\begin{proof}[Proof of \cref{lem:half.doubles.to.whole}]
Write $\Gamma=\Aut(G)$. We will apply the following algebraic lemma of \cite[Lemma 2.1]{MR3403962} (see also \cite[Proposition 1.3]{benjamini2001percolation}).
	\begin{lemma}
	\label{lem:large_sets_double}
Let $\Gamma$ be a finite group and let $A \subseteq \Gamma$ be symmetric and contain the identity. If $m\geq 1$ is such that $|A| > |\Gamma|/(m+1)$ then $A^{3m}=\langle A \rangle$.
	\end{lemma}
For each $v\in V$, fix some $\gamma_v \in \Gamma$ such that $\gamma_v o=v$. Set $A=\{\gamma \in \Gamma :\P_p(o \leftrightarrow \gamma o) \ge \alpha \wedge p\}$, noting that $A$ is symmetric and contains the identity. Since $\{\gamma \in \Gamma: \gamma o=v\}=\gamma_v \Gamma_o$ for every $v\in V$, we have $|A|>|\Gamma|/(m+1)$. Moreover,  we have trivially that $\{\gamma \in \Gamma : d(o,\gamma o) \leq 1 \} \subseteq A$, and since this set generates $\Gamma$ we deduce that $\langle A \rangle = \Gamma$ also. (This is why we considered $\alpha \wedge p$ instead of $\alpha$.)
Thus, we may apply \cref{lem:large_sets_double} to deduce that $A^{3m}=\Gamma$. Setting $k=3m$, for each $v\in V$ there therefore exist $a_1,\ldots,a_k\in A$ such that $\gamma_v=a_k\cdots a_1$ and hence that $v=a_k\cdots a_1 o$. Writing $v_0=o$ and $v_i = a_i \cdots a_1 o$ for each $1 \leq i \leq k$ we deduce by Harris-FKG and the definition of $A$ that
\begin{align*}
\P_p(o\leftrightarrow v)
   \ge\prod_{\ell=1}^k\P_p(v_{i-1}\leftrightarrow v_{i}) 
   =\prod_{\ell=1}^k\P_p(v_{i-1} \leftrightarrow a_i v_{i-1}) \ge (\alpha \wedge p)^k 
\end{align*}
as required.
\end{proof}

\begin{proof}[Proof of Lemma \ref{lem:cluster/connect}]
Markov's inequality implies that if $\E |K_u| \geq \beta |V|$ then
\begin{align*}
\P\Bigl(|V\setminus K_u|>(1-\alpha)|V|\Bigr)&\le\frac{\E|V\setminus K_u|}{(1-\alpha)|V|} \le\frac{1-\beta}{1-\alpha}
\end{align*}
which is equivalent to the desired conclusion.
Conversely,  
if $\P_p(|K_u|\ge\alpha|V|)\ge q$ then $\E_p[\,|K_u|\,]\ge\alpha q|V|$. This implies that there are at least $\alpha q|V|/2$ vertices $v$ satisfying $\P_p(u\leftrightarrow v)\ge\alpha q/2$, so that if $G$ is vertex transitive then the second desired conclusion follows from \cref{lem:half.doubles.to.whole}.
\end{proof}

The following variation on \cref{lem:half.doubles.to.whole} will also be useful.   Recall that when $H$ is a subgroup of $\Aut(G)$ and $v$ is a vertex of $G$, write $Hv$ for the orbit $Hv=\{h v :h \in H\}$ and $H_v$ for the stabiliser $H_v=\{h\in H:h(v)=v\}$.

\begin{lemma}\label{lem:half.doubles.to.whole2}
 Let $G=(V,E)$ be a finite vertex-transitive graph, let $o\in V$, let $p \in [0,1]$, and let $H<\Aut(G)$. If $\alpha>0$ is such that $|\{v\in Ho:\P_p(o\leftrightarrow v)\ge\alpha\}|>|Ho|/2$
then $\P_p(u\leftrightarrow v)\ge\alpha^2$ for every $v\in Ho$.
\end{lemma}

\begin{proof}
Let $A = \{h\in H : \P_p(o\leftrightarrow ho) \geq \alpha\}$, which is symmetric and contains the identity. Then we have by transitivity that $|A| = |\{v\in Ho:\P_p(o\leftrightarrow v)\ge\alpha\}| \cdot |H_o|$ and $|H| = |Ho| \cdot |H_o|$, so that $|A| > |H|/2$. It follows that $A^2 =H$, since for each $h\in H$ we have that $h A \cap A \neq \varnothing$ and hence that there exist $a_1,a_2 \in A$ such that $h a_1 = a_2$. The claim now follows similarly to the proof of \cref{lem:half.doubles.to.whole}.
\end{proof}

\subsection{Quotients and rough isometries}

We now discuss several useful ways in which percolation on two different graphs can be compared. We will be particularly interested in the cases that either one graph is a quotient of the other or the two graphs are rough-isometric.

\medskip

\noindent \textbf{Monotonicity under quotients.} We first recall a coupling of percolation on a graph and a quotient of that graph due to Benjamini and Schramm \cite[Theorem 1]{bperc96} (see also \cite[Section 2]{MR2276449}), which implies in particular that if a graph $G$ admits a quotient in which $p_c<1$ then $G$ also has $p_c<1$. See also \cite{MR4038050} for strengthened forms of this result.

\begin{proposition}[Benjamini--Schramm]\label{prop:Benj-Schr}
 Let $G=(V,E)$ be a locally finite graph, let $H<\Aut(G)$, and let $\pi:G\to G/H$ be the quotient map. For each $v\in G$ and $p\in [0,1]$, the cluster of $\pi(v)$ in Bernoulli-$p$ percolation on $G / H$ is stochastically dominated by the image under $\pi$ of the cluster of $v$ in Bernoulli-$p$ percolation on $G$. That is,
\[
\P_p^G\Bigl(\pi(K_v) \in A\Bigr)\ge\P_p^{G/H}\Bigl( K_{\pi(v)} \in A\Bigr)
\]
for every increasing measurable set $A \subseteq \{0,1\}^{V/H}$.
\end{proposition}

See e.g.\ \cite[Chapter 4.1]{grimmett2018probability} for background on stochastic domination.

The next lemma is a straightforward observation allowing us to make a comparison in the opposite direction when $H$ has bounded edge-orbits, at the cost of increasing $p$ in a way that depends on the size of these orbits.  It implies in particular that if $H$ has bounded orbits and $G$ has $p_c<1$ then $G/H$ has $p_c<1$ also.

\begin{lemma}\label{lem:bdd.orbits}
 Suppose that $G=(V,E)$ is a locally finite graph and that $H<\Aut(G)$ satisfies $|He|\le k$ for some $k \geq 1$ and every $e\in E$. 
For each $v\in G$ and $p\in [0,1]$, the cluster of $\pi(v)$ in Bernoulli-$(1-(1-p)^{k})$ percolation on $G / H$ stochastically dominates the image under $\pi$ of the cluster of $v$ in Bernoulli-$p$ percolation on $G$. That is,
\[
\P_{1-(1-p)^{k}}^{G/H}\left(K_{\pi(v)} \in A\right) \ge \P_p^G\left(\pi(K_v) \in A\right)
\]
for every increasing measurable set $A \subseteq \{0,1\}^{V/H}$.
\end{lemma}

\begin{proof}
Let $\omega$ be Bernoulli-$p$ bond percolation on $G$, and let $\eta \in \{0,1\}^{E/H}$ be defined by taking $\eta(e)=1$ if and only if there is at least one $\omega$-open edge in the preimage $\pi^{-1}(e)$. For each vertex $v$ of $G$, we clearly have that the cluster of $\pi(v)$ in $\eta$ contains the image of the cluster of $v$ in $\omega$.
 Moreover, the edges of $G/H$ are open or closed in $\eta$ independently of one another and each is open with probability at most $1-(1-p)^k$. It follows by a standard argument \cite[Chapter 5.2]{LP:book} that $\eta$ is stochastically dominated by Bernoulli-$(1-(1-p)^k)$ percolation on $G / H$ and the result follows.
\end{proof}

\noindent \textbf{Rough isometries and rough embeddings}. 
We now make note of a simple folkloric lemma allowing us to compare percolation on two \emph{rough-isometric} graphs, or more generally on a graph that can be \emph{roughly embedded} into another graph. We first recall the relevant definitions, referring the reader to e.g.\ \cite[Chapter 2.6]{LP:book} for further background.
Let $G_1=(V_1,E_1)$ and $G_2=(V_2,E_2)$ be connected, locally finite graphs and let $\alpha \geq 1$ and $\beta \geq 0$. We abuse notation and write $d( \,\cdot\,,\,\cdot\,)$ for the graph distance on both $G_1$ and $G_2$. A function $\phi:V_1 \to V_2$ is said to be an $(\alpha,\beta)$-\textbf{rough isometry} if the following conditions hold:
\begin{enumerate}
	\item ($\phi$ roughly preserves distances.) The inequality $\alpha^{-1} d(x,y)-\beta \leq d(\phi(x),\phi(y)) \leq \alpha d(x,y)+\beta$ holds for every $x,y \in V_1$. 
	\item ($\phi$ is roughly surjective.) For every $y\in V_2$ there exists $x\in V_1$ with $d(\phi(x),y)\leq \beta$.
\end{enumerate}
Note that the relation of rough isometry is approximately symmetric in the sense that for each $\alpha \geq 1$ and $\beta\geq 0$ there exist $\alpha' =\alpha'(\alpha,\beta) \geq 1$ and $\beta'=\beta'(\alpha,\beta)\geq 0$ such that whenever there exists an $(\alpha,\beta)$-rough isometry $\phi:V_1\to V_2$ between two graphs $G_1=(V_1,E_1)$ and $G_2=(V_2,E_2)$ then there also exists an $(\alpha',\beta')$-rough isometry $\psi : V_2 \to V_1$. (Indeed, simply choose $\psi(y)$ to be an arbitrary element of the set $\{x: d(\phi(x),y)\leq \beta\}$ for each $y\in V_2$.) 
More generally, we say that $\phi$ is an $(\alpha,\beta)$-\textbf{rough embedding} if the following conditions hold:
\begin{enumerate}
	\item ($\phi$ is roughly Lipschitz.) The inequality $d(\phi(x),\phi(y)) \leq \alpha d(x,y)+\beta$ holds for every $x,y \in V_1$. 
	\item ($\phi$ is bounded-to-one.) $|\phi^{-1}(x)| \leq \beta$ for every $x \in V_2$.
\end{enumerate}
Note that every $(\alpha,\beta)$-rough isometry between graphs with degrees bounded by $k$ is an $(\alpha,\beta')$-rough embedding for some $\beta'=\beta'(\alpha,\beta,k)$.

\begin{lemma}
\label{lem:roughembedding}
For each $k, \alpha \geq 1$ and $\beta \geq 0$ there exists a constant $C=C(k,\alpha,\beta)$ such that the following holds. Let $f:[0,1]\to [0,1]$ be the increasing homeomorphism $f(p)= 1-(1-p^{1/C})^C$, let $G_1=(V_1,E_1)$ and $G_2=(V_2,E_2)$ be connected graphs with degrees bounded by $k$, and let $\phi:V_1 \to V_2$ be an $(\alpha,\beta)$-rough embedding. Then for each $v \in V_1$ and $p\in [0,1]$, the cluster of $\phi(v)$ in Bernoulli-$f(p)$ percolation on $G_2$ stochastically dominates the image under $\phi$ of the cluster of $v$ in Bernoulli-$p$ percolation on $G_1$.  
That is,
\[
\P_{f(p)}^{G_2}\left(K_{\phi(v)} \in A\right) \ge \P_p^{G_1}\left(\phi(K_v) \in A\right)
\]
for every increasing measurable set $A \subseteq \{0,1\}^{V_2}$.
\end{lemma}

\begin{proof}
For each edge $e_1 \in E_1$ with endpoints $x$ and $y$, let $\Phi(e_1)$ be the set of edges of $G_2$ belonging to some shortest path connecting $\phi(x)$ and $\phi(y)$ in $G_2$.
The definitions are easily seen to imply that there exists a constant $C=C(k,\alpha,\beta)$ such that 
$|\Phi(e_1)| \leq C$ for every $e_1 \in E_1$ and
$|\{e_1 \in E_1 : e_2 \in \Phi(e_1) \}| \leq C$ for every $e_2 \in E_2$. Let $q \in [0,1]$, let $(\eta(e_1,e_2))_{e_1 \in E_1, e_2 \in E_2}$ be a family of independent Bernoulli random variables each with probability $q$ to be $1$, and consider the random variables $\omega_1\in \{0,1\}^{E_1}$ and $\omega_2 \in \{0,1\}^{E_2}$ defined by
\begin{align*}
\omega_1(e_1) &= \mathbbm{1}\bigl[\eta(e_1,e_2)=1 \text{ for every $e_2 \in \Phi(e_1)$}\bigr] && \text{for every $e_1\in E_1$ and }\\
 \omega_2(e_2) &= \mathbbm{1}\bigl[\eta(e_1,e_2)=1 \text{ for some $e_1 \in E_1$ with $e_2 \in \Phi(e_1)$}\bigr] && \text{for every $e_2 \in E_2$.}
\end{align*}
Observe that for each $v\in V_1$ the image under $\phi$ of the cluster of $v$ in $\omega_1$ is contained in the cluster of $\phi(v)$ in $\omega_2$. Moreover, each of the families $(\omega_1(e_1))_{e_1 \in E_1}$ and $(\omega_2(e_2))_{e_2 \in E_2}$ have independent entries with $q^C \leq \P(\omega_1(e_1)=1) \leq q$ for every $e_1 \in E_1$ and $(1-q)^C \leq \P(\omega_2(e_2)=0) \leq 1-q$ for every $e_2 \in E_2$. It follows in particular that $\omega_1$ stochastically dominates Bernoulli-$q^C$ bond percolation on $G_1$ and $\omega_2$ is stochastically dominated by Bernoulli-$(1-(1-q)^C)$ bond percolation on $G_2$, from which the claim follows easily.
\end{proof}

It is a standard and easily verified fact that if $\Gamma$ is a group, $S_1$ and $S_2$ are generating sets of $\Gamma$, and $m\geq 1$ is such that $S_1 \subseteq \hat S_2^m$ then the identity function $\Gamma \to \Gamma$ is an $(m,1)$-rough embedding (i.e., an $m$-Lipschitz injection) from $\Cay(\Gamma,S_1)$ to $\Cay(\Gamma,S_2)$.
\cref{lem:roughembedding} therefore has the following immediate corollary.

\begin{corollary}
\label{cor:changing_generators}
For each $k$ and $m$ there exists a constant $C=C(k,m)$ such that the following holds. Let $f:[0,1]\to [0,1]$ be the increasing homeomorphism $f(p)= 1-(1-p^{1/C})^C$, let $\Gamma$ be a group, and let $S_1$ and $S_2$ be finite generating sets of $\Gamma$ of size at most $k$ and with $\hat S_1 \subseteq \hat S_2^m$. 
 Then 
for each $v \in \Gamma$ and $p\in [0,1]$, the cluster of $v$ in Bernoulli-$f(p)$ percolation on $\Cay(\Gamma,S_2)$ stochastically dominates the cluster of $v$ in Bernoulli-$p$ percolation on $\Cay(\Gamma,S_1)$.
\end{corollary}

\begin{remark}
\cref{lem:roughembedding,cor:changing_generators} could also be phrased slightly more strongly as statements concerning stochastic ordering of the \emph{random equivalence relations} given by connectivity.
\end{remark}

\begin{remark}
Russ Lyons has pointed out to us that our \cref{lem:roughembedding,cor:changing_generators} are essentially equivalent to \cite[Proposition 7.14 and Theorem 7.15]{LP:book}, which in turn are based on \cite[Theorem 6.1 and Remark 6.2]{LS99}. 
\end{remark}

\subsection{Quasitransitive graphs}\label{subsec:quasitrans}
We reduce \cref{thm:fin.perc,thm:gap} to the transitive case via the following folkloric result.
\begin{prop}\label{prop:quasitrans}
Let $n,k\in\N$ and suppose $G=(V,E)$ is a connected $n$-quasitransitive graph of degree at most $k$. Then there exists a connected transitive graph $G'=(V',E')$ of degree at most $(k+1)^{2n}$ satisfying $\diam(G')\le\diam(G)$ and an injective $(2n,n)$-rough isometry $G'\to G$. Moreover, if $G$ is finite then we may insist that $|V|/n\le|V'|\le|V|$.
\end{prop}

The proof of \cref{prop:quasitrans} begins with the following simple observation.
\begin{lemma}\label{lem:quasi.orbits.dense}
Let $n\in\N$, suppose $G=(V,E)$ is a connected $n$-quasitransitive graph and let $\Gamma=\Aut(G)$. Then every $v\in V$ lies at a distance of at most $n-1$ from each $\Gamma$-orbit of $V$.
\end{lemma}
\begin{proof}[Proof of \cref{lem:quasi.orbits.dense}]
This is equivalent to the claim that the quotient graph $G/\Gamma$ has diameter at most $n-1$, which is trivial since this graph is connected and has at most $n$ vertices.
\end{proof}

\begin{proof}[Proof of \cref{prop:quasitrans}]
Let $\Gamma=\Aut(G)$ and let $V'\subseteq V$ be some $\Gamma$-orbit, noting that if $G$ is finite then we may take $V'$ to be an orbit of maximum size so that $|V|/n\le|V'|\le|V|$ as required. Define $E'=\{(x,y)\in V'\times V':1\le d_G(x,y)\le2n\}$. Since $\Gamma$ acts by isometries on $G$ it also acts by isometries on $G'=(V',E')$, and this action is transitive by definition of $V'$. Moreover, the degree of a vertex $x$ in $G'$ is at most $|B_G(x,2n)|\le(k+1)^{2n}$, as required.

We claim that the inclusion map $V'\to V$ is a $(2n,n)$-rough isometry $G'\to G$. Indeed, \cref{lem:quasi.orbits.dense} shows that for every $v\in V$ there exists $u\in V'$ with $d_G(u,v)\le n$, and we trivially have that $d_G(u,v)\le2n d_{G'}(u,v)$ for every $u,v\in V'$.  On the other hand, given $u,v\in V'$ with $d_G(u,v)=m\in\N$, let $x_0=u,x_1,\ldots,x_{m-1},x_m=v\in V$ be such that $d(x_{i-1},x_i)=1$ for each $i=1,\ldots,m$. By \cref{lem:quasi.orbits.dense} there exist $y_0=u,y_1,\ldots,y_{m-1},y_m=v\in V'$ such that $d_G(x_i,y_i)\le n$ for each $i=0,\ldots,m$, and hence $d_{G'}(y_{i-1},y_i)\le1$ for each $i=1,\ldots,m$, so that $d_{G'}(u,v)\le m$.
This completes the proof that the inclusion $V'\to V$ is a $(2n,n)$-rough isometry $G'\to G$, and also proves that $G'$ is connected with diameter at most $\diam(G)$.
\end{proof}

\begin{corollary}\label{cor:quasitrans}
If \cref{thm:fin.perc,thm:gap} hold for $n=1$ then they hold for all $n\ge1$.
\end{corollary}

\begin{proof}
In the case of \cref{thm:gap} this is immediate from \cref{prop:quasitrans}, \cref{lem:roughembedding} and the fact that superlinear growth is preserved under rough isometries, so we concentrate on \cref{thm:fin.perc}. Let $k,n \geq 1$, and $\lambda>0$ and let $G=(V,E)$ be a finite, $n$-quasitransitive graph of degree at most $k$ satisfying
\[
\diam(G)\le\frac{\lambda|V|}{(\log|V|)^a}.
\]
By \cref{prop:quasitrans} there exists a transitive graph $G'=(V',E')$ of degree at most $(k+1)^{2n}$ with $|V|/n\le|V'|\le|V|$ and $\diam(G')\le\diam(G)$, and hence
\[
\diam(G')\le\diam(G)\le\frac{\lambda|V|}{(\log|V|)^a}\le\frac{\lambda n|V'|}{(\log|V'|)^a},
\]
and an injective $(2n,n)$-rough isometry $\phi:G'\to G$. If \cref{thm:fin.perc} holds for transitive graphs then we may conclude by that theorem and the second part of \cref{lem:cluster/connect} that for each $\eps>0$ there exists $\delta_1=\delta_1(k,n,\lambda,\eps)$ such that $\P_p^{G'}(u \leftrightarrow v) \geq 1-\eps$ for every $u,v \in V'$ and $p\geq 1-\delta_1$. 
 It then follows from \cref{lem:roughembedding} 
 that for each $\eps>0$ there exists $\delta_2=\delta_2(k,n,\lambda,\eps)$ such that $\P_p^{G}(u \leftrightarrow v) \geq 1-\eps$ for every $u,v \in \phi(V')$ and $p\geq 1-\delta_2$. It follows by Harris-FKG that $\P_p^{G}(u \leftrightarrow v) \geq (1-\eps)p^{2n}$ for every $u,v\in V$ and $p \geq 1-\delta_2$, from which the claim follows easily by the first part of \cref{lem:cluster/connect}.
\end{proof}

\section{Nilpotent and Abelian groups}
\label{sec:Nilpotent}

In this section we study percolation on Cayley graphs of Abelian and nilpotent groups. We study percolation in boxes in $\Z^d$ in \cref{subsec:boxes}, use techniques from additive combinatorics to generalize these results to arbitrary Abelian Cayley graphs in \cref{subsec:Abelian}, then use an inductive argument to study percolation on Cayley graphs of nilpotent groups in \cref{subsec:NilpotentNilpotent}. Finally, in \cref{subsec:virtuallynilpotent}, which can be read independently of the rest of the section, we prove uniform upper bounds on $p_c$ for Cayley graphs of infinite, virtually nilpotent groups.

\subsection{Percolation in an elongated Euclidean box}
\label{subsec:boxes}

Malon and Pak proved Theorem \ref{thm:fin.perc.ab} for certain specific types of generating sets called \emph{Hall bases} \cite[Theorem 1.2]{MR2238046}. In this section we prove a variant of their result that concerns percolation on \emph{boxes} rather than tori. 
In order to state this result it will be convenient to introduce some notation. Given $n_1,\ldots,n_d\in\N$, we define the \emph{box} $B(n_1,\ldots,n_d)\subset\Z^d$ via
\[
B(n_1,\ldots,n_d)=\bigl\{(x_1,\ldots,x_d)\in\Z^d:|x_i|\le n_i \text{ for every $1\leq i \leq d$}\bigr\}.
\]
We view $B(n_1,\ldots,n_d)$ as an induced subgraph of $\Z^d$, so that
\[
\diam(B(n_1,\ldots,n_d))=2(n_1+\cdots+n_d).
\]
We now establish an analogue of \cref{thm:fin.perc.ab} for Euclidean boxes. 
\begin{proposition}\label{prop:mal-pak}
Let $\lambda\ge1$ and let $n_1,\ldots,n_d\in\N$. Suppose that $B=B(n_1,\ldots,n_d)$ satisfies 
\begin{equation}\label{eq:mal-pak.hyp}
\diam(B)\le\frac{\lambda|B|}{\log |B|}.
\end{equation}
Then for every $\eps>0$ there exists $p_0=p_0(\lambda,\eps)$ such that
$\P_p(x\leftrightarrow y)\ge1-\eps$
for every $p\geq p_0$ and $x,y\in B$.
\end{proposition}

\begin{remark}
This recovers \cite[Theorem 1.2]{MR2238046} by Lemma \ref{lem:cluster/connect}.
\end{remark}

We first consider the case $d=2$. The analysis of this case follows by standard methods and is similar to \cite[Chapter 11.5]{grimmett2010percolation}, so we will keep our presentation brief and focus on the main conceptual ideas.

\begin{lemma}\label{lem:mal-pak.d=2}
For every $0<\lambda,\eps \leq 1$ there exists $p_0=p_0(\lambda,\eps)<1$ such that if $m,n\in \N$ satisfy $\lambda \log n\le m\le n$ then
$\P_p(x\leftrightarrow y)\ge1-\eps$
for every $p\geq p_0$  and $x,y\in B(n,m)$. 
\end{lemma}

\begin{remark}\label{rem:log.needed} 
One can also show conversely that if $p<1$ is fixed and $m=m(n)=o(\log n)$ then the torus $(\Z / n \Z) \times ( \Z / m \Z)$, and in particular its subgraph $B(n,m)$, does not contain a giant component with high probability under Bernoulli-$p$ percolation as $n\to\infty$. Indeed, this torus contains $\lfloor n/2\rfloor$ copies of the cycle $\Z/m\Z$ with disjoint $1$-neighbourhoods, and each such copy has all edges incident to it closed with probability $(1-p)^{3m}$, independently of all other copies. We deduce that if $n \gg (1-p)^{-3m}$ then there will exist with high probability many cycles having this property. Since the locations of these cycles are uniform among the $\lfloor n/2 \rfloor$ possibilities, their complement will not contain a giant cluster with high probability.
\end{remark}

In order to prove \cref{lem:mal-pak.d=2}, we first prove the following standard lemma concerning two dimensional bond percolation in a square box.

\begin{lemma}
\label{lem:foursides}
For each $p>1/2$ there exists a constant $c(p)>0$, with $c(p)\to1$ as $p\to1$, such that every point $x\in B(n,n)$ has probability at least $c(p)$ to be connected to all four sides of the box $B(n,n)$ in Bernoulli-$p$ bond percolation on $B(n,n)$.
\end{lemma}

\begin{proof} We will prove that for an arbitrary $p>1/2$ there exists a constant $c(p)>0$ such that every point $x\in B(n,n)$ has probability at least $c(p)$ to be connected to all four sides of the box $B(n,n)$ in Bernoulli-$p$ bond percolation on $B(n,n)$; using the fact that $\P_{p^\theta}(A) \geq \P_p(A)^\theta$ for every $p,\theta \in [0,1]$ and every increasing event $A$ \cite[Theorem 2.38]{grimmett2010percolation}, we may then take $c(p)\to 1$ as $p\to 1$ as required.
It follows from a standard duality argument \cite[Lemma 11.21]{grimmett2010percolation} that $B(n,n)$ has a left-right crossing with probability at least $1/2$ when $p\geq 1/2$, and hence by symmetry and Harris-FKG that $B(n,n)$ has both a left-right crossing and a top-bottom crossing with probability at least $1/4$ when $p\geq 1/2$. Moreover, it is a consequence of Kesten's theorem \cite{MR575895} that the critical probability for the quarter-plane $[0,\infty)^2 \subseteq \Z^2$ is $1/2$ (see e.g.\ \cite[Chapter 11.5]{grimmett2010percolation}), and hence that if $p>1/2$ then there exists $q=q(p)>0$ such that the origin is connected to infinity in the quarter-plane with probability at least $q$. 
Letting $x \in B(n,n)$ and considering the four copies of the quarter-plane with corner at $x$, we have by Harris-FKG that with probability at least $q^4/4$, $B(n,n)$ has both a left-right crossing and a top-bottom crossing and $x$ is connected to the boundary of $B(n,n)$ within each of the four quarter-planes with corner at $x$. On this event, we see by a simple topological argument that $x$ must be connected to all four sides of the box as required; see \cref{fig:pathcrossing} for an illustration.
\end{proof}

\begin{remark} For our purposes it would suffice to prove \cref{lem:foursides} for $p > 2/3$, say, where one can use elementary path counting arguments (i.e., the Peierls argument) in place of Kesten's theorem.
\end{remark}

\begin{figure}
\centering
\includegraphics[height=3.5cm]{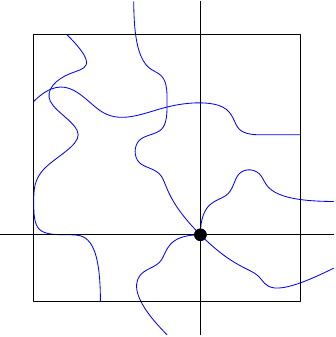} \qquad
\includegraphics[height=3.5cm]{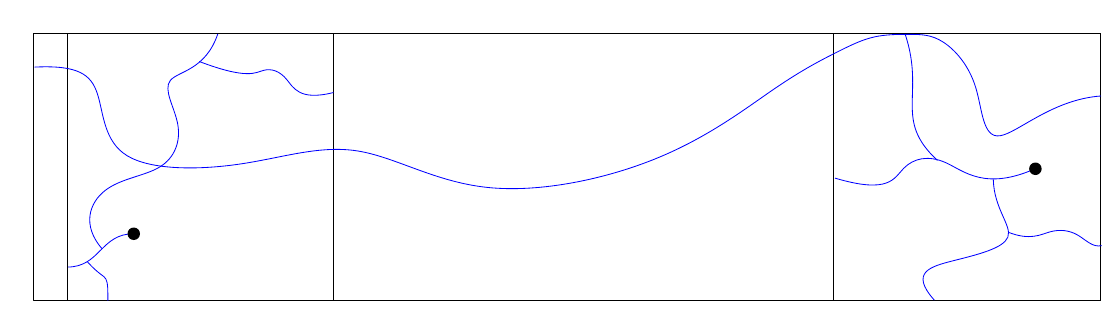}
\caption{Left: If a square box admits both a left-right and top-bottom open crossing and a point $x$ in the box is connected to infinity in all four quarter-planes with corner at $x$, then $x$ is connected to all four sides of the box within the box. Right: Consider a rectangular box $R$ of width greater than height, let $x$ and $y$ be points in the rectangle, and let $B_1$ and $B_2$ be square boxes contained in $R$, of height equal to that of $R$, and containing $x$ and $y$ respectively. If $R$ admits an open left-right crossing and $x$ and $y$ are connected to the top and bottom of $B_1$ and $B_2$ respectively by open paths, then $x$ and $y$ are connected by an open path.}
 \label{fig:pathcrossing}
\end{figure}

\begin{proof}[Proof of \cref{lem:mal-pak.d=2}]
This follows by a standard duality argument for percolation on the square lattice.
Consider an edge $e$ in the top of the rectangle $B(n,m)$. 
The expected number of dual paths of closed edges of length $r$ that start at $e$ is at most $3^{r-1}(1-p)^r$, and it follows by Markov's inequality that if $p>2/3$ then the probability that $e$ is connected to the bottom of the rectangle by a dual path of closed edges is at most $\sum_{r\geq m} 3^{r-1}(1-p)^r = (9p-6)^{-1}3^m(1-p)^m$. Since $m\geq \lambda \log n$, there exists $p_0=p_0(\lambda)<1$ such that if $p \geq p_0$ then this probability is at most $1/n^2$. If we take such a $p$, then it follows by a union bound that there does not exist any closed dual top-bottom crossing of the rectangle $B(n,m)$ with probability at least $1-2/n$. On this event there must exist an open path in the primal connecting the left and right sides of the rectangle.

On the other hand, for each $p> 1/2$ we have by \cref{lem:foursides} that there exists a constant $c(p)>0$, with $c(p) \to 1$ as $p\to 1$, such that every element of the box $[- m, m]^2$ has probability at least $c(p)$ to be connected to all four sides of the box $[-m,m]^2$ by open paths contained within this box. We can easily deduce from this that there is a giant cluster with high probability when $p$ is close to $1$. Indeed, if $x$ and $y$ are any two points in the rectangular box $B(m,n)$, then if $p$ is close enough to $1$, $x$ and $y$ each have probability at least $1 - \eps$ to be connected to both the top and bottom of the rectangle. By the Harris-FKG inequality, the probability that $x$ and $y$ are both connected to both the top and bottom of the rectangle and that there is an open left-right crossing of the rectangle is at least the product of these probabilities, and hence is close to $1$ when $p$ is close to $1$ and $n$ is large. On this event $x$ must be connected to $y$ (see \cref{fig:pathcrossing} for an illustration), and the claim is easily deduced. \qedhere

\end{proof}

We now apply the two-dimensional case to analyze the general case. Following Malon and Pak, we will do this by finding a constructing a homomorphic image of a two-dimensional box inside a box of general dimension, taking care to make sure the two-dimensional box has diameter of the correct order.

\begin{proof}[Proof of Proposition \ref{prop:mal-pak}]We follow Malon and Pak \cite[Theorem 1.2]{MR2238046}. 
For notational convenience we write $N_i=2n_i+1$, so that $|B|=\prod_{i=1}^dN_i$. We may assume without loss of generality that $n_1\le\cdots\le n_d$. We may also assume that $|B|\ge100e^\lambda$, noting that this combines with \eqref{eq:mal-pak.hyp} to force $d\ge2$. 
We claim that there exists $1 \leq k < d$ such that
\begin{equation}
\label{eq:mal-pak.claim}
N_1\cdots N_k\ge\lambda^{-1}\log|B| \qquad \text{ and } \qquad N_{k+1}\cdots N_d\ge\lambda^{-1}\log|B|.
\end{equation}
 Indeed, we will prove that if $k$ is minimal such that
$N_1\cdots N_k\ge\lambda^{-1}\log|B|$ then $k<d$ 
and $N_{k+1}\cdots N_d\ge\lambda^{-1}\log|B|$. 
If $k=1$ then this inequality is immediate since
\[
N_2 \cdots N_d \geq \sqrt{N_1 N_2 \cdots N_d} = \sqrt{|B|} \geq \lambda^{-1} \log |B|,
\]
where the final inequality follows by calculus and the assumption that $|B| \geq 100 e^\lambda$.
 We may therefore assume that $k>1$. In this case the bounds \eqref{eq:mal-pak.hyp} and $d\ge2$ imply that
\[
N_d\le \diam(B) \leq \frac{\lambda |B|}{\log|B|} = \frac{\lambda N_1\cdots N_d}{\log|B|},
\]
and hence that $k<d$.
If \eqref{eq:mal-pak.claim} does not hold then we have that
\begin{equation}\label{eq:mal-pak.contrap}
N_1\cdots N_k = \frac{|B|}{N_{k+1}\cdots N_d}>\frac{\lambda|B|}{\log|B|}
\end{equation}
and hence that
\begin{align*}
\log|B|&>\lambda N_{k+1}\cdots N_d&\text{(since \eqref{eq:mal-pak.claim} does not hold)\phantom{,}}\\
   &\ge\lambda N_k&\text{(since $k<d$ and the $N_i$ are increasing)\phantom{,}}\\
   &>\frac{\lambda^2N_1\cdots N_k}{\log|B|}&\text{(by minimality of $k$ and since $k>1$)\phantom{,}}\\
   &>\frac{\lambda^3|B|}{\log^2|B|}&\text{(by \eqref{eq:mal-pak.contrap}).}
\end{align*}
Since $\lambda \geq 1$ this is contrary to the assumption that $|B|\ge100$, so that $N_{k+1}\cdots N_d\ge\lambda^{-1}\log|B|$ as claimed.

Let $1 \leq k <d$ be such that \eqref{eq:mal-pak.claim} holds and set $m=\frac12(N_1\cdots N_k-1)$ and $n=\frac12(N_{k+1}\cdots N_d-1)$, noting that $m,n\in\N$ since the $N_i$ are all odd. It follows from \eqref{eq:mal-pak.claim} and the assumption that $|B|\geq 100 e^\lambda$ that
\begin{equation}\label{eq:size.m.n}
\min\{m,n\}\ge
\frac{\lambda^{-1}}{2} \log |B| -\frac{1}{2}
\geq
\frac{\lambda^{-1}}{4} \log |B|
\geq
\frac{\lambda^{-1}}4\log\max\{m,n\}.
\end{equation}

Recall that a \emph{Hamiltonian path} in a graph is a path that visits each vertex exactly once.
Continuing to follow Malon and Pak, choose Hamiltonian paths
\begin{align*}
\phi_1&:\{-m,\ldots,m\}\to B(n_1,\ldots,n_k),\\
\phi_2&:\{-n,\ldots,n\}\to B(n_{k+1},\ldots,n_d)
\end{align*}
such that $\phi_1(0)=0$ and $\phi_2(0)=0$, noting that such paths trivially exist, and note that the map $\phi=(\phi_1,\phi_2):B(m,n)\to B$ is a bijective graph homomorphism satisfying $\phi(0)=0$. In particular, $B(m,n)$ is isomorphic to a spanning subgraph of $B$, so that the desired result follows from \eqref{eq:size.m.n} and Lemma \ref{lem:mal-pak.d=2}.
\end{proof}

\subsection{Abelian groups}
\label{subsec:Abelian}

In this section we prove the following generalisation of Theorem \ref{thm:fin.perc.ab}.
Throughout this section we use additive notation for Abelian groups. In particular, we write $0$ for the identity element of an Abelian group, and given a subset $A$ of an abelian group and a positive integer $r$ we write $rA=\{a_1+\cdots+a_r:a_i\in A\}$. Given vertices $x$ and $y$ and a set of vertices $A$, we write $\{x \xleftrightarrow{A} y\}$ for the event that $x$ is connected to $y$ by an open path of edges with both endpoints in $A$.

\begin{theorem}\label{thm:fin.perc.ab.gen}
For each $k \geq 1$ and $\lambda,\eps\in (0,1]$ there exist constants $C=C_k \in \N$ and $p_0=p_0(k,\lambda,\eps)<1$ such that the following holds. Let $\Gamma$ be an Abelian group with generating set $S=\{x_1,\ldots,x_k\}$ and let $r\geq 1$ be such that $|r\hat S|\ge \lambda (r+1)\log (r+1)$.
Then 
\[
\P_p\bigl(x\xleftrightarrow{Cr\hat S} y\bigr)\geq 1-\eps
\]
for every $p\geq p_0$ and $x,y\in r\hat S$. 
\end{theorem}

\begin{remark*}
Theorem \ref{thm:fin.perc.ab.gen} is not true in an arbitrary group. For example, if $S$ is a generating set for a non-Abelian free group $\Gamma$ then $|\hat S^r|$ grows exponentially in $r$, but under percolation on $\Cay(\Gamma,S)$ we have $\P_p(x\leftrightarrow y)\to0$ as $d(x,y)\to\infty$ for every fixed $p<1$.
\end{remark*}

Before we prove Theorem \ref{thm:fin.perc.ab.gen}, let us confirm that it really does generalise Theorem \ref{thm:fin.perc.ab}.
\begin{proof}[Proof of Theorem \ref{thm:fin.perc.ab}]
This follows by applying \cref{thm:fin.perc.ab.gen} with $r=\diam_S(\Gamma)$ and using that $|\Gamma|\geq \diam_S(\Gamma)$, and then applying the first part of \cref{lem:cluster/connect}.
\end{proof}

We will prove \cref{thm:fin.perc.ab.gen} by constructing subgraphs isomorphic to Euclidean boxes inside $\Cay(\Gamma,S)$ and then applying the results of the previous subsection.
We first recall a standard definition from additive combinatorics.

\begin{definition}[Progressions]
Given elements $a_1,\ldots,a_k$ of an Abelian group $\Gamma$, and $L_1,\ldots,L_k\ge0$, define the \textbf{progression} $P=P_{a_1,\ldots,a_k}(L_1,\ldots,L_k)$ via
\[
P_{a_1,\ldots,a_k}(L_1,\ldots,L_k)=\{\ell_1a_1+\cdots+\ell_ka_k:|\ell_i|\le L_i\text{ for every }i\}.
\]
(Note that the $L_i$ need not be integers.)
The progression $P$ is called \textbf{proper} if each of its elements has a unique representation of the form $\ell_1a_1+\cdots+\ell_ka_k$ with $|\ell_i|\le L_i$ for every $i$. Note that if $L_1,\ldots,L_k$ are integers and $P=P_{a_1,\ldots,a_k}(L_1,\ldots,L_k)$ is a progression in an Abelian group then $m P = P_{a_1,\ldots,a_k}(mL_1,\ldots,mL_k)$ for every integer $m \geq 1$.
\end{definition}

We now state our main additive-combinatorial theorem powering the results of this section.

\begin{theorem}\label{prop:boxes}
For each $k\geq 1$ let $C_k=2^{6k}(k!)^3$. Let $\Gamma$ be an Abelian group with generating set $S=\{x_1,\ldots,x_k\}$.
  Then for each $r\geq 1$ there exist non-negative integers $L_1,\ldots,L_k \leq r$ such that the progression $P=P_{x_1,\ldots,x_k}(L_1,\ldots,L_k)$ is proper and satisfies $P \subseteq r \hat S \subseteq C_k(P+\hat S)$.

\end{theorem}

Note that we do not claim that the progression $C_kP$ is proper. Before proving this theorem, let us see how it implies \cref{thm:fin.perc.ab.gen}, and hence in particular \cref{thm:fin.perc.ab}.

\begin{proof}[Proof of \cref{thm:fin.perc.ab.gen} given \cref{prop:boxes}]
Let $k\geq 1$ and $0<\lambda,\eps\leq 1$.
Fix $r\geq 1$, an Abelian group $\Gamma$, and a generating set $S=\{x_1,\ldots,x_k\}$ of $\Gamma$ such that 
  $|r\hat S| \geq \lambda (r+1) \log (r+1)$. First note that
\begin{equation}\label{eq:fin.perc.ab.gen}
|r\hat S|\ge\frac{\lambda}{2}(r+1)\log|r\hat S|.
\end{equation}
Indeed, if $|r\hat S|\le (r+1)^2$ then this is immediate from the assumption that $|r\hat S| \geq \lambda (r+1) \log (r+1)$, whilst if $|r\hat S|\ge (r+1)^2$ then we have $|r\hat S|/\log|r\hat S|\ge|r\hat S|^{1/2}\ge r+1$, from which the claim is also immediate.

Let $C=C_k$ be the constant coming from Proposition \ref{prop:boxes} and let $P=P_{x_1,\ldots,x_k}(L_1,\ldots,L_k)$  be as in the statement of that proposition, so that $P$ is proper and
\begin{equation}\label{eq:P.in.rS}
P\subseteq r\hat S  \subseteq CP+C\hat S.
\end{equation}
Since $P$ is proper, there is a subgraph of $G=\Cay(\Gamma,S)$ with vertex set $P$ that is isomorphic to the box $B=B(L_1,\ldots,L_k)$. (The subgraph of $G$ \emph{induced} by $P$ may `wrap around the sides' of $P$ and be strictly larger than this subgraph, but this does not cause any problems.)
Moreover, we also have that
\begin{align*}
\diam(B)&\leq k(2r+1)&\text{(by the bound on $L_i$)\phantom{.}}\\
   &\le\frac{4k|r\hat S|}{\lambda\log|r\hat S|}&\text{(by \eqref{eq:fin.perc.ab.gen})\phantom{.}}\\
   &\le\frac{4C^k(2k+1)^{C+1}|B|}{\lambda\log|B|}&\text{(by \eqref{eq:P.in.rS}),}
\end{align*}
where we used that $|CP| \leq \prod_{i=1}^k (2CL_i+1) \leq C^k \prod_{i=1}^k (2L_i+1) = C^k|B|$ and $|CP+C\hat S| \leq |CP| \cdot |\hat S|^C \leq (2k+1)^C |CP|$ in the last line.
Proposition \ref{prop:mal-pak} therefore implies that there exists $p_0=p_0(k,\lambda,\eps)<1$ with $p_0 \geq (1-\eps)^{1/4C}$ such that  $\P_p(x\xleftrightarrow{P} y)\ge(1-\eps)^{1/4C}$ for every $x,y\in P$ and $p \geq p_0$. It follows trivially that
\[
\P_p(0 \xleftrightarrow{P+\hat S} x)\ge(1-\eps)^{1/4C}
\]
for every $p\geq p_0$ and $x\in P \cup \hat S$ also.
Let $y \in CP+C\hat S$, so that we can write $y = a_1 + a_2 + \ldots + a_\ell $ for some $\ell \leq 2C$ and $a_1,\ldots,a_\ell \in P \cup \hat S$. Writing $y_0=0$ for the identity of $\Gamma$ and $y_i=\sum_{j=1}^{i} a_j$ for each $0\leq i \leq \ell$, we deduce from the Harris-FKG inequality that
\[
\P\bigl(0 \xleftrightarrow{2CP+2C\hat S} y\bigr) \geq \prod_{i=0}^{\ell-1} \P\Bigl(y_i \xleftrightarrow{y_i+P+\hat S} y_{i+1} \Bigr) \geq (1-\eps)^{1/2}
\]
for every $p\geq p_0$ and $y \in CP+C\hat S$. A further application of Harris-FKG then yields that \[\P_p\bigl(x \xleftrightarrow{2CP+2C\hat S} y\bigr)\geq 1-\eps\] for every $p\geq p_0$ and $x,y\in CP+C\hat S$. The claim follows since $r \hat S \subseteq CP+C\hat S$ and $2CP \subseteq 2Cr \hat S$. \qedhere

\end{proof}

The remainder of this section is dedicated to proving \cref{prop:boxes}, and is of an entirely additive-combinatorial nature. In order to facilitate an inductive proof we will prove a more general and technical statement concerning progressions that are \emph{proper modulo a set}. We begin with some relevant definitions.

\begin{definition}[Divisibility by a subset]
Let $\Gamma$ be an Abelian group and let $Q\subseteq \Gamma$ be symmetric and contain the identity. If a subset $A\subset \Gamma$ satisfies
\begin{equation}\label{eq:equiv.rel}
(\forall x,y,z\in A)\Bigl(((x-y\in Q)\wedge (y-z\in Q))\implies x-z\in Q\Bigr)
\end{equation}
then we may define an equivalence relation ``$\equiv$ mod $Q$'' on $A$ by saying that $x\equiv y$ mod $Q$ if and only if $x-y\in Q$. We write $A/Q$ for the set of equivalence classes of this equivalence relation, and call $A/Q$ the \textbf{quotient} of $A$ by $Q$. If in addition to \eqref{eq:equiv.rel} we have that
\begin{equation}
(\forall x,x',y,y'\in A)(((x\equiv x'\text{ mod }Q)\,\wedge\,(y\equiv y'\text{ mod }Q))\implies(x+y\equiv x'+y'\text{ mod }Q)
\label{eq:divisibilitydef}
\end{equation}
then we say that $A$ is \textbf{divisible by $Q$}. Note that if $A$ is divisible by $Q$ then so is every subset of $A$. If $Q\subset \Gamma$ is symmetric and contains $0$, we will say that a progression $P \subseteq \Gamma$ is \textbf{proper mod $Q$} if it is proper, divisible by $Q$, and no two of its elements belong to the same equivalence class of $P/Q$. Equivalently, $P$ is proper mod $Q$ if it is proper and $x-y \notin Q$ for each two distinct elements $x,y\in P$. 
\end{definition}

These definitions also make sense in a non-Abelian setting, but we restrict attention to the Abelian case here for simplicity of notation.

These definitions satisfy the following elementary inductive property.

\begin{lemma}\label{eq:sum.proper}
Let $\Gamma$ be an Abelian group, and let $Q\subseteq \Gamma$ be symmetric and contain the identity.
If the progression $P_{a_1,\ldots,a_m}(K_1,\ldots,K_m)$ is proper mod $Q$ and  $P_{b_1,\ldots,b_n}(L_1,\ldots,L_n)$ is proper mod $P_{a_1,\ldots,a_m}(2K_1,\ldots,2K_m)+Q$ then $P_{a_1,\ldots,a_m,b_1,\ldots,b_n}(K_1,\ldots,K_m,L_1,\ldots,L_n)$ is proper mod $Q$.
\end{lemma}
\begin{proof}
Let $k_1,\ldots,k_m,k'_1,\ldots,k'_m,\ell_1,\ldots,\ell_n,\ell'_1,\ldots,\ell'_n\in\Z$ be such that $|k_i|,|k'_i|\le K_i$ for all $i$ and $|\ell_j|,|\ell'_j|\le L_j$ for all $j$. It suffices to prove that if
\[
k_1a_1+\cdots+k_ma_m+\ell_1b_1+\cdots+\ell_nb_n\in k'_1a_1+\cdots+k'_ma_m+\ell'_1b_1+\cdots+\ell'_nb_n+Q.
\]
then $\ell_i=\ell_i'$ for every $1\leq i \leq n$ and $k_j = k_j'$ for every $1 \leq j \leq m$. To prove this, we first rearrange to obtain that
\begin{multline*}
(\ell_1-\ell_1')b_1+\cdots+(\ell_n-\ell'_n)b_n\in (k'_1-k_1)a_1+\cdots+(k'_m-k_m)a_m+ Q \\\subseteq P_{a_1,\ldots,a_m}(2K_1,\ldots,2K_m)+Q.
\end{multline*}
The properness of $P_{b_1,\ldots,b_n}(L_1,\ldots,L_n)$ mod $P_{a_1,\ldots,a_m}(2K_1,\ldots,2K_m)+Q$ then implies that $\ell_i=\ell'_i$ for every $1\leq i \leq n$. It follows in particular that
\[
(k_1-k_1')a_1+\cdots+(k_m-k_m')a_m \in Q
\]
and the properness of $P_{a_1,\ldots,a_m}(K_1,\ldots,K_m)$ mod $Q$ implies that $k_j=k'_j$ for every $1\leq j\leq m$.
\end{proof}

We now prove the following proposition, which generalizes \cref{prop:boxes}. We will always apply this proposition with $Q=\{0\}$, but we include $Q$ in the statement in order to facilitate the inductive proof. Note also that we do not require that the set $A$ generates $\Gamma$.

\begin{proposition}\label{prop:proper.prog}
For each $k \geq 1$ let $C_k=2^{6k} (k!)^3$.
Let $\Gamma$ be an Abelian group, let $Q\subset \Gamma$ be symmetric and contain the identity, and let $A=\{a_1,\ldots,a_k\}\subset \Gamma$ be a subset of $\Gamma$ with $k$ elements. For each $r\in\N$ such that $r\hat A$ is divisible by $Q$ 
 there exist non-negative integers $L_1,\ldots,L_k\le r$ such that the progression $P = P_{a_1,\ldots,a_k}(L_1,\ldots,L_k)$ is proper mod $Q$ and satisfies 
 \[P \subseteq r\hat A\subseteq C_k (P+Q+\hat A).\] 
\end{proposition}

\begin{proof}
We prove the claim by induction on $k$, beginning with the base case $k=1$. To prove this case, let $A=\{a\}$ for some $a\in\Gamma$ and suppose that $r\hat A=\{-ra,\ldots,ra\}=P_a(r)$ is divisible by $Q$. If $P_a(r)$ is proper mod $Q$ then the proposition is satisfied, so we may assume not and set $L<r$ to be the maximum non-negative integer such that $P_a(L)$ is proper mod $Q$. By maximality there exist distinct $\ell,\ell'\in\Z$ with $|\ell|,|\ell|\le L+1$ such that $(\ell-\ell')a\in Q$. The divisibility of $P_a(r)$ by $Q$ then implies that $m(\ell-\ell')a\in Q$ for all $m\in\Z$ such that $|m(\ell-\ell')|\le r$, and so $r\hat A\subseteq P_a(|\ell-\ell'|)+Q$. In particular, we have that $P_a(L)\subseteq r\hat A\subseteq P_a(|\ell-\ell'|)+Q \subseteq P_a(2L+2)+Q \subseteq2(P_a(L)+Q+\hat A)$ and the proposition is satisfied. 

Now suppose that $k\geq 2$, let $\Gamma$, $Q$, $A$, and $r$ be as in the statement, and suppose that the claim has already been proven for all smaller values of $k$. It suffices to prove that if $r$ satisfies the additional assumption that $r \hat A \not \subseteq (r-1)\hat A +Q$, or equivalently that
\begin{equation}\label{eq:r<diam}
\text{there exists $x\in r \hat A$ such that $x \not\equiv y$ mod $Q$ for every $y \in (r-1)\hat A$},
\end{equation}
then 
there exist non-negative integers $L_1,\ldots,L_k\le r$ such that  $P = P_{a_1,\ldots,a_k}(L_1,\ldots,L_k)$ is proper mod $Q$ and satisfies 
 \[
 P \subseteq r\hat A\subseteq (C_k-1) (P+Q+\hat A).
 \]
 Indeed, suppose that we have proven this claim and that $r$ does \emph{not} satisfy \eqref{eq:r<diam}. Then either $r \hat A \subseteq Q$, in which case the claim is trivial, or we may take $r'\in\N$ to be 
maximal with $1\leq r' \leq r$ such that there exists $x \in r' \hat A$ that is not equal to any element of $(r'-1) \hat A$ mod $Q$.
Since $r \hat A$ is divisible by $Q$, so that equivalence mod $Q$ is an equivalence relation on $r \hat A$, the maximality of $r'$ implies that every element of $r \hat A$ is equivalent to an element of $r' \hat A$ mod $Q$, so that $r \hat A \subseteq r' \hat A + Q$. Applying the claim to the set $r' \hat A$, we obtain that there exist non-negative integers $L_1,\ldots,L_k\le r' \leq r$ such that the progression $P = P_{a_1,\ldots,a_k}(L_1,\ldots,L_k)$ is proper mod $Q$ and satisfies $P \subseteq r'\hat A\subseteq (C_k-1) (P+Q+\hat A)$, so that
$P \subseteq r \hat A\subseteq C_k (P+Q+\hat A)$ as required.

We now carry out the induction step under the additional hypothesis \eqref{eq:r<diam} as discussed above.
 By condition \eqref{eq:r<diam} we may pick $x\in r\hat A \setminus((r-1)\hat A+Q)$. Note then that $x=m_1a_1+\cdots+m_ka_k$ for some $m_1,\ldots,m_k\in\Z$ such that $|m_1|+\cdots+|m_k|=r$. By relabelling the generators $a_1,\ldots,a_k$ if necessary, we may assume without loss of generality that $m_k=\max_i|m_i|$, which implies in particular that
\begin{equation}\label{eq:md=r}
\frac{r}{k}\le m_k\le r.
\end{equation}
We next claim that if $\ell_1,\ldots,\ell_k\in\Z$ satisfy $|\ell_k|\le m_k$ and $\ell_1a_1+\cdots+\ell_ka_k\in Q$ then 
\[|\ell_k|\le|\ell_1|+\cdots+|\ell_{k-1}|.\] 
Indeed, since $Q$ is symmetric we may assume that $0\le\ell_k\le m_k$. For such $(\ell_i)_{i=1}^k$ we have trivially that 
$x\in(m_1-\ell_1)a_1+\cdots+(m_k-\ell_k)a_k+Q$, and using that $x\notin(r-1)\hat A+Q$ we deduce that
\begin{multline*}
r \leq |m_1-\ell_1|+\cdots+|m_k-\ell_k|\le|m_1|+\cdots+|m_k|+|\ell_1|+\ldots+|\ell_{k-1}|-|\ell_k|\\
     =r+|\ell_1|+\ldots+|\ell_{k-1}|-|\ell_k|
\end{multline*}
as claimed.
Applying this claim with $\ell_1=\cdots=\ell_{k-1}=0$ shows that the progression $P_{a_k}(m_k/2)$ is proper mod $Q$. Similarly, taking arbitrary $\ell_1,\ldots,\ell_{k-1}$ with $|\ell_i|\le m_k/4k$ establishes the implication
\begin{multline}\label{eq:m/2->m/4}
\Bigl(x-y\in P_{a_k}\Bigl(\frac{m_k}{2}\Bigr)+Q \Bigr) \implies \Bigl(x-y\in P_{a_k}\Bigl(\frac{m_k}{4}\Bigr)+Q\Bigr) \\ \text{ for all } x,y\in P_{a_1,\ldots,a_{k-1}}\Bigl(\frac{m_k}{4k},\ldots,\frac{m_k}{4k}\Bigr).
\end{multline}

We claim that $P_{a_1,\ldots,a_{k-1}}(m_k/4k,\ldots,m_k/4k)$ is divisible by $P_{a_k}(m_k/2)+Q$. We start the proof of this claim by showing that the quotient of $P_{a_1,\ldots,a_{k-1}}(m_k/4k,\ldots,m_k/4k)$ by $P_{a_k}(m_k/2)+Q$ is well defined (i.e., that \eqref{eq:equiv.rel} holds). Given $x,y,z\in P_{a_1,\ldots,a_{k-1}}(m_k/4k,\ldots,m_k/4k)$ with $x-y\in P_{a_k}(m_k/2)+Q$ and $y-z\in P_{a_k}(m_k/2)+Q$, it follows from \eqref{eq:m/2->m/4} that there exist $u,v\in P_{a_k}(m_k/4)$ such that $x\equiv y+u\text{ mod }Q$ and $y+u\equiv z+u+v\text{ mod }Q$ (equivalence mod $Q$ being well defined for $x$, $y+u$ and $z+u+v$ since they all belong to $r\hat A$, which is divisible by $Q$). We deduce from this that $x\equiv z+u+v\text{ mod }Q$, which implies in particular that $x-z\in P_{a_k}(m_k/2)+Q$ as required. To prove moreover that $P_{a_1,\ldots,a_{k-1}}(m_k/4k,\ldots,m_k/4k)$ is \emph{divisible} by $P_{a_k}(m_k/2)+Q$ (i.e., to verify \eqref{eq:divisibilitydef}), suppose that $x,x',y,y'\in P_{a_1,\ldots,a_{k-1}}(m_k/4k,\ldots,m_k/4k)$ satisfy $x\equiv x'\text{ mod }P_{a_k}(m_k/2)+Q$ and $y\equiv y'\text{ mod }P_{a_k}(m_k/2)+Q$. As before, \eqref{eq:m/2->m/4} implies that there exist $u,v\in P_{a_k}(m_k/4)$ such that $x\equiv x'+u\text{ mod }Q$ and $y\equiv y'+v\text{ mod }Q$, and since $r\hat A$ is divisible by $Q$ it follows that $x+y\equiv x'+y'+u+v\text{ mod }Q$. This in turn implies that $x+y-x'-y'\in  P_{a_k}(m_k/2)+Q$, and hence that $x+y\equiv x'+y'\text{ mod }P_{a_k}(m_k/2)+Q$ as required. Note moreover that writing $A'=\{a_1,\ldots,a_{k-1}\}$ we have that $\lfloor\frac{m_k}{4k}\rfloor\hat A'$ is a subset of $P_{a_1,\ldots,a_{k-1}}(m_k/4k,\ldots,m_k/4k)$, and is therefore divisible by $P_{a_k}(m_k/2)+Q$ also. 

Define $n \in\N\cup\{0\}$ to be minimal such that $\lfloor\frac{m_k}{4k}\rfloor\hat A'\subseteq n\hat A'+P_{a_k}(m_k/2)+Q$, noting that $n\leq m_k/ 4k$ and, by \eqref{eq:m/2->m/4}, that
\begin{equation}
\label{eq:prop.prog.case.2}
\left\lfloor\frac{m_k}{4k}\right\rfloor\hat A'\subseteq n\hat A'+P_{a_k}\left(\frac{m_k}{4}\right)+Q.
\end{equation}
Applying the induction hypothesis to the sets $A'$ and $P_{a_k}(m_k/2)+Q$, we deduce that there exist integers 
$L_1,\ldots,L_{k-1}\le n$
such that $P_{a_1,\ldots,a_{k-1}}(L_1,\ldots,L_{k-1})$ is proper mod $P_{a_k}(m_k/2)+Q$ and satisfies
\begin{equation}\label{eq:kA.in.P}
n\hat A'\subseteq C_{k-1}P_{a_1,\ldots,a_{k-1}}(L_1,\ldots,L_{k-1})+C_{k-1}P_{a_k}\left(\frac{m_k}{2}\right)+C_{k-1}Q+C_{k-1}\hat A.
\end{equation}
Set $L_k=\lfloor m_k/4k \rfloor$, so that $\lfloor r / 4k^2 \rfloor \leq L_k \leq r/4k$ by \eqref{eq:md=r}. Considering separately the cases $L_k=0$ and $L_k \geq 1$ yields that
\begin{equation}
\label{eq:progression_and_a_bit}
P_{a_k}(r) \subseteq 4k^2 \hat A + 8k^2 P_{a_k}(L_k).
\end{equation}
Since the progression $P_{a_1,\ldots,a_{k-1}}(L_1,\ldots,L_{k-1})$ is proper mod $P_{a_k}(m_k/2)+Q$ and $P_{a_k}(m_k/2k)$ is proper mod $Q$, we may apply \cref{eq:sum.proper} to deduce 
 that the progression 
 $P_{a_1,\ldots,a_k}(L_1,\ldots,L_k)$ is proper mod $Q$ as required. This progression is also clearly contained in $r \hat A$ since $L_i \leq r/4k$ for every $1\leq i \leq k$. 
   It follows from \eqref{eq:prop.prog.case.2} and \eqref{eq:kA.in.P} that 
\[
L_k \hat A'\subseteq 
C_{k-1}P_{a_1,\ldots,a_{k-1}}(L_1,\ldots,L_{k-1})+C_{k-1}P_{a_k}\left(\frac{m_k}{2}\right)+C_{k-1}Q +C_{k-1}\hat A + P_{a_k}\left(\frac{m_k}{4}\right)+Q.
\]
Noting that $P_{a_k}(m_k/2)\subseteq4kP_{a_k}(L_k)+2k\hat A$ (the term $2k\hat A$ being necessary only if $L_k=0$), we deduce that
\begin{align*}
L_k \hat A'
&\subseteq C_{k-1}P_{a_1,\ldots,a_{k-1}}(L_1,\ldots,L_{k-1})+(C_{k-1}+1)\Big(4kP_{a_k}(L_k)+2k\hat A\Big)+(C_{k-1}+1)Q +C_{k-1}\hat A\\
&\subseteq 4k(C_{k-1}+1)\Big(P_{a_1,\ldots,a_k}(L_1,\ldots,L_k)+Q\Big)+(C_{k-1}+1)(2k+1)\hat A.
\end{align*}
It follows from \eqref{eq:md=r} that $r\le km_k\le8k^2L_k+4k^2$, so we deduce that
\[
r \hat A'\subseteq 32k^3(C_{k-1}+1)\Big(P_{a_1,\ldots,a_k}(L_1,\ldots,L_k)+Q+\hat A\Big).
\]
It follows from this and \eqref{eq:progression_and_a_bit} that
\[
r \hat A \subseteq r \hat A' + P_{a_k}(r) \subseteq 32 k^3 (C_{k-1}+2) \Bigl( P_{a_1,\ldots,a_k}( L_1,\ldots, L_k) + Q +\hat A\Bigr).
\]
The claim follows since $32k^3 (C_{k-1}+2) + 1 \leq 64 k^3 C_{k-1}=C_k$ for every $k\geq 2$. \qedhere
\end{proof}

\begin{proof}[Proof of \cref{prop:boxes}]
This follows by applying \cref{prop:proper.prog} with $A=S$ and $Q=\{0\}$.
\end{proof}

\subsection{Finite nilpotent groups}
\label{subsec:NilpotentNilpotent}

In this section we extend Theorem \ref{thm:fin.perc.ab} to finite nilpotent groups. We first recall some basic relevant facts about nilpotent groups, referring the reader to e.g.\ \cite[\S5.2]{MR3971253} or \cite[Ch. 10]{MR0103215} for more detailed background. Let $\Gamma$ be a group.
Given elements $x_1,\ldots,x_k$ of $\Gamma$, the \textbf{simple commutator} $[x_1,\ldots,x_k]$ is defined recursively by $[x_1,x_2]=x_1^{-1}x_2^{-1}x_1x_2$ and  $[x_1,\ldots,x_k]=[[x_1,\ldots,x_{k-1}],x_k]$. Given two subgroups $H_1$ and $H_2$ of $\Gamma$, we define $[H_1,H_2]=\langle [h_1,h_2] : h_1 \in H_1, h_2 \in H_2\rangle$. The \textbf{lower central series}  of $\Gamma$ is the nested sequence of normal subgroups $\Gamma_1 > \Gamma_2 > \cdots$ defined recursively by
\[
\Gamma_1 = \Gamma \qquad \text{ and } \qquad \Gamma_{i+1} = [\Gamma_i,\Gamma] \quad \text{ for each $i\geq 1$}.
\]
Each subgroup $\Gamma_i$ appearing in the lower central series is easily seen to be characteristic in $\Gamma$ (that is, $\Gamma_i$ is fixed by every automorphism of $\Gamma$), and by definition $\Gamma_i/\Gamma_{i+1}$ is central in $\Gamma/\Gamma_{i+1}$ for every $i\geq 1$ (that is, every element of $\Gamma_i/\Gamma_{i+1}$ commutes with every element of $\Gamma/\Gamma_{i+1}$).
The group $\Gamma$ is said to be \textbf{nilpotent} if there exists $s \geq 1$ such that $\Gamma_{s+1} =\{\mathrm{id}\}$. Note in this case that $\Gamma_i =\{\mathrm{id}\}$ for every $i>s$ and $\Gamma_s=\Gamma_s/\Gamma_{s+1}$ is central in $\Gamma=\Gamma/\Gamma_{s+1}$. The minimal such $s$ is known as the \textbf{step} of $\Gamma$, and is equal to $1$ if and only if $\Gamma$ is Abelian.

The primary goal of this section is to prove the following proposition.
\begin{proposition}\label{prop:fin.perc.nilp}
For each $k,s \in \N$, $\lambda \geq 1$, and $\eps>0$ there exists $p_0(k,s,\lambda,\eps)<1$ such that the following holds: If 
 $\Gamma$ is a finite $s$-step nilpotent group and $S$ is a generating set for $\Gamma$ of size at most $k$ satisfying
\begin{equation}\label{eq:nilp.hyp}
\diam_S(\Gamma)\le\frac{\lambda|\Gamma|}{(\log|\Gamma|)^{s}} 
\end{equation}
then Bernoulli-$p$ bond percolation on $\Cay(\Gamma,S)$ satisfies $\P_p(x\lra y)\ge1-\eps$ for every $p\geq p_0$ and $x,y\in \Gamma$. 
\end{proposition}

The proof of \cref{prop:fin.perc.nilp} will proceed by induction on the step $s$ of the group, with the Abelian case already being handled by \cref{thm:fin.perc.ab.gen}. The proof will rely on the fact that if $\Gamma$ is nilpotent of step $s$ then $\Gamma/\Gamma_s$ is nilpotent of step $s-1$.
 We begin by recalling some relevant basic facts about nilpotent groups.

\begin{lemma}[Multilinearity of commutators {\cite[Lemma 5.5.2]{MR3971253}}]\label{lem:comm.linear}
Let $j\ge2$ and let $\Gamma$ be a group. Then the map
\[
\begin{array}{ccccc}
\phi_j&:&\Gamma^j&\to&\Gamma_j\\
&&(\gamma_1,\ldots,\gamma_j)&\mapsto&[\gamma_1,\ldots,\gamma_j]
\end{array}
\]
is a homomorphism in each variable modulo $\Gamma_{j+1}$. Moreover, if $\gamma_i\in[\Gamma,\Gamma]$ for some $i$ then $\phi_j(\gamma_1,\ldots,\gamma_j)\in\Gamma_{j+1}$. In particular, if $\Gamma$ is $j$-step nilpotent group then $\phi_j$ is a homomorphism in each variable and the commutator subgroup $[\Gamma,\Gamma]$ is in the kernel of each of these homomorphisms.
\end{lemma}

\begin{lemma}[{\cite[Proposition 5.2.6]{MR3971253}}]
\label{lem:G_s=<comms>}
Let $s\ge2$ and suppose $\Gamma$ is an $s$-step nilpotent group with generating set $S$. Then $\Gamma_s$ is generated by the set $S_s:=\{[x_1,\ldots,x_s]:x_1,\ldots,x_s\in S\}$.
\end{lemma}

We will split the proof of the induction step into two cases according to whether $|\Gamma_s| \geq \log |\Gamma|$ or $|\Gamma_s| < \log |\Gamma|$. It will be useful to know that $|\Gamma_s|$ cannot be too large. We recall that the \textbf{rank} of a finitely generated group $\Gamma$ is defined to be the minimal cardinality of a generating set of $\Gamma$.

\begin{lemma}[{\cite[Lemma 4.13]{MR3439705}}]
\label{lem:Gamma_s_is_small}
For each $s\geq 2$ and $k\geq 1$ there exists $\eta=\eta(s,k)>0$ such that if $\Gamma$ is an $s$-step nilpotent group with rank at most $k$ then $|\Gamma_s|\leq |[\Gamma,\Gamma]| \leq |\Gamma|^{1-\eta}$.
\end{lemma}

We now treat the case that $|\Gamma_s| \geq \log |\Gamma|$, for which we will be able to use a rather general argument that does not use the induction hypothesis and is not specific to the nilpotent setting. Let $\Gamma$ be a group with finite generating set $S$, and let $H$ be a subgroup. Given $r\geq 1$, we say that $H$ is \textbf{$r$-quasiconnected} if $\hat S^r \cap H$ generates $H$, or equivalently if the cosets of $H$ are connected as subsets of $\Cay(\Gamma,\hat S^r)$. In particular, \cref{lem:G_s=<comms>} implies that if $\Gamma$ is $s$-step nilpotent then $\Gamma_s$ is $C_s$-quasiconnected with respect to any generating set. The next proposition shows very generally that the existence of a quasiconnected central subgroup of moderate size implies non-triviality of the percolation phase transition. As above, we say that a subgroup $H$ of a group $\Gamma$ is \textbf{central} if $\gamma h = h \gamma$ for every $h\in H$ and $\gamma\in \Gamma$, noting that central subgroups are always normal.

\begin{proposition}
\label{prop:the_moderate_centre}
For each $\lambda,\eps>0$ and $k,r\geq 1$ there exists a constant $p_0=p_0(\lambda,\eps,k,r)<1$ such that the following holds. If $\Gamma$ is a finite group with generating set $S$ of size at most $k$ and $H$ is a central, $r$-quasiconnected subgroup of $\Gamma$ such that
\begin{equation}
\label{eq:central_subgroup_assumption}
|H| \geq \lambda \log |\Gamma| \qquad \text{ and } \qquad |\Gamma/H| \geq \lambda \log |\Gamma|
\end{equation}
then Bernoulli-$p$ bond percolation on $\Cay(\Gamma,S)$ satisfies $\P_p(x\leftrightarrow y) \geq 1-\eps$ for every $p\geq p_0$ and $x,y\in \Gamma$.
\end{proposition}

\begin{proof}
We will construct a surjective rough embedding of a Euclidean box onto $\Cay(\Gamma,S)$. The box we construct will satisfy the hypotheses of \cref{prop:mal-pak} by \eqref{eq:central_subgroup_assumption}, so that we can conclude by applying that proposition together with \cref{lem:roughembedding}.

Let $\pi:\Gamma\to\Gamma/H$ be the projection map.
It suffices by \cref{cor:changing_generators} to prove the assertion about the non-triviality of the percolation phase transition for $G = \Cay(\Gamma,\hat S^r)$ rather than $\Cay(\Gamma,S)$. Let $G_1$ be the subgraph of $G$ induced by $H$, which is isomorphic to $\Cay(H,\hat S^r \cap H)$, and let $G_2$ be $\Cay(\Gamma/H,\hat S^r/H)$. 
Recall that every finite graph all of whose degrees are even admits an Eulerian circuit, i.e., a cycle that passes through every edge exactly once \cite[Theorem 1.8.1]{diestel2000graph}. It follows that \emph{every} finite graph admits a path that visits every vertex and crosses each edge exactly twice.
Thus, there exist $n_1,n_2 \geq 1$ and surjective functions $\phi_1:\{1,\ldots,n_1\} \to H$ and $\phi_2:\{1,\ldots,n_2\} \to \Gamma/H$ such that
\begin{itemize}
\item $\phi_i(j)$ and $\phi_i(j+1)$ are adjacent in $G_i$ for each $i\in \{1,2\}$ and $1 \leq j \leq n_i-1$ and
\item the path in $G_i$ associated to $\phi_i$ crosses each edge of $G_i$ at most twice.
\end{itemize} In particular, the path in $G_i$ associated to $\phi_i$ visits each \emph{vertex} of $G_i$ at most $2(2k+1)^r$ times, so that $|H| \leq n_1 \leq 2(2k+1)^r |H|$ and $|\Gamma/H| \leq n_2 \leq 2(2k+1)^r |\Gamma/H|$. Since we also have that $|\Gamma| = |H|\cdot|\Gamma/H|$, it follows from \eqref{eq:central_subgroup_assumption} that there exists a positive constant $\tilde \lambda = \tilde \lambda(\lambda,k,r)$ such that $\min\{n_1,n_2\} \geq \tilde \lambda \log \max \{n_1,n_2\}$.

 Let $B$ be the subgraph of $\Z^2$ induced by  $\{1,\ldots,n_1\} \times \{1,\ldots,n_2\}$.
  For each $1 \leq j \leq n_2-1$ let $s_j \in \hat S^r$ be such that $\phi_2(j+1)=\phi_2(j) \pi(s_j)=\phi_2(1)\pi(s_1\cdots s_j)$.
 We  define
$\phi: \{1,\ldots,n_1\} \times \{1,\ldots,n_2\} \to \Gamma$ by
  $\phi(a,1)=\phi_1(a)$ for each $a\in \{1,\ldots,n_1\}$ and
\[
\phi(a,j) = \phi_1(a) s_1 \cdots s_{j-1}
\qquad \text{ for each $1\leq a \leq n_1$ and $2 \leq j \leq n_2$.}
\]
The centrality of $H$ implies that the map $\phi$ defines a graph homomorphism $B\to G$, i.e., that 
 $\phi(a,j)$ and $\phi(a',j')$ are adjacent in $G$ whenever $|a-a'|+|j-j'|=1$. 
Equivalently, $\phi(a,j)^{-1} \phi(a',j') \in \hat S^r$ whenever $|a-a'|+|j-j'|=1$. Indeed, if $a'=a$ and $j'=j+1$ we have trivially that
\[
\phi(a,j)^{-1}\phi(a,j+1)= s_{j-1}^{-1} \cdots s_1^{-1}\phi_1(a)^{-1}\phi_1(a) s_1 \cdots s_{j} = s_j \in \hat S^r
\]
while if $a'=a+1$ and $j'=j$ then we have by centrality of $H$ that
\[
\phi(a,j)^{-1}\phi(a+1,j) = s_{j-1}^{-1} \cdots s_1^{-1}\phi_1(a)^{-1}\phi_1(a+1) s_1 \cdots s_{j-1} = \phi_1(a)^{-1}\phi_1(a+1),
\]
which belongs to $\hat S^r$ since $\phi_1(a)$ and $\phi_1(a+1)$ are adjacent in $G_1$.

Finally, we observe that the preimage of each element of $\Gamma$ under $\phi$ has between $1$ and $4(2k+1)^{2r}$ elements. Indeed, $\phi(a,j)$ belongs to the right coset of $H$ determined by $\pi(s_1 \cdots s_{j-1})$, so that for each $\gamma\in \Gamma$ there are at most $2(2k+1)^r$ values of $j$ for which $\phi(a,j)$ lies in the same $H$ coset as $\gamma$, and for each $j$ there are at most $2(2k+1)^r$ values of $a$ for which $\phi_1(a)=\gamma s_{j-1}^{-1} \cdots s_1^{-1}$. Thus
$\phi$ is a $(1,4(2k+1)^{2r})$-rough embedding of $B$ into $G$. Since $\phi$ is surjective on vertices and $\min\{n_1,n_2\} \geq \tilde \lambda \log \max \{n_1,n_2\}$ the claim follows from \cref{prop:mal-pak,lem:roughembedding}.
\end{proof}

\cref{prop:the_moderate_centre} together with \cref{lem:G_s=<comms>,lem:Gamma_s_is_small} immediately handles the case of \cref{prop:fin.perc.nilp} in which $|\Gamma_s| \geq \log |\Gamma|$. We next address the case that $|\Gamma_s|<\log |\Gamma|$, for which the induction hypothesis will actually be used. We begin with some preliminaries.
Let $\Gamma$ be an $s$-step nilpotent group with  finite generating set $S$, so that $\Gamma_s$ is generated by $S_s = \{[x_1,\ldots,x_s] : x_1,\ldots,x_s \in S\}$ by \cref{lem:G_s=<comms>}.
  Since $\Gamma_s$ is central in $\Gamma$, we have moreover that $\langle z \rangle \Gamma_s$ is an Abelian subgroup of $\Gamma$ for each $z\in \Gamma$ and that $\langle z \rangle \Gamma_s$ is generated by $S_s \cup \{z\}$. The next lemma shows that when $s\geq 2$ we can always take $z\in S$ in such a way that the Abelian group $\langle z \rangle \Gamma_s$ ``looks at least two-dimensional'' within a ball that contains $\Gamma_s$.

\begin{lemma}
\label{lem:diameter_induction2}
 Let $s\geq 2$ and let $\Gamma$ be an $s$-step nilpotent group with generating set $S$ of size $k$. Then there exists  $z\in S$ and $r \geq 1$ such that
\[
\hat S_s^r = \Gamma_s \qquad \text{ and } \qquad |(\hat S_s \cup\{z,z^{-1}\})^r| \geq \frac{r^2}{4k^{2s}}.
\]
\end{lemma}

\begin{proof}
Enumerate $S_s=\{[x_1,\ldots,x_s]:x_1\ldots,x_s \in S\} = \{c_1,\ldots,c_m\}$, noting that $m=|S_s| \leq k^s$.
 Let $z$ be an element of $S$ of maximal order modulo $[\Gamma,\Gamma]$ and consider the set 
 \[T=\Bigl\{z^{\epsilon_0}c_1^{\epsilon_1}\cdots c_m^{\epsilon_m}:\epsilon_i\in\{-1,0,1\}\Bigr\},\]
 which is contained in $(\hat S_s \cup \{z,z^{-1}\})^{m+1}$.  \cref{lem:G_s=<comms>} implies the set $T$ generates the group $\langle z\rangle \Gamma_s$, which is Abelian since $\Gamma_s$ is central. Let $r$ be the largest order of an element of $S_s$, so that
$\hat S_s^{mr} = \Gamma_s$ by commutativity and \cref{lem:G_s=<comms>}.
  We claim that $z$ must have order at least $r$ modulo $[\Gamma,\Gamma]$. Indeed, letting $[x_1,\ldots,x_s] \in S_s$ have order $r$, we have by \cref{lem:comm.linear} that $[x_1^\ell,x_2,\ldots,x_s]=[x_1,x_2,\ldots,x_s]^\ell \neq \mathrm{id}$ for every $\ell < r$. It follows by a further application of \cref{lem:comm.linear} that $x_1\in S$ has order at least $r$ mod $[\Gamma,\Gamma]$, and hence that $z$ has order at least $r$ modulo $[\Gamma,\Gamma]$ as claimed by maximality.
Let $c \in S_s$ be of order $r$. Since $\Gamma_s<[\Gamma,\Gamma]$ the group elements $\{z^i c^j: 1\leq i,j\leq r\}$ are all distinct, and so $|\hat T^r| \geq r^2$. Since $\hat T \subseteq (\hat S \cup \{z,z^{-1}\})^{m+1}$ it follows that
\[
\hat S_s^{(m+1)r} = \Gamma_s \qquad \text{ and } \qquad |(\hat S_s \cup\{z,z^{-1}\})^{(m+1)r}| \geq r^2,
\]
which implies the claim since $m +1\leq k^s +1 \leq 2k^s$.
\end{proof}

We are now ready to complete the proof of \cref{prop:fin.perc.nilp}.

\begin{proof}[Proof of \cref{prop:fin.perc.nilp}]
We prove the proposition by induction on the step $s$ of the nilpotent group, the Abelian $s=1$ base case having already been handled by \cref{thm:fin.perc.ab.gen}. 
Let $s\geq 2$, $k\in \N$, $\lambda \geq 1$, and $\eps>0$. Let  $\Gamma$ be a finite $s$-step nilpotent group and let $S$ be a generating set for $\Gamma$ of size at most $k$ satisfying 
\[
 \diam_S(\Gamma)\le\frac{\lambda|\Gamma|}{(\log|\Gamma|)^{s}}.
\]
We need to prove that there exists $p_0=p_0(s,k,\lambda,\eps)<1$ such that if $p \geq p_0$ then Bernoulli-$p$ bond percolation on $\Cay(\Gamma,S)$ satisfies $\P_p(x\leftrightarrow y)\geq 1-\eps$ for every $x,y\in \Gamma$.
We may split into two cases according to whether or not $|\Gamma_s| \geq \log |\Gamma|$, taking $p_0$ to be the maximum of the constants produced in the two cases.

\medskip

In the first case, $|\Gamma_s| \geq \log |\Gamma|$ and the claim follows immediately from \cref{prop:the_moderate_centre} with $H=\Gamma_s$. Indeed, the hypotheses of this proposition are satisfied by \cref{lem:G_s=<comms>}, which implies that $\Gamma_s$ is $r$-quasiconnected for some $r=r(s)$, and \cref{lem:Gamma_s_is_small}, which implies that $|\Gamma_s| \leq c|\Gamma|/\log |\Gamma|$ for some constant $c=c(s,k)$.

\medskip

Now suppose that $|\Gamma_s|<\log |\Gamma|$. In this case, we clearly have that
\[
\diam_S(\Gamma/\Gamma_s) \leq \diam_S(\Gamma) \leq \frac{\lambda|\Gamma|}{(\log |\Gamma|)^s} \leq \frac{\lambda|\Gamma/\Gamma_s|}{(\log |\Gamma|)^{s-1}} \leq \frac{\lambda|\Gamma/\Gamma_s|}{(\log |\Gamma/\Gamma_s|)^{s-1}} .
\]
Letting $G_1 = \Cay(\Gamma/\Gamma_s,S)$, it follows from the induction hypothesis that there exists 
$q_1=q_1(s,k,\lambda,\eps) <1$ such that if $p\geq q_1$ then
\begin{align*}
\P_p^{G_1}
\bigl(x \lra y \bigr) &\geq \sqrt{1-\eps} \qquad \text{for every $x,y \in \Gamma/\Gamma_s$}.
\end{align*}
Let $z \in S$ and $r\geq 1$ be as in the statement of \cref{lem:diameter_induction2} and let $G_2=\Cay(\langle z \rangle \Gamma_s,S_s \cup \{z\})$. Using the conclusions of \cref{lem:diameter_induction2} together with \cref{thm:fin.perc.ab.gen} yields that there exists $q_2=q_2(s,k,\eps)$ such that
\[
\P_{q_2}^{G_2}(x \leftrightarrow y) \geq \sqrt{1-\eps} \qquad \text{ for every $x,y \in \Gamma_s$.}
\]
Letting $G'=\Cay(G,S \cup S_s)$, we deduce from these estimates together with \cref{prop:Benj-Schr} that
\begin{align*}
\P_{q_1}^{G'}
\bigl(x \lra y \Gamma_s \bigr) &\geq 1-\eps \qquad \text{for every $x,y \in \Gamma$ and}
 \\
\P_{q_2}^{G'}
\bigl(x \lra y \bigr) &\geq 1-\eps \qquad \text{for every $x,y \in \Gamma$ belonging to a common coset of $ \Gamma_s$.}
\end{align*}
Letting $q_3<1$ be defined by $1-q_3 = (1-q_1\vee q_2)^2$ so that Bernoulli-$q_3$ percolation on $G'$ has the same law as the union of two independent copies of Bernoulli-$q_1\vee q_2$ percolation, we obtain that if $p\ge q_3$ then
\[
\P_{p}^{G'}(x \lra y) \geq \P_{q_1}^{G'}
\bigl(x \lra y \Gamma_s \bigr) \min\Bigl\{ \P_{q_2}^{G'}
\bigl(y \lra w \bigr) : w \in y \Gamma_s \Bigr\} \geq 1-\eps
\]
for every $x,y\in \Gamma$. Since $\eps>0$ was arbitrary, we may conclude by  applying \cref{cor:changing_generators}. \qedhere

\end{proof}

\subsection{Subgroups of bounded index and virtually nilpotent groups}
\label{subsec:virtuallynilpotent}

In this section we show how to reduce the study of percolation on a group to the study of percolation on a subgroup of bounded index. Our main objective in so doing is to prove the following extension of Proposition \ref{prop:fin.perc.nilp}.

\begin{theorem}\label{thm:fin.perc.nilp}
Let $k,n,s\in\N$ and let $\lambda,\eps>0$. Then there exist $p_0=p_0(k,n,s,\lambda,\eps)<1$ such that if $\Gamma$ is a finite group containing an $s$-step nilpotent subgroup $H$ of index $n$, and if $S$ is a generating set for $\Gamma$ of size at most $k$ satisfying
\[
\diam_S(\Gamma)\le\frac{\lambda|\Gamma|}{(\log|\Gamma|)^s},
\]
then Bernoulli-$p$ bond percolation on $\Cay(\Gamma,S)$ satisfies $\P_p(x\leftrightarrow y) \geq 1-\eps$ for every $p\geq p_0$ and $x,y\in \Gamma$.
\end{theorem}

Another objective is to prove the following quantitative version of \cite[Corollary 7.19]{LP:book}, which is an important ingredient in the proof of Theorem \ref{thm:gap}. Note here that the upper bound on $p_c$ does not depend on the size of the generating set.

\begin{theorem}\label{thm:inf.perc.nilp}
For each $n\geq 1$ there exists $\eps=\eps(n)>0$ such that if $\Gamma$ is an infinite, finitely generated group that is not virtually cyclic and that contains a nilpotent subgroup of index at most $n$, and if $S$ is a finite generating set of $\Gamma$, then $p_c(\Cay(\Gamma,S))\le1-\eps$.
\end{theorem}

In fact, proving these results is fairly straightforward given the results of the previous section and the following standard lemmas.

\begin{lemma}[{\cite[Lemma 11.2.1]{MR3971253}}]\label{lem:coset.reps}
Let $n\in\N$. Suppose $\Gamma$ is a group with a finite generating set $S$ and $H$ is a subgroup of index $n$ in $\Gamma$. Then $\hat S^n$ contains a complete set of coset representatives for $H$ in $\Gamma$.
\end{lemma}

\begin{lemma}[{\cite[Lemma 4.2]{MR3439705}}]\label{lem:fin.ind.gen}
Let $n\in\N$. Suppose $\Gamma$ is a group with a finite generating set $S$, and $H$ is a subgroup of index $n$ in $\Gamma$. Then $H\cap \hat S^{2n-1}$ generates $H$ with diameter at most $\diam_S(\Gamma)$.
\end{lemma}

\begin{lemma}\label{lem:nilp.Z^2.quotient}
Let $\Gamma$ be a nilpotent group that is not virtually cyclic. Then there exists a surjective homomorphism $\Gamma\to\Z^2$.
\end{lemma}

Although this lemma is well known as folklore, we were not able to find a reference and have therefore included a proof for completeness.

\begin{proof}[Proof of \cref{lem:nilp.Z^2.quotient}]
The abelianisation $\Gamma/[\Gamma,\Gamma]$ is a finitely generated abelian group, so $\Gamma/[\Gamma,\Gamma]$ $\cong\Z^d\times T$ for some $d\geq 0$ and some finite abelian group $T$. It suffices to show that $d\ge2$. If, on the contrary, $d\le1$, then writing $\pi:\Gamma\to\Gamma/[\Gamma,\Gamma]$ for the quotient homomorphism there exists a finite subset $S$ of $\Gamma$ such that $\pi(S)$ generates $\Gamma/[\Gamma,\Gamma]$ and contains at most one element of infinite order. It then follows from \cref{lem:comm.linear} that each simple commutator $[x_1,\ldots,x_j]$ with each $x_i\in S$ has finite order modulo $\Gamma_{j+1}$.
The fact that $\pi(S)$ generates $\Gamma/[\Gamma,\Gamma]$ implies that $S$ generates $\Gamma$ \cite[Corollary 10.3.3]{MR0103215}, so \cref{lem:G_s=<comms>} shows that for each $j=2,\ldots,s$ the simple commutators $[x_1,\ldots,x_j]$ with each $x_i\in S$ generate $\Gamma_j$ modulo $\Gamma_{j+1}$. Thus each quotient $\Gamma_j/\Gamma_{j+1}$ with $j=2,\ldots,s$ is an abelian group generated by finitely many elements of finite order. Since $\Gamma_{s+1}=\{\mathrm{id}\}$, it follows that $\Gamma_2=[\Gamma,\Gamma]$ is finite, and hence that $\Gamma$ is finite by cyclic. Finite-by-cyclic groups are well known to be virtually cyclic: indeed, if $\Gamma/H\cong\Z$ for some $H\lhd\Gamma$ and if $xH$ is a generator of $G/H$ then $H$ is a complete set of coset representatives for the cyclic subgroup $\langle x\rangle$ of $\Gamma$. This completes the proof.
\end{proof}

\begin{proof}[Proof of Theorem \ref{thm:fin.perc.nilp}]
Let $U=H\cap \hat S^{2n-1}$, so that 
\[
\diam_U(H)\le\frac{\lambda n|H|}{(\log|H|)^s}
\]
by \cref{lem:fin.ind.gen}. Applying \cref{prop:fin.perc.nilp} to $G'=\Cay(H,U)$ yields that 
for every $\eps>0$ there exists $p_0=p_0(k,n,s,\lambda,\eps)<1$ such that $\P_p^{G'}(h_1\leftrightarrow h_2) \geq 1-\eps$ for every $p\geq p_0$ and $h_1,h_2 \in H$.
Since $G'$ is a subgraph of $G''=\Cay(\Gamma,\hat S^{2n-1})$, it follows that $\P_p^{G''}(h_1\leftrightarrow h_2) \geq 1-\eps$ for every $p\geq p_0$ and $h_1,h_2\in H$ also. Now, we have by \cref{lem:coset.reps} that for every $\gamma \in \Gamma$ there exists $h\in H$ such that $h$ and $\gamma$ have distance at most $1$ in $\Cay(\Gamma,\hat S^{2n-1})$ and hence that $\P_p^{G''}(h\leftrightarrow \gamma)\ge p$ for every $p\in[0,1]$. It follows by Harris-FKG that
$\P_p^{G''}(x \leftrightarrow y) \geq p^2(1-\eps)$
for every $p \geq p_0$ and $x,y\in \Gamma$, and the result follows easily from \cref{cor:changing_generators}.
\end{proof}

\begin{proof}[Proof of Theorem \ref{thm:inf.perc.nilp}]
Let $H$ be a nilpotent subgroup of index at most $n$ in $\Gamma$ that is not virtually cyclic. \cref{lem:nilp.Z^2.quotient} implies that there exists a surjective homomorphism $\pi:H\to\Z^2$, and \cref{lem:fin.ind.gen} then implies that $\pi(H\cap\hat{S}^{2n-1})$ is a generating set for $\Z^2$, and in particular that there exist elements $x_1,x_2\in H\cap\hat{S}^{2n-1}$ such that $\pi(x_1)$ and $\pi(x_2)$ are linearly independent. Since $x_1,x_2\in\hat{S}^{2n-1}$ there exists a subset $S_0\subset\hat{S}$ of size at most $4n-2$ such that $x_1,x_2\in S_0^{2n-1}$. Let $\Gamma_0=\langle S_0\rangle$ and $H_0=\langle H\cap S_0^{2n-1}\rangle$. Since $\pi(x_1)$ and $\pi(x_2)$ are linearly independent they generate a subgraph of $\Cay(\pi(H_0),\pi(H\cap S_0^{2n-1}))$ that is isomorphic to the two-dimensional square lattice. This lattice has $p_c=\frac12$ by a famous theorem of Kesten \cite{MR575895}, so that $p_c(\Cay(H_0,H\cap S_0^{2n-1}))\le\frac12$ by \cref{prop:Benj-Schr}. Since $H_0<\Gamma_0$ we may consider the inclusion map $H_0\to\Gamma_0$, which induces a $(2n-1,1)$-rough embedding $\Cay(H_0,H\cap S_0^{2n-1}))\to\Cay(\Gamma_0,S_0)$. \cref{lem:roughembedding} therefore implies that there exists $\eps=\eps(n)$ such that $p_c(\Cay(\Gamma_0,S_0))\le1-\eps$. This proves the theorem, since $\Cay(\Gamma_0,S_0)$ is a subgraph of $\Cay(\Gamma,S)$.
\end{proof}

\begin{remark}
Instead of invoking Kesten's theorem, one could instead use the easy bound $p_c(\Z^2) \leq 2/3$ to deduce a result of the same form.
\end{remark}

\section{Percolation from isoperimetry}
\label{sec:Isoperimetry}

The goal of this section is to prove Theorems \ref{thm:iso.inf}, \ref{thm:pu} and \ref{thm:uniquenesscriterion}. 
The section is organised as follows:
In \cref{subsec:DGRSY} we explain how the methods of Duminil-Copin, Goswami, Raoufi, Severo, and Yadin \cite{1806.07733} can be adapted to the setting of finite graphs, then apply the resulting theorems to prove \cref{thm:iso.inf} in \cref{subsec:largetosmall}. Finally, in \cref{subsec:pu} we prove \cref{thm:uniquenesscriterion} and then deduce \cref{thm:pu} from Theorems \ref{thm:iso.inf} and \ref{thm:uniquenesscriterion}.

\subsection{Connecting large sets via the Gaussian free field}
\label{subsec:DGRSY}

In this section we discuss those results that can be obtained by a direct application of the methods of \cite{1806.07733} to finite graphs.
Before stating these results, we first introduce some relevant definitions. Let $G=(V,E)$ be a finite, connected graph and let $B \subseteq V$ be a distinguished set of \emph{boundary vertices}. We say that the pair $(G,B)$ satisfies a $d$-dimensional isoperimetric inequality with constant $c$, abbreviated \eqref{Assumption:IDB}, if
\begin{equation}
\label{Assumption:IDB}
\tag{$\mathrm{ID}(B)_{d,c}$} |\partial_E K| \geq c |K|^{(d-1)/d} \qquad \text{ for every $K \subseteq V \setminus B$.}
\end{equation}
Note that, in contrast to our earlier definition \eqref{Assumption:ID}, we now require that \emph{every} large subset of $V\setminus B$ has large boundary, not just those with at most half the total volume of $V$. In particular, $B$ must itself be large for \eqref{Assumption:IDB} to hold.  For each two non-empty, disjoint sets of vertices $A$ and $B$ in the finite connected graph $G$, we write $\Ceff(A \leftrightarrow B)$ for the \textbf{effective conductance} between $A$ and $B$, which is defined by
\[
\Ceff(A \leftrightarrow B) = \sum_{a\in A} \deg(a) \bP_a(\tau_B < \tau^+_A) = \sum_{b\in B} \deg(b) \bP_b(\tau_A < \tau^+_B)
\]
where $\bP_v$ is the law of a simple random walk started at $v$ and $\tau_A$ and $\tau^+_A$ denote the first time and first positive time that the walk visits $A$ respectively. It is a theorem of Lyons, Morris, and Schramm~\cite{LMS08} that if $G=(V,E)$ satisfies a $d$-dimensional isoperimetric inequality \eqref{Assumption:ID} then there exists a positive constant $c'=c'(d,c)$ such that
\begin{equation}
\label{eq:LMS}
\Ceff(A \leftrightarrow B) \geq c' \min\{|A|,|B|\}^{(d-2)/d}
\end{equation}
for every two disjoint non-empty sets $A,B \subseteq V$. See e.g.\ \cite{LP:book} for further background on effective conductances.

\medskip

The following theorem follows by an essentially identical proof to that of \cite[Theorem 1.2]{1806.07733}.

\begin{theorem}[Connecting to large sets]
\label{prop:connecting_big_sets}
Let $G=(V,E)$ be a finite, connected graph with degrees bounded by $k$ and let $B \subseteq V$ be such that $(G,B)$ satisfies a $d$-dimensional isoperimetric inequality \eqref{Assumption:IDB} for some $d>4$ and $c>0$. Then there exists $p_0=p_0(d,c,k)<1$ such that
\begin{equation}
\label{eq:connecting_big_sets}
\P_p(A \leftrightarrow B) \geq 1-\exp\Bigl[ - \frac{1}{2} \Ceff(A\leftrightarrow B) \Bigr]
\end{equation}
for every $p_0 \leq p \leq 1$ and every non-empty set $A \subseteq V$.
\end{theorem}

Rather than reproduce the entire proof of \cite[Theorem 1.2]{1806.07733}, which would take rather a lot of space, we instead give a brief summary of the main ideas of that proof with particular emphasis given to the (very minor) changes needed to prove \cref{prop:connecting_big_sets}.

\begin{remark}\label{rem:DGRSY}
Note that if $G=(V,E)$ is an \emph{infinite} connected, locally finite graph satisfying \eqref{Assumption:ID}, $\Lambda \subseteq V$ is a finite set of vertices, $G_\Lambda$ is the subgraph of $G$ induced by $\Lambda$, and $\partial^-_V \Lambda$ is the internal vertex boundary of $\Lambda$ then $(G_\Lambda,\partial^-_V \Lambda)$ satisfies $(\mathrm{ID}(\partial^-_V \Lambda)_{d,c})$. Taking an exhaustion of $G$ by finite sets, this allows one to recover \cref{thm:quantitative_DGRSY} from \cref{prop:connecting_big_sets} and the inequality \eqref{eq:LMS}.
\end{remark}

The proof relies crucially on the relationships between Bernoulli bond percolation and the \emph{Gaussian free field}. The basic idea of the proof is that, by using the Gaussian free field, we can construct a \emph{percolation in random environment} model that is easier to prove has a phase transition than for standard Bernoulli percolation. The main technical step of the proof of \cite{1806.07733} then shows that, under a suitable isoperimetric assumption, we can `integrate out' the randomness of the environment and compare the new model to standard Bernoulli percolation of sufficiently high retention probability.

\medskip

We now recall the relevant definitions.
Let $G=(V,E)$ be a finite, connected graph and let $B$ be a non-empty subset of $V$. 
Let $P_B : V^2 \to \R$ be the transition matrix of a random walk on $G$ that is killed when it first visits $B$, so that 
\[P_B(u,v) = \frac{\text{number of edges between $u$ and $v$}}{\deg(u)} \mathbbm{1}(u,v \notin B)\]
for every $u,v\in V$. The \textbf{Green function} $\bG_B : V^2 \to \R$ is defined by
\[
\bG_B(u,v) = \frac{\mathbbm{1}(u,v\in V \setminus B)}{\deg(v)}\sum_{n\geq 0} P_B^n(u,v)
\]
for every $u,v\in V$, so that $\deg(v)\bG_B(u,v)$ is the expected number of times a random walk started at $u$ visits $v$ before first hitting $B$. (Note that the normalization by the degree is not always included in the definition of the Green function, but is convenient for our applications here as it makes the Green function symmetric.) 
The \textbf{Gaussian free field} (GFF) on $G$ with Dirichlet boundary conditions on $B$ is the mean-zero Gaussian random vector $\varphi = (\varphi_v)_{v\in V}$ with covariances given by the Green function. That is,  $\varphi$ is a Gaussian random vector with
\[
\E_B^\mathrm{GFF} \left[\varphi_u\right] = 0 \qquad \text{ and } \qquad \E_B^\mathrm{GFF} \left[\varphi_u\varphi_v\right] = \bG_B(u,v)
\]
for every $u,v\in V$, where we write $\E_B^\mathrm{GFF}$ for expectations taken with respect to the law of $\varphi$. Equivalently, the law $\P_B^\mathrm{GFF}$ of the GFF on $G$ with Dirichlet boundary conditions on $B$ can be defined as the measure on $\R^{V}$ that is supported on $\{\psi \in \R^V : \psi(b)=0 $ for every $b \in B\} \cong \R^{V \setminus B}$ and has density with respect to Lebesgue measure $\operatorname{Leb}$ on $\R^{V\setminus B}$ given by
\begin{equation}
\label{eq:GFF_Hamiltonian}
\frac{\dif\P_B^\mathrm{GFF}}{\dif\operatorname{Leb}}(\varphi) = \frac{1}{Z(G,B)} \exp\left[-\frac{1}{2} \sum_{e\in E^\rightarrow} \bigl|\varphi\bigl(e^+\bigr)-\varphi\bigl(e^-\bigr)\bigr|^2 \right],
\end{equation}
 where $E^\rightarrow$ denotes the set of oriented edges of the graph $G$ and $Z(G,B)$ is a normalizing constant. (Although our graphs are not oriented, we can still think of each edge as having exactly two possible orientations.)

\medskip

 Let $\lambda >0$. It follows from the representation \eqref{eq:GFF_Hamiltonian} that if we condition on the absolute values $(|\varphi_v+\lambda|)_{v\in V}$, then the signs $(\sign(\varphi_v+\lambda))_{v\in V}$ are distributed as an \emph{Ising model} on $G$ with plus boundary conditions on $B$ and with coupling constants given by $J(e)=|\varphi_x+\lambda||\varphi_y+\lambda|$ for each edge $e$ with endpoints $x$ and $y$. (The discussion here serves only as background; we will not use the Ising model directly in this paper.) Using this fact with $\lambda=1$ together with the relationship between the Ising model and Bernoulli percolation given by the Edwards-Sokal coupling led the authors of \cite{1806.07733} to derive an important estimate concerning connection probabilities in a certain \emph{percolation in random environment} model derived from the GFF.
Before stating this estimate, let us first give the relevant definitions.
Given a vector of probabilities $\bp=(\bp_e)_{e\in E}$, $\P_\bp$ denotes the law of the inhomogeneous Bernoulli bond percolation process in which each edge $e$ is either deleted or retained independently at random with retention probability $\bp_e$.
 Finally, given a real number $x$ we write $x_+= \max \{x,0\}$. The following proposition follows by the same proof as \cite[Proposition 2.1]{1806.07733}. (Indeed, the proof of that proposition implicitly establishes the proposition as stated here and then applies a limiting argument to deduce a similar claim for connections to infinity in infinite graphs.)

\begin{proposition}
\label{prop:GFFpercestimate}
Let $G=(V,E)$ be a finite, connected graph and let $A,B$ be non-empty disjoint subsets of $V$. Then
\begin{equation}
\label{eq:GFFpercestimate}
\E_B^\mathrm{GFF}\left[\P_{\mathbf{p}(\varphi)}(A \leftrightarrow B) \right] \geq 1- \exp \left[ - \frac{1}{2} \Ceff(A \leftrightarrow B) \right]
\end{equation}
where $\bp(\varphi)_e:=1-\exp\left[-2(\varphi_x+1)_+(\varphi_y+1)_+ \right]$ for each edge $e$ with endpoints $x$ and $y$. 
\end{proposition}

\begin{proof}[Proof of \cref{prop:GFFpercestimate}]
For each $t\in \R^A$ let 
$X^t_A(\varphi)=\exp\left[-\sum_{x\in A}t_x (\varphi_x+1)\right]$.
Equation (2.5) of \cite{1806.07733} states in our notation that
\[
\E_B^\mathrm{GFF}\left[\P_{\mathbf{p}(\varphi)}(A \nleftrightarrow B) \right] \leq \E_B^\mathrm{GFF}\left[X^t_A(\varphi)\right] 
\]
for every $t\in \R^A$. We have by the definitions that $\sum_{x\in A}t_x (\varphi_x+1)$ is a Gaussian with mean $\sum_{x\in A} t_x$ and variance $\sum_{x,y\in A} t_xt_y \bG_B(x,y)$, so that
\[
\E_B^\mathrm{GFF}\left[X^t_A(\varphi)\right] = \exp\left[ -\sum_{x\in A} t_x + \frac{1}{2}\sum_{x,y\in A} t_xt_y \bG_B(x,y)\right].
\]
Taking $t_x = \deg(x) \mathbf{P}_x(\tau_B<\tau_A^+)$ for each $x \in A$ yields the claimed bound since we have by a standard calculation that $\sum_{x\in A} t_x = \Ceff(A \leftrightarrow B)$ and 
\begin{align*}
 \sum_{y\in A}t_y \bG_B(x,y) &= \sum_{y\in A}\sum_{n\geq 0} P_B^n(x,y) \mathbf{P}_y(\tau_B<\tau_A^+)
  \\&= \sum_{y\in A}\bP_x(\text{the walk visits $A$ for the last time at $y$ before hitting $B$})=1
\end{align*}
for every $x\in A$.
\end{proof}

 \cref{prop:GFFpercestimate} establishes the same estimate as \cref{prop:connecting_big_sets} but applying to the percolation in random environment model associated to the GFF rather than to our original Bernoulli percolation model. An important and technical part of the argument of \cite{1806.07733} shows that, in high-dimensional graphs, the two models can be compared in such a way that \cref{prop:GFFpercestimate} implies \cref{prop:connecting_big_sets}.

\begin{proposition}
\label{prop:GFF_comparison}
Let $G=(V,E)$ be a finite, connected graph with degrees bounded by $k$ and let $B \subseteq V$ be such that $(G,B)$ satisfies a $d$-dimensional isoperimetric inequality \eqref{Assumption:IDB} for some $d>4$ and $c>0$. Then there exists $p_0=p_0(d,c,k)<1$ such that
\begin{equation}
\P_p(A \leftrightarrow B) \geq \E_B^\mathrm{GFF}\left[\P_{\mathbf{p}(\varphi)}(A \leftrightarrow B) \right]
\end{equation}
for every $p_0 \leq p \leq 1$ and every non-empty set $A \subseteq V$.
\end{proposition}

Note that the proof does \emph{not} establish a stochastic domination relation between the two models, and indeed such a relation does not hold. 

\begin{proof}[Proof of \cref{prop:GFF_comparison}]
 We first apply the classical relationship between isoperimetric inequalities and return probability bounds (see e.g.\ \cite[Theorem 3.2.7]{KumagaiBook} or \cite[Corollary 6.32]{LP:book}) to deduce from \eqref{Assumption:IDB} that there exists a constant $C=C(d,c,k)$ such that
\[
P_B^n(u,v) \leq C n^{-d/2} 
\]
for every $u,v\in V$. This bound is of the same form as the hypothesis $(\mathrm{H}_d)$ of \cite{1806.07733}, and given this bound the proof proceeds exactly as that of \cite[Proposition 3.2]{1806.07733};
one need only check that the value of $p_0$ given by that proof depends only on $d,c,$ and $k$. Let us now give a brief indication of how this can be done. That proof gives a value for $p_0$ of the form $1-(1-q)e^{-h}$, where $q\in[\frac12,1)$ and $h>0$ both depend on two other quantities $\alpha>0$ and $n_0\in\N$, and in principle on the graph $G$. The quantity $q$ is computed explicitly from $n_0$ in the proof of Proposition 3.2, with no dependence on the graph. The quantity $h$ is computed in Lemma 3.5, and depends only on $n_0$ and the maximum degree of the graph since all the constants arising in the quoted theorem of Liggett, Schonmann and Stacey \cite{MR1428500} depend only on the maximum degree. The value of $\alpha$ and a preliminary value for $n_0$ are given by Lemma 3.6; they depend on several constants appearing in the conclusion of the proof of that lemma (pages 21 and 22 of the published version), each of which can easily be checked to depend only on the parameters $d,c$, and $k$ (indeed, these constants are introduced only to simplify a completely explicit expression involving the maximum degree and sums of return probabilities). The integer $n_0$ is then possibly increased at the start of the proof of Proposition 3.2, but only to ensure that it is larger than some constant depending only on $\alpha$.
\end{proof}

\begin{proof}[Proof of \cref{prop:connecting_big_sets}]
This follows immediately from \cref{prop:GFFpercestimate,prop:GFF_comparison}.
\end{proof}

 \subsection{From large sets to small sets}
\label{subsec:largetosmall}

We now apply \cref{prop:connecting_big_sets} to complete the proof of \cref{thm:iso.inf}. We first establish the following crude relationship between the two kinds of isoperimetric inequality we consider.

\begin{lemma}
\label{lem:isoperimetry_twodefinitions}
Let $G=(V,E)$ be a finite, connected graph satisfying the isoperimetric inequality \eqref{Assumption:ID} for some $d>1$ and $c>0$. If $B\subseteq V$ has $|B| \geq \eps |V|^{1-\delta}$ for some $0< \eps,\delta \leq 1$, then $(G,B)$ satisfies the isoperimetric inequality $(\mathrm{ID}(B)_{d',c'})$ with 
\[d'=\frac{d}{1+\delta(d-1)} \qquad \text{ and } \qquad c'=c \eps^{(d-1)/d}.\]
\end{lemma}

\begin{proof}
Let $K \subseteq V \setminus B$. We need to prove that $|\partial_E K|\geq c'|K|^{(d'-1)/d'}$. The claim follows trivially from \eqref{Assumption:ID} if $|K|\leq |V|/2$. Otherwise we have that
\begin{multline*}|\partial_E K| \geq c|V \setminus K|^{(d-1)/d} \geq c |B|^{(d-1)/d} \geq c\eps^{(d-1)/d}|V|^{(1-\delta)(d-1)/d} \\ \geq c \eps^{(d-1)/d} |K|^{(1-\delta)(d-1)/d}=c \eps^{(d-1)/d}|K|^{(d'-1)/d}
\end{multline*}
as claimed.
\end{proof}

Putting together \cref{prop:connecting_big_sets} and \cref{lem:isoperimetry_twodefinitions} yields the following immediate corollary.

\begin{corollary}
\label{cor:connecting_big_sets}
Let $G=(V,E)$ be a finite, connected graph with degrees bounded by $k$ satisfying  a $d$-dimensional isoperimetric inequality \eqref{Assumption:ID} for some $d>4$ and $c>0$. Then for each $\delta<(d-4)/4(d-1)$ and $\eps>0$ there exists $p_0(\delta,\eps)=p_0(\delta,\eps,d,c,k)<1$  such that
\begin{equation}
\P_p(A \leftrightarrow B) \geq 1-\exp\left[ -\frac{1}{2}\Ceff(A \leftrightarrow B) \right]
\end{equation}
for every $p_0(\delta,\eps) \leq p \leq 1$ and every two non-empty sets $A,B \subseteq V$ with $|B|\geq \eps |V|^{1-\delta}$.
\end{corollary}

We now apply \cref{cor:connecting_big_sets} to prove \cref{thm:iso.inf}. To do this, we will apply \cref{prop:connecting_big_sets} with $B$ taken to be a \emph{random} set given by a so-called \emph{ghost field}, which we now introduce. For each $p\in [0,1]$ and $h>0$ we write $\P_{p,h}$ for the joint distribution of the Bernoulli-$p$ bond percolation configuration $\omega$ and an independent \textbf{ghost field} $\cG$ of intensity $h$, that is, a random subset of $V$ in which each vertex is included independently at random with probability $1-e^{-h}$. Note that the ghost field $\cG$ has the property that 
\[
\P_{p,h}(A \cap \cG = \varnothing) = e^{-h|A|} 
\]
for every $A \subseteq V$. Moreover, it follows from a standard Chernoff bound calculation for sums of Bernoulli random variables \cite[Theorem 4.5]{mitzenmacher2017probability} that there exists a universal constant $a \in(0,1)$ such that if $h \leq 1$ then $|\cG|$ satisfies the lower tail estimate
\begin{equation}
\label{eq:ghost_Chernoff}
\P_{p,h}\Bigl(|\cG| \leq \frac{1}{2} h |V|\Bigr) \leq \P_{p,h}\Bigl(|\cG| \leq \frac{e}{2e-2} (1-e^{-h})|V|\Bigr) \leq \exp\left[ -a h |V| \right].
\end{equation}
(Indeed, one may take $a=(e-2)^2/8(e-1)^2$.)

\begin{proof}[Proof of \cref{thm:iso.inf}]
Let $d > 6+2\sqrt{7}$ and write $\alpha = (d-2)/d$. Let $a \in(0,1)$ be the universal constant from \eqref{eq:ghost_Chernoff}. By the results of Lyons, Morris, and Schramm \cite{LMS08} stated in \eqref{eq:LMS}, there exists a positive constant
$\eta=\eta(d,c,k) \le a$ such that $\Ceff(A\leftrightarrow B) \geq \eta \min\{|A|,|B|\}^{\alpha}$ for every two disjoint non-empty sets $A,B \subseteq V$. The condition $d>6+2\sqrt{7}$ allows us to take $\delta$ such that 
\[\frac{2}{d} < \delta < \frac{d-4}{4(d-1)}.\]
Fix one such choice of $\delta$ and let
$p_0=p_0(\delta,\eta/16,d,c,k)<1$ be as in \cref{cor:connecting_big_sets} so that
\begin{equation}
\label{eq:corconnectingrestatement}
\P_{p_0}(A \nleftrightarrow B) \leq \exp\left[ -\frac{\eta}{2} \min\{|A|,|B|\}^\alpha \right]
\end{equation}
for every two non-empty sets $A,B \subseteq V$ with $\max\{|A|,|B|\} \geq \frac{\eta}{16}|V|^{1-\delta}$. Since small values of $|V|$ can be handled by increasing $p$, we may assume throughout the rest of the proof that $|V|\ge(8/\eta)^{1/(1-\delta)}$ so that $\frac{\eta}{8}|V|^{-\delta}\ge|V|^{-1}$.

For each $A \subseteq V$, let $K_A = \bigcup_{v\in A} K_v$ be the union of all clusters intersecting $A$.
 We first apply \cref{cor:connecting_big_sets} with one of the sets \emph{equal to the ghost field}  $\cG$ to show that $K_A$ is much larger than $|A|$ with high probability when $p\geq p_0$ and $|A|$ is either large or small in a certain sense. 
 Indeed, \cref{cor:connecting_big_sets} implies immediately that if $|V|^{-1} \leq h \leq 1$ and $A \subseteq V$ are such that either $|A| \geq \frac{\eta}{16}|V|^{1-\delta}$ or $h\geq \frac{\eta}{8} |V|^{-\delta}$ then we have by \eqref{eq:ghost_Chernoff} and the choice of $\eta$ that
\begin{align}
\P_{p_0,h}(A \nleftrightarrow \cG) &\leq \P_{p_0,h}\left(|\cG| \leq \frac{1}{2} h|V|\right)+\exp\left[-\frac{\eta}{4} \min\{|A|,h |V|\}^{\alpha}\right]
\nonumber
\\ &\leq \exp\left[ -a h |V| \right] +  \exp\left[-\frac{\eta}{4} \min\{|A|,h |V|\}^{\alpha}\right] \leq 2 \exp\left[-\frac{\eta}{4} \min\{|A|,h |V|\}^{\alpha}\right].
\label{eq:ghost1}
\end{align}
On the other hand, we have by definition of the ghost field that
\begin{equation}
\label{eq:ghost2}
\P_{p_0,h}(A \nleftrightarrow \cG \mid K_A) = e^{-h|K_A|}, \qquad \text{ so that } \qquad \E_{p_0} \left[e^{-h|K_A|}\right] = \P_{p_0,h}(A \nleftrightarrow \cG)
\end{equation}
for every $h>0$ and $A \subseteq V$. It follows by \eqref{eq:ghost1}, \eqref{eq:ghost2} and Markov's inequality that
\begin{align*}
\P_{p_0}\left(|K_A| \leq \frac{\eta}{8h} \min\{|A|,h |V|\}^{\alpha}\right) &=
\P_{p_0}\left(e^{-h|K_A|} \geq \exp\left[-\frac{\eta}{8} \min\{|A|,h |V|\}^{\alpha}\right]\right)
\\
&\leq 
\exp\left[\frac{\eta}{8} \min\{|A|,h |V|\}^{\alpha}\right]
\E_{p_0} \left[e^{-h|K_A|}\right]\nonumber
\\
&\leq 2 \exp\left[-\frac{\eta}{8} \min\{|A|,h |V|\}^{\alpha}\right]
\end{align*}
for every $|V|^{-1} \leq h \leq 1$ and $A \subseteq V$ such that either $|A| \geq \frac{\eta}{16}|V|^{1-\delta}$ or $h\geq \frac{\eta}{8}|V|^{-\delta}$.
Since $\frac{\eta}{8}|V|^{-\delta} \geq |V|^{-1}$ by assumption, we may apply the previous inequality with $h=\frac{\eta}{8}|V|^{-\delta}$ to obtain that
\begin{equation}\label{eq:K_A<V^delta}
\P_{p_0}\left(|K_A| \leq |V|^{\delta} \min\left\{|A|,\frac{\eta}{8}|V|^{1-\delta}\right\}^{\alpha}\right) \leq 2 \exp\left[-\frac{\eta}{8} \min\left\{|A|,\frac{\eta}{8}|V|^{1-\delta}\right\}^{\alpha}\right]
\end{equation}
for every $A \subseteq V$.

Consider the function $\varphi:[0,\frac{\eta}{8}|V|^{1-\delta}]\to[0,\frac{\eta}{8}|V|^{1-\delta}]$ defined by
\[
\varphi(x) = \min\left\{|V|^\delta x^\alpha,\frac{\eta}{8}|V|^{1-\delta}\right\}.
\]
 It follows from \eqref{eq:K_A<V^delta} that
\[
\P_{p_0}\left(|K_A| \le \ph(|A|)\right) \leq 2 \exp\left[-\frac{\eta}{8}|A|^\alpha\right]
\]
for every $A\subseteq V$ with $|A|\le\frac{\eta}{8}|V|^{1-\delta}$. Now for each $i\geq 1$, let $p_i$ be defined by $1-p_i= (1-p_0)^{i+1}$,  so that Bernoulli-$p_i$ percolation has the same distribution as the union of $i+1$ independent copies of Bernoulli-$p_0$ percolation. This relationship immediately yields that the recursive inequality
\[
\P_{p_{i+1}}(|K_A| \leq n) \leq \P_{p_i}(|K_A| \leq m) + \max\bigl\{\P_{p_0}(|K_{A'}| \leq n) : A'\subseteq V,\, |A'| > m \bigr\}
\]
holds for every $i\geq 0$, $A \subseteq V$ and $n,m \geq 1$, which combines with \eqref{eq:K_A<V^delta} and the fact that $x\le\ph^i(x)\le\frac{\eta}{8}|V|^{1-\delta}$ for every $i\ge0$ and $x\in[0,\frac{\eta}{8}|V|^{1-\delta}]$ to imply that
\begin{align*}
\P_{p_i}\left(|K_A| \leq \varphi^{i+1}(|A|)\right)
&\leq
\P_{p_{i-1}}\left(|K_A| \leq \varphi^i(|A|)\right) + 2 \exp\left[-\frac{\eta}{8}\varphi^i(|A|)^\alpha\right],
\end{align*}
and we deduce by induction on $i$ that
\begin{equation}
\P_{p_i}\left(|K_A| \leq \varphi^{i+1}(|A|)\right)\le2 \sum_{j=0}^{i-1}\exp\left[-\frac{\eta}{8}\ph^j(|A|)^\alpha\right]
\end{equation}
for every $i\geq 1$ and $A \subseteq V$ with $|A|\le\frac{\eta}{8}|V|^{1-\delta}$.
Let $\eps=\delta-(1-\alpha)(1-\delta)$, which is positive since $\delta > 2/d$. Observe that
\begin{equation}\label{eq:ph.increases.uniformly}
|V|^\delta x^\alpha \geq  |V|^{\delta-(1-\alpha)(1-\delta)} x =  |V|^\eps x \qquad \text{for every $0 \leq x \leq|V|^{1-\delta}$,}
\end{equation}
  so that $\ph$ is increasing and moreover that
 \begin{equation} \varphi^i(x) \geq \min\Bigl\{|V|^{\eps i},\frac{\eta}{8}|V|^{1-\delta}\Bigr\} \qquad \text{for every $i\geq 0$ and $1 \leq x \leq|V|^{1-\delta}$}.
 \end{equation}
  It follows that there exists a constant $i_0=i_0(d,c,k)$ such that $\ph^{i_0+1}(x) =\frac{\eta}{8} |V|^{1-\delta}$ for every $x\in[1,\frac{\eta}{8}|V|^{1-\delta}]$, and hence that there exists a constant $C_1=2 i_0$ such that
\[
\P_{p_{i_0}}\left(|K_A| \leq \frac{\eta}{8}|V|^{1-\delta}\right) \leq 2 \sum_{j=0}^{i_0-1}\exp\left[-\frac{\eta}{8}\ph^j(|A|)^\alpha\right] \leq C_1 \exp\left[-\frac{\eta}{8}|A|^\alpha\right]
\]
for every non-empty set $A \subseteq V$ with $|A|\leq \frac{\eta}{8}|V|^{1-\delta}$, and hence for every non-empty set $A \subseteq V$ since the inequality holds vacuously in the case $|A|> \frac{\eta}{8}|V|^{1-\delta}$. Considering again the coupling of $p_{i_0+1}$ percolation with $p_{i_0}$ and $p_0$ percolation, we deduce that if $A,B \subseteq V$ are non-empty then
\begin{align}
&\P_{p_{i_0+1}}(A \nleftrightarrow B) \nonumber\\
&\hspace{0.05cm}\leq \P_{p_{i_0}}\left(|K_A| \leq \frac{\eta}{8}|V|^{1-\delta}\right) + \P_{p_{i_0}}\left(|K_B| \leq \frac{\eta}{8}|V|^{1-\delta}\right) + \max\left\{\P_{p_0}(A' \nleftrightarrow B') : |A'|,|B'| \geq\frac{\eta}{8} |V|^{1-\delta}\right\}
\nonumber\\
&\hspace{0.05cm}\leq C_1 \exp\left[-\frac{\eta}{8} |A|^\alpha\right] + C_1 \exp\left[-\frac{\eta}{8} |B|^\alpha\right] + \exp\left[-\frac{\eta^2}{16} |V|^{(1-\delta)\alpha}\right]
\nonumber\\
&\hspace{0.05cm}\leq C_2 \exp\left[-\frac{\eta}{8} \min\{|A|,|B|\}^\alpha\right]\label{eq:AtoBfinal}
\end{align}
for some constant $C_2=C_2(d,c,k)=(2C_1+1)$, where we used \eqref{eq:corconnectingrestatement} to bound the third term in the second line. 


To complete the proof, it remains to show that the constant prefactor $C_2$ in \eqref{eq:AtoBfinal} can be made arbitrarily small by increasing $p$ in a uniform manner. This is particularly important in the case that $A$ and $B$ are singletons, in which case the right hand side of \eqref{eq:AtoBfinal} might be larger than $1$, rendering the inequality useless. 
Let $B(v,r)$ denote the graph-distance ball of radius $r$ around $v$ for every $v\in V$ and $r\geq 1$.
Since $|B(v,r)| \geq r+1$ for every $v\in V$ and $r\leq \diam(G)$, it follows from \eqref{eq:AtoBfinal} that there exists a constant $r_0=r_0(d,c,k)$ such that
\[
\P_{p_{i_0+1}}\left(B(u,r_0) \leftrightarrow B(v,r_0) \right) \geq \frac{1}{2}
\]
for every $u,v\in V$. Applying the Harris-FKG inequality, it follows that there exists a positive constant $c_1=c_1(d,c,k)$ such that
\begin{multline}
\P_{p_{i_0+1}}(u \leftrightarrow v) \geq \P_{p_{i_0+1}}\Bigl(B(u,r_0) \leftrightarrow B(v,r_0),\; B(u,r_0) \subseteq K_u \text{ and } B(v,r_0) \subseteq K_v \Bigr) \\\geq \frac{1}{2}\left(p_{i_0+1}^{k^{r_0+1}}\right)^2 \geq c_1
\label{eq:pttoptlowerc1}
\end{multline}
for every $v\in V$, where in the second line we used that $B(v,r)$ is incident to at most $k^{r+1}$ edges. It follows from \eqref{eq:AtoBfinal} and \eqref{eq:pttoptlowerc1} by calculus that there exists a positive constant $\eta'=\eta'(d,c,k)$ such that
\begin{align*}
\P_{p_{i_0+1}}(A \nleftrightarrow B) &\leq \min\left\{(1-c_1), C_2 \exp\left[-\frac{\eta}{8} \min\{|A|,|B|\}^\alpha\right]\right\}\leq \exp\left[-\eta' \min\{|A|,|B|\}^\alpha\right]
\end{align*}
for every two non-empty sets $A,B \subseteq V$.
The claim as stated in the theorem, in which we can increase $p$ in a uniform way to introduce an arbitrarily small constant prefactor, follows easily by calculus from this together with the  fact \cite[Theorem 2.38]{grimmett2010percolation} that $\P_{p^\theta}(A) \geq \P(A)^\theta$ for every $p,\theta \in (0,1)$ and every increasing event $A$. \qedhere
\end{proof}

\subsection{Consequences for the uniqueness threshold}
\label{subsec:pu}

In this section we prove our results concerning the uniqueness threshold. More specifically, we first prove \cref{thm:uniquenesscriterion}, then deduce \cref{thm:pu} from \cref{thm:iso.inf,thm:uniquenesscriterion}. The results of this section are not used in the proof of our main theorem, \cref{thm:fin.perc}, and the reader may safely skip this section if they are interested only in that result.

The proof of \cref{thm:uniquenesscriterion} will apply the following theorem of Schonmann~\cite[Theorem 3.1 and Corollaries 3.2 and 3.3]{MR1676831}.

\begin{theorem}[Schonmann]
\label{thm:Schonmann}
Let $G=(V,E)$ be an infinite, connected, bounded degree graph, let $0 < p_0 < 1$, and suppose that 
\[
\lim_{n\to \infty} \inf_x \P_{p_0}(B(x,n) \nleftrightarrow \infty) = 0.
\]
Then $\P_p(x\leftrightarrow \infty)$ is continuous on $[p_0,1]$ for every $x\in V$. Moreover, if $p_0\leq p_1 \leq 1$ is such that there is a unique infinite cluster $\P_{p_1}$-almost surely, then there is a unique infinite cluster $\P_{p_2}$-almost surely for every $p_1 \leq p_2 \leq 1$.
\end{theorem}

\begin{remark}
We believe that the bounded degree assumption is not really necessary for this result to hold, and hence should not be necessary for \cref{thm:uniquenesscriterion} to hold either.
\end{remark}

The proof of \cref{thm:uniquenesscriterion} will also apply \cite[Theorem 2.45]{grimmett2010percolation}, which states that if $\sA \subseteq\{0,1\}^E$ is an increasing event and $I_r(\mathscr{A})$ denotes the event that $\sA$ holds in any configuration obtained from $\omega$ by deleting at most $r$ edges then
\begin{equation}
\label{eq:interiorprobability}
\P_{p_2}\Bigl(I_r(\mathscr{A})\Bigr) \geq 1 - \left(\frac{p_2}{p_2-p_1}\right)^{r} \P_{p_1}(\mathscr{A}^c)
\end{equation}
for every $0 \leq p_1 < p_2 \leq 1$ and $r\geq 1$. That is, if $\mathscr{A}$ holds with high probability at $p_1$ then it will hold and be stable to the perturbation of a large number of edges with high probability at $p_2>p_1$.

\begin{proof}[Proof of \cref{thm:uniquenesscriterion}]
It follows from the hypotheses that there exists a function $f:\N\to(0,1]$ with $f(n)\to 0$ as $n\to\infty$ such that
\[
\P_{p_0}(A \nleftrightarrow B) \leq f(\min\{|A|,|B|\})
\]
for every two finite sets $A,B \subseteq V$.
Letting $(V_n)_{n\geq 1}$ be an exhaustion of $V$ by finite connected graphs, we note that
\begin{equation}
\P_{p_0}(A \nleftrightarrow \infty) = \lim_{n\to \infty} \lim_{m\to\infty} \P_{p_0}(A \nleftrightarrow V_m \setminus V_n) \leq f(|A|)
\label{eq:onesetf(A)}
\end{equation}
for every finite set $A \subseteq V$, and hence that the hypotheses of \cref{thm:Schonmann} are satisfied.
To complete the proof, it suffices by \cref{thm:Schonmann} to prove that the set
\[\sD := \{p \in (p_0,1]: \P_p(\text{there exist multiple infinite clusters})>0\}\]
has dense complement in $(p_0,1]$. We will do this by proving that $\sD$ is contained in the set of discontinuities of point-to-point connection probabilities and hence that $\sD$ has at most countably many elements.

\medskip

Fix $p_0 < p_1 <1$. Given finite sets $A$ and $B$ and $r\geq 1$, let $\{A \xLleftRrightarrow{r} B\}$ be the event that there exists a collection of $r$ edge-disjoint open paths each of which starts in $A$ and ends in $B$. By Menger's theorem this is equivalent to the event that there does not exist any set of $r-1$ open edges whose deletion disconnects $A$ from $B$. 
Applying \eqref{eq:interiorprobability} yields that there exists a constant $C_1$, depending on the choice of $p_1$, such that
\[
\P_{p_1}\bigl(A \xLleftRrightarrow{r} B\bigr) \geq 1 - \left(\frac{p_1}{p_1-p_0}\right)^{r-1} \P_{p_0}(A \nleftrightarrow B) \geq 1-e^{C_1 r} f(\min\{|A|,|B|\})
\]
for every $A,B \subseteq V$ finite and $r\geq 1$.
 It follows in particular that there exists a positive constant $c$, depending on the choice of $p_1$, such that if we define $g(n) = \lfloor -c \log f(n) \rfloor$ for every $n\geq 1$ then
\begin{equation}
\label{eq:manydisjointpaths}
\P_{p_1}\bigl(A \xLleftRrightarrow{r} B\bigr) \geq 1-\sqrt{f(\min\{|A|,|B|\})}
\end{equation}
for every $A,B \subseteq V$ finite and $1 \leq r \leq g(\min\{|A|,|B|\})$. Since $f(n)\to0$ as $n\to\infty$, we have that $g(n)\to\infty$ as $n\to \infty$.

\medskip

Let $p\geq p_1$ and suppose that $u$ and $v$ belong to distinct infinite clusters with positive probability in Bernoulli-$p$ percolation on $G$. 
Let $\omega_1$ and $\omega_2$ be independent copies of Bernoulli-$p$ percolation on $G$, and let $\omega\in \{0,1\}^E$ be defined by
\[
\omega(e) = \begin{cases} \omega_1(e) &\text{ if $e$ touches the cluster of $u$ or $v$ in $\omega_1$}\\
\omega_2(e) &\text{ otherwise,}
\end{cases}
\]
where we say that an edge \emph{touches} the cluster of a vertex if it has at least one endpoint in that cluster; note that these touching edges are exactly those edges that are revealed when exploring the cluster.
It is easily seen that $\omega$ is itself distributed as Bernoulli-$p$ bond percolation on $G$, and that the clusters of $u$ and $v$ are the same in $\omega_1$ and $\omega$. Condition on $\omega_1$ and suppose that $u$ and $v$ belong to distinct infinite clusters of $\omega_1$. Let $m \geq 1$ and let $A$ and $B$ be finite subsets of $K_u$ and $K_v$ respectively such that $|A|,|B| \geq m$. Using that $\omega_2$ is independent of $\omega_1$ and applying \eqref{eq:manydisjointpaths}, we have with probability at least $1-\sqrt{f(m)}$ that there exists a collection of at least $g(m)$ disjoint open paths connecting $A$ to $B$ in $\omega_2$. Since $\omega$ and $\omega_2$ coincide for those edges that do not touch the cluster of $u$ or $v$ in $\omega_1$ we deduce that, with probability at least $1-\sqrt{f(m)}$, there exists a collection of $g(m)$ edge-disjoint  paths in $G$ each of which starts in the cluster of $u$ in $\omega$, ends in the cluster of $v$ in $\omega$, and is $\omega$-open other than at its first and last edge. Since $m$ was arbitrary, $g(m)\to\infty$ as $m\to \infty$, and the probability of the aforementioned event tends to $1$ as $m\to \infty$, we deduce that, for each $k\geq 1$, there almost surely exists a collection of $k$ edge-disjoint paths in $G$ each of which starts in the cluster of $u$ in $\omega$, ends in the cluster of $v$ in $\omega$, and is $\omega$-open other than at its first and last edge.
 Considering the standard monotone coupling of percolation at $p$ and $p+\eps$, we deduce that
 \[
\P_{p+\eps}(u \leftrightarrow v) \geq \P_p(u \leftrightarrow v) + (1-(1-\eps^2)^k)\P_p(\text{$u$ and $v$ belong to distinct infinite clusters})
\]
for every $p_1 \leq p \leq p+\eps \leq 1$, $u,v\in V$, and $k\geq 1$ and hence that
\[
\P_{p+\eps}(u \leftrightarrow v) \geq \P_p(u \leftrightarrow v) + \P_p(\text{$u$ and $v$ belong to distinct infinite clusters})
\]
for every $p_1 \leq p \leq p+\eps \leq 1$ and $u,v\in V$. It follows that 
\begin{equation*}\sD \cap [p_1,1]
\subseteq \bigcup_{u,v\in V}\bigl\{p \in [p_1,1] : \P_p(u \leftrightarrow v) \text{ is discontinuous at $p$}\bigr\},
\end{equation*}
and since increasing functions have at most countably many points of discontinuity we deduce that $\sD \cap [p_1,1]$ is at most countable also. Since $p_0<p_1<1$ was arbitrary, we deduce that $\sD$ is at most countable and hence that $\sD$ has dense complement in $(p_0,1]$ as required. \qedhere

\end{proof}

\begin{proof}[Proof of \cref{thm:pu}]
It follows from \cref{thm:iso.inf} that there exist positive constants $\eta=\eta(d,c,k)>0$ and $p_0=p_0(d,c,k)<1$ such that 
\begin{equation}
\label{eq:AtoBpuproof}
\P_p(A \leftrightarrow B) \geq 1-\exp\left[-\eta \min\{|A|,|B|\}^{(d-2)/d}\right]
\end{equation}
for every $p\geq p_0$ and every two finite sets of vertices $A$ and $B$. (Indeed, simply take $n$ to be sufficiently large that $A,B \subseteq V_n$ and apply \cref{thm:iso.inf} to $G_n$.) The claim therefore follows from \cref{thm:uniquenesscriterion}.
\end{proof}

\section{The structure theory of vertex-transitive graphs}
\label{sec:structuretheory}

In this section we review the structure theory of vertex-transitive graphs that will  be used in the proofs of our main theorems and prove some related supporting technical propositions.

Let us begin with a brief historical overview of the relevant theory.
The structure theory of vertex-transitive graphs that we use in this paper has its roots in celebrated results of Gromov and Trofimov from the 1980s. Gromov's theorem states that every group of polynomial growth is virtually nilpotent \cite{gromov81poly}. Trofimov's work shows that every transitive graph of polynomial growth is roughly isometric to a Cayley graph \cite{MR811571} (see also \cite[Remark 2.2]{tt.Trof} for more details), the underlying group of which is then virtually nilpotent by Gromov's theorem. Combined with a formula of Bass \cite{bass72poly-growth} and Guivarc'h \cite{guivarch73poly-growth}, these results immediately imply that if $G=(V,E)$ is an infinite transitive graph of polynomial growth then there exist constants $c,C>0$ and an \emph{integer} $d\ge1$ such that
\begin{equation}\label{eq:BG}
cn^d\le|B(o,n)|\le Cn^d
\end{equation}
for all $o\in V$ and $n\in\N$.

It is a classical result of Coulhon and Saloff-Coste~\cite{MR1232845} that the lower bound of \eqref{eq:BG} implies in particular that $G$ satisfies \eqref{Assumption:ID} (with a possibly smaller constant $c$). This extends without too much difficulty to transitive graphs (see e.g. \cite[Proposition 4.1]{tt.resist}). Moreover, in a transitive graph the upper bound of \eqref{eq:BG} easily implies that $G$ does not satisfy \eqref{Assumption:ID} with any larger value of $d$ (see e.g. \cite[Proposition 6.7]{tt.resist}). An infinite transitive graph thus has a well-defined integer ``dimension'' $d$ that manifests itself both as the graph's growth rate and as its isoperimetric dimension.

Of course, if $G$ is finite then the existence of constants $c,C$ such that \eqref{eq:BG} holds with $d=0$ is completely trivial. For our analysis of percolation on finite graphs, therefore, we need something more finitary and quantitative. Moreover, even in an infinite graph the growth rate and isoperimetry can behave differently on different scales: for example, the graph of $\Z^{d_1} \times (\Z/m \Z)^{d_2-d_1}$ with $d_2>d_1$ looks $d_2$-dimensional on scales up to $m$, and thereafter looks $d_1$-dimensional. Indeed, as was first noted by Tao \cite{MR3658282}, the growth degree of a transitive graph of polynomial growth can increase and decrease several times as the scale increases, before finally settling down to the rate detected by the bounds \eqref{eq:BG}. See \cite[Example 1.11]{MR3658282} for a particularly illuminating example of a Cayley graph in which the growth rate is \emph{faster} at large scales than at small scales.

The key result allowing us to understand the different ``local'' dimensions of transitive graphs at different scales is the celebrated theorem of Breuillard, Green and Tao \cite{bgt12} describing the structure of finite approximate groups. Roughly speaking, an approximate group is a subset of a group that is ``approximately closed'' under the group operation. Such sets arose implicitly in the original proof of Gromov's theorem, and ``approximate closure'' can be seen as a natural finitary analogue of polynomial growth (see e.g. \cite[Proposition 11.3.1]{MR3971253}). Breuillard, Green and Tao essentially show that every finite approximate group has a large finite-by-nilpotent piece; when applied in the context of polynomial growth, this implies in particular a quantitative, finitary version of Gromov's theorem \cite[Corollary 11.7]{bgt12} (see also \cite{shalom-tao,hrushovski} for earlier results in this direction). Tessera and the second author also recently used approximate groups to give a quantitative, finitary version of Trofimov's result \cite{tt.Trof}, which complements Breuillard, Green and Tao's finitary Gromov-type theorem in the same way that Trofimov's original work complements Gromov's theorem. For more general background on approximate groups see \cite{MR3971253,tointon.survey}; for some other examples of applications of approximate groups, see \cite{green.survey} and \cite[\S11]{bgt12}.

After some fairly delicate additional work, these results lead to a number of refinements of the bounds \eqref{eq:BG} and their isoperimetric consequences. For a complete picture of the state of the art, as well as a detailed bibliography, see  \cite{tt.resist,tt.Lie}. Of particular relevance to the present work is a result of Tessera and the second author stating that if $|B(o,n)| \ge n^d$ for some vertex $o$ of a transitive graph $G=(V,E)$ and some $n\in\N$ then all subsets of $V$ size at most $\frac12|B(o,n)|$ satisfy the $d$-dimensional isoperimetric inequality \eqref{Assumption:ID} for some $c=c(d)$ (see Theorem \ref{thm:isop_resist} below). This confirmed a conjecture of Benjamini and Kozma \cite{BenKoz05}.

This isoperimetric inequality was in fact motivated by another application to probability: Tessera and the second author use it to give a quantitative, finitary refinement of Varopoulos's famous result that the simple random walk on a vertex-transitive graph is transient if and only if the graph has superquadradic growth \cite{tt.resist}, verifying and extending another conjecture of Benjamini and Kozma. In particular, this leads to a gap for the escape probability of a random walk on a vertex-transitive graph \cite[Corollary 1.7]{tt.resist}, very similar in spirit to the gap for the critical percolation probability we obtain here in \cref{thm:gap}. For some other applications of this structure theory to probability see \cite[Corollary 1.10 and Theorem 1.11]{MR3439705} and \cite[Remark 2.9]{tt.Trof}.

\medskip

We now present some specific results for use in the present work. The first follows directly from Tessera and the second author's finitary version of Trofimov's theorem.

\begin{thm}\label{prop:structure.last.large.ball}
For each $d\geq 1$ there exists a positive constant $M=M(d)$ such that the following holds. Let $G=(V,E)$ be a vertex-transitive graph and set $m=\sup\{n\le\diam(\Gamma):|B(o,n)|\ge n^d\}$. 
 If $m<\diam(G)$ then there are groups $H\lhd\Aut(G)$ and $\Gamma<\Aut(G/H) \cong \Aut(G)/H$ such that
\begin{enumerate}[label=(\roman*)]
\item every orbit of $H$ has diameter at most $\max\{m^{1/2},M\}$;
\item $\Gamma$ acts transitively on $V/H$;
\item the stabiliser in $\Gamma$ of each vertex of $G/H$ has size at most $M$;
\item $\Gamma$ has a nilpotent subgroup of step and index at most $M$. 
\end{enumerate}
\end{thm}

\begin{proof}
The statement of \cite[Corollary 2.4]{tt.Trof} features parameters $d\ge0$ and $\lambda\in(0,1)$, a constant $n_0(d,\lambda)$, and various implied constants depending only on $d$ and $\lambda$; let $M(d)$ be the maximum of $n_0(d,\lambda)$ and these implied constants when $\lambda=1/2$. If $M\le m<\diam(G)$ then the desired conclusion follows from applying \cite[Corollary 2.4]{tt.Trof} with $n=m+1$. If $m<M\le\diam(G)$ then the desired conclusion follows from applying \cite[Corollary 2.4]{tt.Trof} with $n=M$. If $\diam(G)<M$ then we may take $H=\Aut(G)$ and $\Gamma=\{1\}$.
\end{proof}

The next result shows that a bound of the form $|B(o,n)| \leq c n^d$ is enough to ensure that the ``local'' dimension of a transitive graph can never go above $d$ at higher scales. This verified a conjecture of Benjamini.

\begin{thm}[{\cite[Corollary 1.5]{tt.Trof}}]
\label{thm:dimension_drop}
Let $d\geq 1$ be an integer. There exist positive constants  $c=c(d)$ and $C=C(d)$ such that if $G=(V,E)$ is a connected, vertex-transitive, locally finite graph, $o$ is a vertex of $G$, and $n\geq 1$ is such that $|B(o,n)| \leq c n^d$ then 
\[
\frac{|B(o,m_2)|}{|B(o,m_1)|} \leq C\left(\frac{m_2}{m_1}\right)^{d-1} 
\]
for every $n \leq m_1 \leq m_2$. 
\end{thm}

The next theorem is essentially what results from combining Theorem \ref{thm:dimension_drop} with the Coulhon--Saloff-Coste argument, and is a slight generalisation of \cite[Theorem 1.20]{tt.resist}.

\begin{thm}
\label{thm:isop_resist}
For each $d \geq 1$ there exists a positive constant $c=c(d)$ such that the following holds. If $G=(V,E)$ is a connected, locally finite, vertex-transitive graph, $o$ is a vertex of $G$, and $n \geq 1$ 
then
\[
|\partial_E A|\ge|\partial_V^+A| \geq c(d) \min\left\{1, \frac{1}{n}|B(o,n)|^{1/d} \right\} \,|A|^{(d-1)/d}
\]
for every set $A \subseteq V$ with $|A| \leq \frac{1}{2} |B(o,n)|$. In particular, if $\sup \{n : |B(o,n)| \geq n^d\}=\diam(G)$ then $G$ satisfies the $d$-dimensional isoperimetric inequality \eqref{Assumption:ID}.
\end{thm}

Note that the case in which $m=\diam(G)$ includes the case in which $G$ is infinite and $m=\infty$.

\begin{proof}
It is stated in  \cite[Theorem 1.20]{tt.resist} that for each positive integer $k$ there exists a positive constant $c(k) \leq 1$ such that if $n\geq 1$ is such that $|B(o,n)| \geq n^d$ for some $d\geq 1$ then
\[
|\partial_V^+ A| \geq c(\lfloor d\rfloor) |A|^{(d-1)/d}
\]
for every set $A \subseteq V$ with $|A| \leq \frac{1}{2} |B(o,n)|$. 
 The case of \cref{thm:isop_resist} in which $|B(o,n)| \geq n^d$, which includes the case that $n=1$, follows immediately. Now suppose that $n\geq 2$ and that $|B(o,n)| < n^d$. Let $c'(d) = \min\{1,c(1),\ldots,c(\lfloor d \rfloor)\}>0$ for each $d\geq 1$. We can write $|B(o,n)| = n^{d-\delta}$ where 
\[
\delta = d-\frac{\log |B(o,n)|}{\log n},
\]
and deduce that
\[
|\partial_V^+ A| \geq c'(d) |A|^{(d-\delta-1)/(d-\delta)}
\]
for every $A \subseteq V$ with $|A| \leq \frac12 |B(o,n)|$: This is trivial when $d-\delta<1$, and follows from \cite[Theorem 1.20]{tt.resist} otherwise. A little algebra gives that
\[
\frac{d-\delta-1}{d-\delta} = \frac{d-1}{d}-\frac{\delta}{d(d-\delta)} 
\]
and we easily obtain that
\begin{multline*}
|\partial_V^+ A| \geq c'(d) |A|^{(d-1)/d} |A|^{-\delta/d(d-\delta)} \geq 
c'(d) |A|^{(d-1)/d} |B(o,n)|^{-\delta/d(d-\delta)} 
\\
= c'(d) \left(\frac{|B(o,n)|}{n^d} \right)^{1/d} \,|A|^{(d-1)/d}
\end{multline*}
as claimed.
\end{proof}

Again, speaking very roughly, \cref{prop:structure.last.large.ball,thm:isop_resist} tell us that for every locally finite transitive graph $G=(V,E)$, there is a scale $m$ such that $G$ ``looks high-dimensional" on scales smaller than $m$ and ``looks low-step nilpotent" on scales larger than $m$.

\subsection{Transitive graphs as quotients of Cayley graphs}

In this section we describe a construction due to Abels \cite{MR0344375} expressing any vertex-transitive graph as a quotient of a Cayley graph of its isometry group. We start in a fairly abstract setting. Suppose $\Gamma$ is a group with a symmetric generating set $S$, and suppose that $H<\Gamma$ is a subgroup with respect to which $S$ is \emph{bi-invariant} in the sense that $HSH=S$. This implies that given $x,y\in\Gamma$ and $h\in H$ we have $x\sim y$ in $\Cay(\Gamma,S)$ if and only if $xh\sim yh$ in $\Cay(\Gamma,S)$, where we write $u\sim v$ to mean that two vertices $u,v$ of a given graph are neighbours. We may therefore define an injective homomorphism $\rho:H\to\Aut(\Cay(\Gamma,S))$ by setting $\rho(h)(x)=xh^{-1}$ for $x\in\Gamma$. We then denote by $\operatorname{A}(\Gamma,S,H)$ the quotient graph $\Cay(\Gamma,S)/\rho(H)$. Note that the vertex set of $\operatorname{A}(\Gamma,S,H)$ is the set $\Gamma/H$ of left cosets of $H$ in $\Gamma$, and the action of $\Gamma$ on $\Cay(\Gamma,S)$ induces a transitive action of $\Gamma$ on $\operatorname{A}(\Gamma,S,H)$ given by $g(xH)=(gx)H$, so that $\operatorname{A}(\Gamma,S,H)$ is a transitive graph.

It turns out that every transitive graph whose automorphism group is discrete (and hence every finite transitive graph) can be realised in this way, as follows. Note that if $\Gamma$ is a closed subgroup of $\Aut(G)$ then the stabilizer $\Gamma_o$ is compact, so that $\Gamma$ is discrete if and only if $\Gamma_o$ is finite.

\begin{prop}\label{prop:cay-ab}
Let $G=(V,E)$ be a vertex-transitive graph, and suppose that $\Gamma$ is a discrete transitive subgroup of $\Aut(G)$.
 Let $S=\{\gamma\in\Gamma:d(\gamma o,o)\le1\}$. Then $S$ is a symmetric generating set for $\Gamma$ satisfying $\diam_S(\Gamma)=\diam(G)$ and $|S^n|=|B(o,n)||\Gamma_o|$ for every non-negative integer $n$. Moreover, $S$ is bi-invariant with respect to the stabiliser $\Gamma_o$, and $G\cong\operatorname{A}(\Gamma,S,\Gamma_o)$. In particular, there exists $H\le\Aut(\operatorname{Cay}(\Gamma,S))$ with $H\cong\Gamma_o$ such that $G\cong\operatorname{Cay}(\Gamma,S)/H$.
\end{prop}

\begin{proof}
The fact that $S$ is a symmetric generating set for $\Gamma$ satisfying $\diam_S(\Gamma)=\diam(G)$ follows immediately from \cite[Lemma 3.4]{tt.Trof}. The sets $S^n$ are of the claimed cardinality by \cite[Lemma 3.8]{tt.Trof}. To see that $S$ is bi-invariant with respect to $\Gamma_o$, first note that we trivially have $S\Gamma_o=S$. Applying this and the symmetry of $S$ repeatedly then gives $\Gamma_oS\Gamma_o=\Gamma_oS=(S\Gamma_o)^{-1}=S^{-1}=S$, as required. Note also that bi-invariance in turn implies that for every $\gamma,\gamma'\in\Gamma$ we have
\begin{equation}\label{eq:CA.bi-inv}
\gamma\Gamma_o\sim\gamma'\Gamma_o\text{ in }\operatorname{A}(\Gamma,S,\Gamma_o)\quad\implies\quad \gamma\sim\gamma'\text{ in }\Cay(\Gamma,S).
\end{equation}
Indeed, if $\gamma\Gamma_o\sim\gamma'\Gamma_o$ in $\operatorname{A}(\Gamma,S,\Gamma_o)$ then by definition there exists $h\in\Gamma_o$ such that $\gamma h\sim\gamma'$ in $\Cay(\Gamma,S)$, which in turn means that there exists $s\in S$ such that $\gamma hs=\gamma'$. The bi-invariance of $S$ then implies that $hs\in S$, and hence that $\gamma\sim\gamma'$ in $\Cay(G,S)$, as claimed.

Since $\Gamma$ acts transitively on $G$, the map
\[\begin{array}{ccccc}
\ph&:&G&\to&\Gamma/\Gamma_o\\
    &&\gamma o&\mapsto&\gamma\Gamma_o
\end{array}\]
is a well-defined bijection by the orbit--stabiliser theorem. This map $\ph$ defines a graph isomorphism $G\to\operatorname{A}(\Gamma,S,\Gamma_o)$, since for every $\gamma,\gamma'\in\Gamma$ we have
\begin{align*}
\gamma\Gamma_o\sim\gamma'\Gamma_o\text{ in }\operatorname{A}(\Gamma,S,\Gamma_o)&\iff \gamma^{-1}\gamma'\in S\setminus \Gamma_o&\text{(by \eqref{eq:CA.bi-inv})}\\
   &\iff \gamma^{-1}\gamma' o\sim o\text{ in }G\\
   &\iff \gamma o\sim\gamma' o\text{ in }G
\end{align*}
as claimed.
\end{proof}

\subsection{Isoperimetry in induced subgraphs}
\label{subsec:induced_isoperimetry}

We will see in \cref{sec:mainproofs} that is a straightforward matter to deduce \cref{thm:gap}, which concerns the value of $p_c$ in \emph{infinite} transitive graphs, from the structure-theoretic results \cref{prop:structure.last.large.ball,thm:isop_resist} together with our analyses of virtually nilpotent groups from \cref{sec:Nilpotent} and graphs of large isoperimetric dimension from \cref{sec:Isoperimetry}. Our main results regarding \emph{finite} transitive graphs require a rather more delicate approach 
  owing to the possibility that the auxiliary graph $G/H$ is of `intermediate size', i.e., that $1 \ll |V/H| \ll |V|$, in which case we cannot rely on the results of either \cref{sec:Nilpotent} or \cref{sec:Isoperimetry} alone to establish the existence of a giant component. 
The purpose of this section is to prove the following proposition, which is a variation on \cref{thm:isop_resist} and will be used to apply the results of \cref{sec:Isoperimetry} in the case that $G/H$ is of intermediate size.

\begin{proposition}\label{prop:iso.subgraph}
For each  integer $d \geq 1$, $k\geq 1$, $\eps>0$, and $\rho \in(0,1)$ there exist positive constants  $\ell=\ell(d,k,\eps,\rho)$, $\eps_0=\eps_0(d)$, and $c=c(d,k,\eps,\rho)$ such that the following holds. Let $G=(V,E)$ be a vertex-transitive graph of degree $k$ and let $o$ be a vertex of $G$. If $n \geq 1$ is such that
$\eps n^d \leq |B(o,n)| \leq \eps_0 n^d$
then there exists a subset $I\subseteq B(o, \ell n)$ such that
\begin{enumerate}
\item 
$|I\cap B(o,  n )|\ge\rho|B(o,  n )|$ and
\item the subgraph of $G$ induced by $I$ satisfies the $d$-dimensional isoperimetric inequality \eqref{Assumption:ID}.
\end{enumerate}
\end{proposition}

Note that a nontrivial argument is still required to deduce our main theorems from this together with our analyses of the nilpotent and high-dimensional cases. This argument is carried out in \cref{sec:mainproofs}.

\medskip

We now begin to work towards the proof of \cref{prop:iso.subgraph}.
We begin by adapting the arguments of Coulhon--Saloff-Coste and Tessera and the second author to prove the following result, which essentially says that if a finite set $A$ of vertices in a vertex-transitive graph is sparse in every ball of radius $r$ then its external vertex boundary $\partial_V^+\!A$ has size at least a constant times $|A|/r$.

\begin{proposition}[Locally sparse sets have large boundary]
\label{prop:sparse.iso}
 Let $G=(V,E)$ be a locally finite, vertex-transitive graph, and let $A$ be a finite set of vertices of $G$. If $\rho\in (0,1)$ and $r\geq 1$ are such that $|A\cap B(x,r)|\le\rho |B(x,r)|$ for every $x\in V$ then
\[
|\partial_E A| \geq |\partial_V^+\!A|\ge\frac{1-\rho}{6r}|A|.
\]
\end{proposition}

Proposition \ref{prop:sparse.iso} can be seen as a generalisation of \cite[Proposition 4.1]{tt.resist}, and is implicitly contained in the proof of that result. We provide the details here for the convenience of the reader.

We now briefly recall some relevant definitions that will be used in the proof of \cref{prop:sparse.iso}.
  Given a locally finite, vertex-transitive graph $G$, the group $\Aut(G)$ of automorphisms of $G$ is a locally compact group with respect to the topology of pointwise convergence, and every closed subgroup of $\Aut(G)$ is also a locally compact group in which vertex stabilisers are compact and open. Moreover, an arbitrary closed subgroup $\Gamma<\Aut(G)$ admits a (left) \emph{Haar measure} $\mu$, the properties of which include that
\begin{enumerate}[label=(\arabic*)]
\item$\mu(K)<\infty$ if $K \subseteq \Gamma$ is compact,
\item$\mu(U)>0$ if $U \subseteq \Gamma$ is open and nonempty,
\item$\mu(\gamma A)=\mu(A)$ for every Borel set $A\subset \Gamma$ and every $\gamma \in \Gamma$, and
\item\label{item:Haar.unique} if $\mu'$ is another Haar measure on $\Gamma$ then there exists $\lambda>0$ such that $\mu'=\lambda\cdot\mu$.
\end{enumerate}
See \cite[\S15]{MR551496} for a detailed introduction to Haar measures.
Given a locally compact group $\Gamma$ with left Haar measure $\mu$, we define the space $L^1(\Gamma)$ with respect to $\mu$, so that $\Gamma$ acts on $L^1(\Gamma)$ via $\gamma f(x)=f(x\gamma)$. Note that since a right translate of a Haar measure is again a Haar measure, by property \ref{item:Haar.unique} there exists a homomorphism $\Delta_\Gamma:\Gamma\to\R^+$, called the \emph{modular function} of $\Gamma$, such that
\[
\mu(A \gamma)=\Delta_\Gamma(\gamma^{-1})\mu(A)
\]
for every Borel set $A$. See also \cite[Section 2.1]{Hutchcroftnonunimodularperc} for background on the modular function for probabilists.

\begin{remark}
A locally compact group is said to be \textbf{unimodular} if its modular function is identically equal to $1$, or equivalently if its left Haar measures are also right-invariant. Every finite or countable discrete group is unimodular since its Haar measure is equal to counting measure. Thus, in our primary setting of finite transitive graphs one may assume that all the groups appearing below have $\Delta \equiv 1$, simplifying the analysis somewhat.
\end{remark}

The following lemma follows implicitly from the proof of \cite[Proposition 4.4]{tt.resist}.

\begin{lemma}\label{lem:lc.sparse.iso}
Let $\Gamma$ be a locally compact group with a left Haar measure $\mu$ and a precompact symmetric open generating set $S$. Let $A\subset \Gamma$ be a precompact open set. If $\rho\in (0,1)$ and $r\geq 1$ are such that $\mu(A\cap\gamma S^r)\le\rho\mu(S^r)$ for every $\gamma\in \Gamma$ then
\[
\sup_{s\in S}\mu(As\setminus A)\ge\frac{1-\rho}{6r}\mu(A).
\]
\end{lemma}
\begin{proof}
We follow the proof of \cite[Proposition 4.4]{tt.resist}. If there exists $s\in S$ such $\Delta_\Gamma(s)\ge1+\frac{\log 2}{r}$ then
\[
\mu(As^{-1}\setminus A)\ge\mu(As^{-1})-\mu(A)=(\Delta_\Gamma(s)-1)\mu(A)\geq \frac{\log 2}{r}\mu(A),
\]
which certainly gives the required bound. We may therefore assume that $\Delta_\Gamma(s)\le1+\frac{\log 2}{r}$ for every $s\in S$. This implies that
\begin{equation}\label{eq:mod.log2}
\Delta_\Gamma(\gamma)\le\left(1+\frac{\log 2}{r}\right)^r\leq 2
\end{equation}
for every $\gamma \in S^r$.

Define a linear operator $M:L^1(\Gamma)\to L^1(\Gamma)$ via 
\[M(f)(x)=\frac{1}{\mu(S^r)}\int_{\gamma \in S^r}\,f(x\gamma) \dif \mu(\gamma).\] The hypothesis on $A$ implies that its indicator function $\mathbbm{1}_A$ satisfies $M(\mathbbm{1}_A)(x)\le\rho$ for every $x\in \Gamma$, and hence that
\begin{equation}\label{eq:Mf}
\|\mathbbm{1}_A-M(\mathbbm{1}_A)\|_1\ge(1-\rho)\mu(A).
\end{equation}
On the other hand, given $\gamma \in S^r$ written as $\gamma =s_1\cdots s_r$ with each $s_i\in S$, the triangle inequality implies that
\[
\|\mathbbm{1}_A-\gamma \mathbbm{1}_A\|_1\le\sum_{i=0}^{r-1}\|s_{r-i+1}\ldots s_r\mathbbm{1}_A-s_{r-i}\ldots s_r\mathbbm{1}_A\|_1=\sum_{i=0}^{r-1}\Delta_\Gamma(s_{r-i+1}\ldots s_r)\|\mathbbm{1}_A-s_{r-i}\mathbbm{1}_A\|_1.
\]
By \eqref{eq:mod.log2}, this implies that $\|\mathbbm{1}_A-\gamma\mathbbm{1}_A\|_1\le2r\sup_{s\in S}\|\mathbbm{1}_A-s\mathbbm{1}_A\|_1=2r\sup_{s\in S}\mu(A\vartriangle As)$,
and averaging this bound over $S^r$ then gives that
\begin{equation}\label{eq:1-M1<A.sd.As}
\|\mathbbm{1}_A-M(\mathbbm{1}_A)\|_1\le2r\sup_{s\in S}\mu(A\vartriangle As).
\end{equation}
To conclude, simply note that for all $s\in S$ we have that
\begin{align*}
\mu(A\vartriangle As)&=\mu(A\setminus As)+\mu(As\setminus A)\\
   &=\Delta_\Gamma^{-1}(s^{-1})\mu(As^{-1}\setminus A)+\mu(As\setminus A)\\
   &\le3\sup_{s'\in S}\mu(As'\setminus A)&\text{(by \eqref{eq:mod.log2})},
\end{align*}
so that the desired bound follows from \eqref{eq:Mf} and \eqref{eq:1-M1<A.sd.As}.
\end{proof}

\begin{proof}[Proof of Proposition \ref{prop:sparse.iso}]
We follow the proof of \cite[Proposition 4.1]{tt.resist}. Set $\Gamma=\Aut(G)$ and fix a vertex $o$ of $G$. By transitivity we may pick, for each $x\in V$, an automorphism $\gamma_x\in \Gamma$ such that $\gamma_x o =x$. Write $\Gamma_o$ for the stabiliser of $o$ in $\Gamma$, and given an arbitrary subset $X\subset V$, write $\Gamma_{o\to X}=\{\gamma \in \Gamma:\gamma o\in X\}$, noting that
\[
\Gamma_{o\to X}=\bigcup_{x\in X}\gamma_x\Gamma_o.
\]
The stabiliser $\Gamma_o$ is open by definition, and it is shown in \cite[Lemma 4.4]{tt.Trof} that it is compact, so $\Gamma_{o\to X}$ is compact and open whenever $X$ is finite. Normalising the left Haar measure $\mu$ on $\Gamma$ so that $\mu(\Gamma_o)=1$, we also have that
\begin{equation}\label{eq:pullback.mu}
\mu(\Gamma_{o\to X})=|X|,
\end{equation}
since the sets $\gamma_x \Gamma_o = \{\gamma \in \Gamma : \gamma o =x\}$ are disjoint.

It is shown in \cite[Lemma 4.8]{tt.Trof} that the set $S=\{\gamma \in \Gamma:d(\gamma o,o)\le1\}$ is a compact open generating set for $\Gamma$ containing the identity, and in \cite[Lemma 3.4]{tt.Trof} that
\begin{equation}\label{eq:Lem3.4}
S^r=\Gamma_{o\to B(o,r)}
\end{equation}
for every $r\in\N$. For every $\gamma \in \Gamma$ we have that
\begin{align*}
\Gamma_{o\to A}\cap \gamma S^r&=\Gamma_{o\to A}\cap \gamma \Gamma_{o\to B(o,r)}&\text{(by \eqref{eq:Lem3.4})}\\
   &=\Gamma_{o\to A\cap B(\gamma o,r)},
\end{align*}
and hence that
\begin{align*}
\mu(\Gamma_{o\to A}\cap \gamma S^r)&=|A\cap B(\gamma o,r)|&\text{(by \eqref{eq:pullback.mu})\phantom{.}}\\
   &\le\rho|B(o,r)|&\text{(by hypothesis)\phantom{.}}\\
   &=\rho\mu(S^r)&\text{(by \eqref{eq:pullback.mu} and \eqref{eq:Lem3.4}).}
\end{align*}
Lemma \ref{lem:lc.sparse.iso} and \eqref{eq:pullback.mu} therefore imply that 
\[
\sup_{s\in S}\mu(\Gamma_{o\to A}s\setminus \Gamma_{o \to A})\ge\frac{1-\rho}{6r}|A|.
\]
It is shown in \cite[(4.7)]{tt.resist} that $\mu(\Gamma_{o\to A}S\setminus \Gamma_{o\to A})=|\partial_V^+A|$, so this implies the required bound.
\end{proof}

\begin{corollary}\label{lem:min.contains.ball}
For each $d\geq 1$ there exists a positive constant $c=c(d)$ such that the following holds. If $G=(V,E)$ is a connected, locally finite, vertex-transitive graph, $o$ is a vertex of $G$, $\rho \in (0,1)$, and $n \geq 1$ then
\[
|\partial_E A| \geq |\partial_V^+ A| \geq c(d) (1-\rho) \min\left\{1, \frac{1}{n}|B(o,n)|^{1/d} \right\} \cdot |A|^{(d-1)/d}
\]
for every set $A \subseteq V$ such that $|A \cap B(x, n)| \leq \rho |B(x,n)|$ for every $x\in V$.
\end{corollary}

\begin{proof}
If $|A|\le\frac12|B(o,n)|$ then the desired bound follows immediately from \cref{thm:isop_resist}. Meanwhile, if $|A|>\frac12|B(o,n)|$ then the desired bound
\[
|\partial_E A| \geq |\partial_V^+ A| \geq \frac{1-\rho}{6 n} |A| \geq \frac{(1-\rho)}{12 n}  |A|^{(d-1)/d} |B(o,n)|^{1/d}
\]
 follows from Proposition \ref{prop:sparse.iso}.
\end{proof}

\begin{proof}[Proof of Proposition \ref{prop:iso.subgraph}]
We apply an argument similar to one used by Le Coz and Gournay \cite[Lemma 3.2]{coz2019separation}.
Let $\eps_0=\eps_0(d)$ be the constant $c(d)$ from \cref{thm:dimension_drop}. 
Fix $\eps>0$ and $n\geq 1$ such that $\eps n^d \leq |B(o,n)| \leq \eps_0 n^d$. 
Let $\ell \geq 1$ and consider the quantity
\[
\eta(\ell) = \min\left\{ \frac{|\partial_E A|}{|A|^{(d-1)/d}} : A \subseteq B(o, 2 \ell n) \right\}.
\]
Writing $B_r=B(o,r)$ and noting that $|\partial_E B_r|\le k|\partial_V^-B_r|=k(|B_r|-|B_{r-1}|)$,
we have that
\[
\eta(\ell) \leq  k \cdot \frac{|B_{r}|-|B_{r-1}|}{|B_r|^{(d-1)/d}}.
\]
for every $0\leq r \leq 2\ell n$.
Averaging over $\ell n < r \leq 2\ell n$, we deduce from \cref{thm:dimension_drop} that there exist constants $C_1$ and $C_2$ depending only on $d$ such that 
\[
\eta(\ell) \leq \frac{k}{\ell n}\sum_{r=\ell n+1}^{2\ell n}\frac{|B_{r}|-|B_{r-1}|}{|B_r|^{(d-1)/d}} \leq \frac{k}{\ell n} \frac{|B_{2\ell n}|}{|B_{\ell n}|^{(d-1)/d}} \leq \frac{C_1k}{\ell n} |B_{2\ell n}|^{1/d} \leq \frac{C_2k}{\ell n} (n^d \ell^{d-1})^{1/d} = C_2k \ell^{-1/d}.
\]
The only important feature of this bound is that the right hand side tends to zero as $\ell\to\infty$, at a rate depending only on $d$ and $k$.
Indeed, letting $c_1=c_1(d)$ be the constant from \cref{lem:min.contains.ball}, if $I \subseteq B(o, 2\ell n)$ is a set attaining the minimum in the definition of $\eta(\ell)$ then there exists $x\in V$ such that
\[
  |I \cap B(x,n)| \geq \left(1-\frac{\eta(\ell)}{c_1 \eps^{1/d}}\right)|B(x,n)|.
\]
It follows that there exists a constant $\ell=\ell(d,k,\rho,\eps)$ such that if $I \subseteq B(o,2\ell n)$ is a set attaining the minimum in the definition of $\eta(\ell)$ then there exists $x\in V$ such that
\[
  |I \cap B(x,n)| \geq \rho |B(x,n)|.
\]

Fix one such set $I$ and $x\in V$. Since $I \subseteq B(x,5\ell n)$, it suffices to prove that the subgraph of $G$ induced by $I$ satisfies a $d$-dimensional isoperimetric inequality with constants depending only on $d,k,\eps$ and $\rho$. 
Since $|B(o,2\ell n)| \geq |B(o,n)| \geq  \eps n^d$, we can apply \cref{thm:isop_resist} to deduce that there exists a positive constant $c_2=c_2(d,k,\rho,\eps)$ such that
\[
|\partial_E A| \geq c_2 |A|^{(d-1)/d}
\]
for every subset $A$ with $|A| \leq |I|/2 \leq |B(o,2\ell n)|/2$. We are not done at this point of course, since what we really need is a lower bound on the size of the boundary of $A$ \emph{considered as a subset of the subgraph of $G$ induced by $I$.} Write $\partial_I A$ for this boundary.
Fix $A\subseteq I$ with $|A|\le|I|/2$, write $A'=I\setminus A$ and, following Le Coz and Gournay, note that
\begin{equation}\label{eq:AA'I}
2|\partial_I A| = |\partial_E A|+|\partial_E A'|-|\partial_E I|.
\end{equation}
 We have by minimality of $I$ that
\[\frac{|\partial_E A|}{|A|^{(d-1)/d}}\geq 
\frac{|\partial_E I|}{|I|^{(d-1)/d}} \qquad \text{ and } \qquad \frac{|\partial_E A'|}{|A'|^{(d-1)/d}}\geq 
\frac{|\partial_E I|}{|I|^{(d-1)/d}}. \]
Thus, writing $\alpha = |A|/|I| $ and $\gamma=(d-1)/d$, it follows that
\[
|\partial_E I|-|\partial_E A'| \leq \left(1-(1-\alpha)^\gamma\right)|\partial_E I| \leq \frac{1-(1-\alpha)^\gamma}{\alpha^\gamma}|\partial_E A|
\]
and hence by \eqref{eq:AA'I} that
\[
|\partial_I A| \geq \frac{\alpha^\gamma+(1-\alpha)^\gamma-1}{2\alpha^\gamma} |\partial_E A| \geq \frac{2^{1/d}-1}{2^{1/d}} |\partial_E A| \geq c_3 |A|^{(d-1)/d}
\]
for some positive constant $c_3=c_3(d,k,\rho,\eps)$, where in the central inequality we used the readily verified fact that $\frac{\alpha^\gamma+(1-\alpha)^\gamma-1}{2\alpha^\gamma}$ is a decreasing function of $\alpha$ on $[0,1/2]$. \qedhere

\end{proof}

\section{Proofs of the main theorems}
\label{sec:mainproofs}

\cref{thm:gap}, which concerns \emph{infinite} vertex-transitive graphs, can be deduced easily from \cref{prop:structure.last.large.ball,thm:isop_resist} together with \cref{thm:quantitative_DGRSY,thm:inf.perc.nilp}.

\begin{proof}[Proof of \cref{thm:gap}]
By \cref{cor:quasitrans} we may assume that $G$ is transitive. Set $m=\sup\{n\in\N:|B(o,n)|\ge n^5\}$. If $m=\infty$ then the theorem follows from Theorems \ref{thm:isop_resist} and \ref{thm:quantitative_DGRSY}. If $m<\infty$ then let $M$ be the constant $M(5)$ appearing in \cref{prop:structure.last.large.ball}, and note that that theorem implies that there exist groups $H\lhd\Aut(G)$ and $\Gamma<\Aut(G/H)$ such that $\Gamma$ acts transitively on $V/H$ and has a nilpotent subgroup of step and index at most $M$, and such that the stabiliser in $\Gamma$ of each orbit $Hv$ with $v\in V$ has size at most $M$.  \cref{prop:cay-ab} implies that $S=\{\gamma\in\Gamma:d(\gamma(Ho),Ho)\le1\}$ is a finite symmetric generating set for $\Gamma$, and since the growth of $G$ is superlinear, the same proposition also implies that the growth of $\Cay(\Gamma,S)$ is superlinear. As is well known (see e.g. \cite[Lemma 11.1.2 and Proposition 11.1.3]{MR3971253}), this means that $\Gamma$ is not virtually cyclic. \cref{thm:inf.perc.nilp} therefore implies that there is an absolute constant $\eps>0$ such that $p_c(\Cay(\Gamma,S))\le1-\eps_0$. \cref{prop:cay-ab} implies that $G/H$ is isomorphic to a quotient of $\Cay(\Gamma,S)$ by a subgroup of $\Aut(\Cay(\Gamma,S))$ of order at most $M$, so \cref{lem:bdd.orbits} imples that $p_c(G/H)\le1-\eps$ for some absolute constant $\eps>0$. The theorem then follows from \cref{prop:Benj-Schr}.
\end{proof}

The remainder of this section is dedicated to the proof of \cref{thm:fin.perc}. 
We begin with some simple and standard geometric lemmas.

\begin{lemma}[cf. Ruzsa's covering lemma \cite{MR1701207}]\label{lem:covering}
Let $A$ be a subset of a graph $G$, and let $m\in\N$. Let $X$ be a maximal subset of $A$ such that the balls $B(x,m)$ are pairwise disjoint. Then $A\subseteq \bigcup_{x\in X}B(x,2m)$.
\end{lemma}

\begin{proof}
The maximality of $X$ implies that for every $a\in A$ there exists $x\in X$ such that $B(x,m)\cap B(a,m)\ne\varnothing$, and hence $a\in B(x,2m)$.
\end{proof}

\begin{lemma}\label{lem:linear.rel.growth}
Let $G$ be a graph of diameter at least $n$ and let $v$ be a vertex of $G$. Then $B(v,n)$ contains at least $(n-2m)/(4m+2)$ disjoint balls of radius $m$ for each $m\leq n/2$. As such, if $G$ is transitive and $1 \leq m_1 \leq m_2 \leq \diam(G)$ then
\[
\frac{|B(v,m_2)|}{|B(v,m_1)|}  \geq 1 \vee \frac{m_2-2m_1}{4m_1+2} \geq \frac{m_2}{8m_1}.
\]
\end{lemma}

\begin{proof}The second claim follows easily from the first by a small calculation. We now prove the first, following \cite[Lemma 5.3]{tt.resist}. Since $\diam(G)\ge n$, there exists a geodesic of length $k=\lceil n/2 \rceil$ starting at $v$. Let $x_0=v,x_1,x_2,\ldots,x_k$ be the vertices of this geodesic, written in increasing order of distance from $v$. The balls $B(x_{(2m+1)i},m)$ with $0\le i\le(k-m)/(2m+1)$ are then disjoint subsets of $B(v,n)$. This is easily seen to imply the claim.
\end{proof}

\begin{proof}[Proof of Theorem \ref{thm:fin.perc}]
Throughout the proof, we will write $\asymp$, $\preceq$, and $\succeq$ for equalities and inequalities holding to within positive multiplicative constants depending only on $k,\lambda$, and $\eps$. Note that small values of $|V|$ can be handled by increasing $p$, so that we may assume wherever necessary that $|V|$ is larger than any given constant depending on $k,\lambda,$ and $\eps$.

By \cref{cor:quasitrans} we may assume that $G$ is transitive. Let $M$ be the absolute constant $M=M(14)\vee M(13)$, where $M(13)$ and $M(14)$ are the constants coming from \cref{prop:structure.last.large.ball}. We will prove the theorem with the constant $a=M+400$. Indeed, we will prove that for each $k,\lambda \geq 1$ and $\eps>0$ there exists a constant $p_0=p_0(k,\lambda,\eps)<1$ such that if $G=(V,E)$ is a finite, connected, vertex-transitive graph of degree at most $k$ and 
\[
\diam(G)\leq \frac{\lambda |V|}{(\log |V|)^{M+400}}
\]
then $\P_p^G(x\leftrightarrow y) \geq 1-\eps$ for every $x,y\in V$; this suffices by the first part of \cref{lem:cluster/connect}. Fix one such choice of $k,\lambda \geq 1$, $\eps>0$, and $G=(V,E)$. Let $\eps_0$ be the positive absolute constant $1\wedge c(14) \wedge \eps_0(12) \wedge \eps_0(13) \wedge \eps_0(14)$, where $c(14)$ is as in \cref{thm:dimension_drop} and $\eps_0(12)$, $\eps_0(13)$, and $\eps_0(14)$ are as in \cref{prop:iso.subgraph}.
Fix a vertex $o$ of $G$ and consider the three scales $1\leq m_1 \leq m_2 \leq m_3$ defined by
\begin{align*}m_1&=\max\bigl\{n\le\diam(G):|B(o,n)|\ge \textstyle{\frac12}\eps_0 n^{14}\bigr\},\\ 
m_2&=\max\bigl\{n\le\diam(G):|B(o,n)|\ge \textstyle{\frac12}\eps_0 n^{13}\bigr\},\\ 
m_3&=\max\bigl\{n\le\diam(G):|B(o,n)|\ge \textstyle{\frac12}\eps_0 n^{12}\bigr\}.
\end{align*}

If $m_3=\diam(G)$ then the theorem follows from  \cref{thm:isop_resist} and \cref{thm:iso.inf}.
If $m_1 \leq 10^{156} \vee M^2$ then letting $R=10^{78}\vee M$ we may apply \cref{prop:structure.last.large.ball} with $d=14$ to obtain groups $H\lhd\Aut(G)$ and $\Gamma<\Aut(G/H)$ such that $\Gamma$ acts transitively on $V/H$ and has a nilpotent subgroup of step and index at most $M$, such that the stabiliser in $\Gamma$ of each orbit $Hv$ with $v\in V$ has size at most $M$, and such that each such orbit has diameter at most $R$, and hence size at most $R^{14}$. \cref{prop:cay-ab} then implies that $S=\{\gamma\in\Gamma:d(\gamma(Ho),Ho)\le1\}$ is a symmetric generating set for $\Gamma$ of size at most $(k+1)M$, that $\diam_S(\Gamma)=\diam(G/H)$, and that $|\Gamma|=|\Gamma_{Ho}||V/H|\asymp|V|$, and hence that
\[
\diam_S(\Gamma)=\diam(G/H)\le\diam(G)\le\frac{\lambda |V|}{(\log|V|)^{M}}
\preceq \frac{|\Gamma|}{(\log|\Gamma|)^{M}}.
\]
It then follows from Theorem \ref{thm:fin.perc.nilp}, Proposition \ref{prop:cay-ab}, Lemma \ref{lem:bdd.orbits} and \cref{prop:Benj-Schr} that for every $\eps>0$ there exists $q_1=q_1(k,\lambda,\eps)<1$ such that
\[
\P_p^G\bigl(u\leftrightarrow Hv\bigr)\geq \sqrt{1-\eps}
\]
for every $u,v\in V$ and $p\geq q_1$. Since the orbits of $H$ have diameter at most $R$, there also exists $q_2=q_2(\eps)<1$ such that
\[
\P_p^G\bigl(u\leftrightarrow v\bigr)\ge \sqrt{1-\eps}
\]
for $p \geq q_2$ and every $u,v$ belonging to the same orbit of $H$. Letting $q_3$ be defined by $1-q_3=(1-q_1)(1-q_2)$, we have as usual that Bernoulli-$q_3$ percolation is distributed as the union of two independent copies of Bernoulli-$q_1$ and Bernoulli-$q_2$ percolation, so that
\[
\P_{q_3}^G\bigl(u \leftrightarrow v) \geq \P_{q_1}^G\bigl(u\leftrightarrow Hv\bigr) \cdot \min \left\{\P_{q_2}^G(w \leftrightarrow v) : w \in H v\right\} \geq 1-\eps
\]
for every $u,v\in V$ and the theorem is proved in this case.

\medskip

From now on we assume that $10^{156} \leq m_1 \leq m_2 \leq m_3 <\diam(G)$, which covers all outstanding cases of the theorem. Note in this case that the three scales $m_1,m_2,$ and $m_3$ satisfy the hypothesis of \cref{prop:iso.subgraph} and are well separated from each other. Indeed, we have that
\[
\frac12\eps_0 m_i^{15-i} \leq |B(o,m_i)| \leq  |B(o,m_i+1)| < \frac12\eps_0 (m_i+1)^{15-i}\le\eps_0m_i^{15-i} \qquad \text{ for each $i=1,2,3$},
\]
from which it follows that
\begin{equation}
\label{eq:m_2>>m_1}
m_2 > m_1^{14/13}-1 \qquad \text{ and } \qquad m_3 > m_2^{13/12}-1
\end{equation}
and hence that
\begin{equation}
\label{eq:m_2>>m_1b}
m_2 \geq 10^{12} \cdot m_1 \geq 10^{168} \vee M^2 \qquad \text{ and } \qquad m_3 \geq 10^{14} \cdot m_2 \geq 10^{182}.
\end{equation}

Let $H\lhd\Aut(G)$ and $\Gamma<\Aut(G/H)$ be the groups given by applying \cref{prop:structure.last.large.ball} with $d=13$. Thus $\Gamma$ acts transitively on $V/H$ and has a nilpotent subgroup of step and index at most $M$, the stabiliser in $\Gamma$ of each orbit $Hv$ with $v\in V$ has size at most $M$, and each such orbit has diameter at most $m_2^{1/2} \vee M$. The orbits of $H$ all have the same cardinality by \cite[Lemma 3.3]{tt.Trof}, and we divide the proof into two cases: the case in which the orbits of $H$ have size at most $(\log|V|)^{400}$, and the case in which they have size greater than $(\log|V|)^{400}$.

\medskip

\textbf{Small $H$-orbits.} 
We begin with the case in which $|Hv| \leq (\log |V|)^{400}$ for every $v\in V$. \cref{prop:cay-ab} implies that $S=\{\gamma\in\Gamma:d(\gamma(Ho),Ho)\le1\}$ is a symmetric generating set for $\Gamma$ of size at most $(k+1)M$, that $\diam_S(\Gamma)=\diam(G/H)$, and that $|\Gamma|=|\Gamma_{Ho}||V/H|$. The bounds on the sizes of the $H$-orbits and their stabilisers thus imply that $|V|/(\log|V|)^{400}\le|\Gamma|\le M|V|$, and hence that
\[
\diam_S(\Gamma)=\diam(G/H)\le\diam(G)\le\frac{\lambda |V|}{(\log|V|)^a}
\preceq \frac{|\Gamma|}{(\log|\Gamma|)^{M}}.
\]
It then follows from Theorem \ref{thm:fin.perc.nilp}, Proposition \ref{prop:cay-ab}, Lemma \ref{lem:bdd.orbits} and \cref{prop:Benj-Schr} that for every $\eps>0$ there exists $q_1=q_1(k,\lambda,\eps)<1$ such that
\[
\P_p^G\bigl(u\leftrightarrow Hv\bigr)\geq \sqrt{1-\eps}
\]
for every $u,v\in V$ and $p\geq q_1$. To complete the proof in this case, it will suffice to prove that there exists $q_2=q_2(k,\lambda,\eps)<1$ such that
\begin{equation}\label{eq:H(u).cluster}
\P_p^G\bigl(u\leftrightarrow v\bigr)\ge \sqrt{1-\eps},
\end{equation}
for every $p \geq q_2$ and every $u,v$ belonging to the same orbit of $H$ since we may then conclude as in the $m_1\leq 10^{156}$ case above.

Since $m_2 \geq M^2$, each orbit of $H$ has diameter at most $m_2^{1/2}$ in $G$ and $H x \subseteq B(o,m_2)$ for every $x\in B(o,m_2')$, where we set $m_2'= \lfloor m_2-m_2^{1/2} \rfloor$, which satisfies $m_2' \geq m_2/2 \geq 2m_1$ by \eqref{eq:m_2>>m_1b}.
 Applying \cref{thm:dimension_drop} with $d=14$ at the scale $m_1$ therefore yields that there exists an absolute constant $0<c_1 \leq 1$ such that $|B(o,m_2')|\geq c_1|B(o,m_2)|$. 
 Applying \cref{prop:iso.subgraph}, we obtain that there exists a constant $\ell_2=\ell_2(k)$ such that for each $v\in V$ there exists a subset $I_{2,v}\subseteq B(v,\ell_2 m_2)$ such that the subgraph of $G$ induced by $I_{2,v}$ satisfies a $13$-dimensional isoperimetric inequality $(\mathrm{ID}_{13,c_2})$ for some $c_2=c_2(k)>0$ and 
\[
|I_{2,v} \cap B(v,m_2)| \geq \left(1-\frac{c_1}{4}\right) |B(v,m_2)| \geq \frac{3}{4} |B(v,m_2)|.
\]
(Here we use the subscript $2$ to remind the reader that $I_{2,v}$ is an object associated to the scale $m_2$.)
Letting $O \subseteq B(o,m_2)$ be the union of the $H$-orbits of the elements of $B(o,m_2')$, we have that $|O| \geq |B(o,m_2')|\geq c_1|B(o,m_2)|$ and hence that
\[
|I_{2,o} \cap O| = |O|-|O\setminus I_{2,o}| \geq |O| - | B(o,m_2) \setminus I_{2,o}| \geq |O| - \frac{c_1}{4}|B(o,m_2)| \geq \frac{3}{4} |O|.
\]
It follows in particular that there exists $y\in B(o,m_2')$ such that
 $|I_{2,o}\cap H y| \geq \frac{3}{4}|H y|$.
Since $\Gamma$ acts transitively on $V/H$, we may apply an automorphism $\gamma$ of $G$ mapping some element of $I_{2,o} \cap H y$ to $o$ to obtain a set $\gamma I_{2,o}$ such that
$o\in \gamma I_{2,o}$ and $|(\gamma I_{2,o})\cap H o | \geq \frac{3}{4}|H o|$. Applying \cref{thm:iso.inf} and \cref{lem:cluster/connect} to the subgraph of $G$ induced by $\gamma I_{2,o}$ yields that there exists $q_2=q_2(k,\eps)$ such that 
\[
\P_p^G(o \leftrightarrow v) \geq (1 - \eps)^{1/4}
\]
for every $v\in(\gamma I_{2,o}) \cap H o$ and $p \geq q_2$. Since $|(\gamma I_{2,o}) \cap Ho|\geq \frac{3}{4} |Ho|$, we may apply Lemma \ref{lem:half.doubles.to.whole2} to conclude that
\[
\P_p^G(o \leftrightarrow v) \geq \sqrt{1 - \eps}
\]
for every $v \in Ho$ and $p \geq q_2$. This immediately implies the claim \eqref{eq:H(u).cluster}.

\medskip

\textbf{Large $H$-orbits.}
We now consider the second case, in which $|Hv| \geq (\log |V|)^{400}$ for every $v\in V$. We will continue to use the sets $(I_{2,v})_{v\in V}$ as constructed in the case of small $H$-orbits.  Note that $B(o,m_2)$ contains the orbit $Ho$, and since $|B(o,m_2)|\le(m_2+1)^{13}$, this implies that $m_2^{13} \succeq (\log|V|)^{400}$, and hence that
\begin{equation}\label{eq:m>log(G)}
m_2\succeq (\log|V|)^{400/13}.
\end{equation}
Note in particular that this implies we may assume $m_2$ to be larger than any given constant depending on $k,\lambda,\eps$, the theorem being trivial for graphs of bounded volume.

Let the constant $\ell_2=\ell_2(k)$ be as in the construction of the sets $I_{2,v}$ above, so that for each $v\in V$ there exists a subset $I_{2,v}\subseteq B(v,\ell_2 m_2)$ such that the subgraph of $G$ induced by $I_{2,v}$ satisfies a $13$-dimensional isoperimetric inequality $(\mathrm{ID}_{13,c_2})$ for some $c_2=c_2(k)>0$ and 
\[
|I_{2,v} \cap B(v,m_2)| \geq \frac{3}{4} |B(v,m_2)|.
\]
By \eqref{eq:m_2>>m_1}, we may assume that $m_2$ is sufficiently large that $m_3 \geq 2(\ell_2 m_2 +1)$. It follows from \cref{lem:linear.rel.growth} and \eqref{eq:m_2>>m_1} that there exists a subset $X_0 \subseteq B(o,m_3-\ell_2 m_2-1)$ satisfying 
\begin{equation}|X_0|\succeq m_3/  m_2 \succeq m_2^{1/12}
\label{eq:X>>m^eps}
\end{equation} that is maximal such that the balls $B(x,\ell_2 m_2)$ with $x\in X_0$ are pairwise disjoint. Lemma \ref{lem:covering} then implies that $B(o,m_3-\ell_2 m_2-1)$ is covered by the sets $B(x,2\ell_2 m_2)$ with $x\in X_0$, so that $B(o,m_3)$ is covered by the sets $B(x, 3\ell_2 m_2+1)$ with $x\in X_0$ and hence that
\begin{equation}\label{eq:B(m')<<B(m)}
|B(o,m_3)|\le|X_0|\cdot |B(o,3 \ell_2 m_2+1)|.
\end{equation} 
We have by \cref{thm:dimension_drop} that there exists a constant $C=C(k)$ such that
 $|B(o,3 \ell_2 m_2+1)|\le C |B(o,m_2)|$, and hence by \eqref{eq:B(m')<<B(m)} and the disjointness of the balls $B(x,\ell_2 m_2)$ with $x\in X_0$ that
\begin{align*}
\Biggl|\bigcup_{x\in X_0}B(x,m_2)\Biggr| = |X_o| \cdot |B(o,m_2)| \ge C^{-1}  |X_0| \cdot |B(o,3 \ell_2 m_2+1)| 
 \geq C^{-1} |B(o,m_3)|.
\end{align*}
As such, there exists a positive constant $\delta=\delta(k)$ such that $\delta \leq 1/4$ and $|\bigcup_{x\in X_0}B(x,m_2)|\geq 8 \delta (k+1) |B(o,m_3)|$.
Note also that
\begin{equation}\label{eq:X_0.in.B(m')}
\bigcup_{x\in X_0}B(x,m_2)\subseteq \bigcap_{v\in V:d(o,v)\le1}B(v,m_3)
\end{equation}
since $X_0 \subseteq  B(o,m_3-\ell_2 m_2-1)$.
 A second application of \cref{prop:iso.subgraph}, this time at the scale $m_3$, implies that 
 there exists a constant $\ell_3=\ell_3(k)$ such that 
 for each $v\in V$ there exists a set $I_{3,v}\subset B(v, \ell_3 m_3)$ such that
\begin{align}
|I_{3,v}\cap B(v,m_3)|\ge(1-\delta)|B(v,m_3)|\label{eq:Iv.big}
\end{align}
and such that the subgraph of $G$ induced by $I_{3,v}$ satisfies a $12$-dimensional isoperimetric inequality with constant depending only on $k$.
 For each $v\in V$, set
\[
J_{3,v}=\bigcap_{u\in V :d(u,v)\le1}I_{3,u}.
\]
(Again, we include the subscript $3$ to emphasize that $J_3$ is associated to the large scale $m_3$.) It follows from \eqref{eq:X_0.in.B(m')} and \eqref{eq:Iv.big} that for each $v$ equal or adjacent to $o$ we have that
\[
\left|\bigcup_{x\in X_0}B(x,m_2) \setminus I_{3,v} \right|\leq  \bigl|B(v,m_3) \setminus I_{3,v}\bigr| \leq \delta |B(v,m_3)| \leq \frac{1}{8(k+1)} \left|\bigcup_{x\in X_0}B(x,m_2) \right|
\]
  and hence that
\[
\left|J_{3,o} \cap\bigcup_{x\in X_0}B(x,m_2)\right|\ge\frac34\left|\bigcup_{x\in X_0}B(x,m_2)\right|.
\]
 Letting $X$ be the set of $x\in X_0$ such that $|J_{3,o}\cap B(x,m_2)|\ge\textstyle\frac12|B(x,m_2)|$, it follows in particular that $|X| \geq |X_0|/2$.
Since $|I_{2,x} \cap B(x,m_2)| \geq \frac{3}{4}|B(x,m_2)|$ for every $x\in V$, it follows moreover that
\begin{equation}
\label{eq:twoscalesbigintersection}
|J_{3,o} \cap I_{2,x} \cap B(x,m_2)|\geq \frac{1}{4} |B(x,m_2)| \geq \frac{1}{4} (\log|V|)^{400}
\end{equation}
for every $x\in X$.

\medskip

For each $x \in V$, let $G_{2,x}$ and $G_{3,x}$ be the subgraphs of $G$ induced by $I_{2,x}$ and $I_{3,x}$, both of which satisfy a $12$-dimensional isoperimetric inequality $(\mathrm{ID}_{12,c})$ for some constant $c=c(k)$. As such, \cref{thm:iso.inf} implies that there exists $p_0=p_0(k,\eps)<1$ and $\eta=\eta(k)>0$ such that if $p\geq p_0$ then 
\begin{equation}
\label{eq:G_2_connectivity}
\P^{G_{2,x}}_p(x \leftrightarrow y) \geq \frac{5}{8}
\end{equation}
for every $x\in V$ and $y \in I_{2,x}$,
\begin{equation}
\label{eq:G3twopoint}
\P^{G_{3,x}}_p(x \leftrightarrow y) \geq (1-\eps)^{1/4}
\end{equation}
for every $x\in V$ and $y \in I_{3,x}$,
 and moreover that
\begin{equation}
\label{eq:big_sets_high_prob}
\P^{G}_p(A \leftrightarrow B) \geq \P^{G_{3,x}}_p(A \leftrightarrow B) \geq 1 - \exp\left[ - \eta \min \{|A|,|B|\}^{10/12} \right]
\end{equation}
for every $x\in V$ and every two non-empty sets $A,B \subseteq I_{3,x}$.

\medskip

Applying Markov's inequality to the random variable $\bigl|J_{3,o} \cap I_{2,x} \cap B(x,m_2)\setminus K_x\bigr|$, it follows by \eqref{eq:twoscalesbigintersection} and \eqref{eq:G_2_connectivity} that
\begin{multline*}
\P^{G_{2,x}}_p\left(\bigl|K_x \cap J_{3,o} \cap I_{2,x} \cap B(x,m_2)\bigr| \geq \frac{1}{16} (\log |V|)^{400} \right)\\
\geq \P^{G_{2,x}}_p\left(\bigl|K_x \cap J_{3,o} \cap I_{2,x} \cap B(x,m_2)\bigr| \geq \frac{1}{4} \bigl|J_{3,o} \cap I_{2,x} \cap B(x,m_2)\bigr| \right) \geq \frac{1}{2}
\end{multline*}
for every $p\geq p_0$ and $x\in X$.
 Since the sets $I_{2,x}$ with $x\in X$ are all disjoint, we may couple Bernoulli-$p$ bond percolation on $G$ with independent copies of Bernoulli-$p$ bond percolation on each $G_{2,x}$ and conclude that
\[
\P_p^G\left(\text{$\exists\,x\in X$ such that $|K_x \cap J_{3,o}\cap I_{2,x}\cap B(x,m_2)|\geq \frac{1}{16} (\log|V|)^{400}$}\right)\ge1-2^{-|X|}
\]
for every $p\geq p_0$.
It follows from the definition of $X$ that $B(x,m_2)\subset B(o,m_3)$ for every $x\in X$, and from \eqref{eq:m>log(G)} and \eqref{eq:X>>m^eps} that 
\[|X|\succeq m_2^{1/12} \succeq (\log|V|)^{400/156}.\] Thus, if $|V|$ is sufficiently large then $|X| \geq (\log |V|)^2$ and we deduce that
\[
\P_p^G\left(\text{$\exists$ an open cluster $K$ such that $|K \cap J_{3,o}\cap B(o,m_3)|\ge\frac{1}{16}(\log|V|)^{400}$}\right)\ge1-2^{-(\log|V|)^2}
\]
for every $p\geq p_0$. 
Since $o$ was taken to be an arbitrary vertex of $G$, it follows by a union bound that
\begin{multline}
\P_p^G\left(\text{$\forall v\in V$ $\exists$ an open cluster $\tilde K(v)$ such that $|\tilde K(v) \cap J_{3,v}\cap B(v,m_3)|\ge\frac{1}{16}(\log|V|)^{400}$}\right)\\
\geq1-|V|\cdot 2^{-(\log|V|)^2}\nonumber
\end{multline}
for every $p\geq p_0$. Since the right hand side tends to $1$ as $|V| \to \infty$ and small values of $|V|$ can be handled by increasing $p$, it follows that there exists $p_1=p_1(k,\eps) \in [p_0,1)$ such that 
\begin{multline}\label{eq:C_v.exist}
\P_p^G\left(\text{$\forall v\in V$ $\exists$ an open cluster $\tilde K(v)$ such that $|\tilde K(v) \cap J_{3,v}\cap B(v,m_3)|\ge\frac{1}{16}(\log|V|)^{400}$}\right)\\
\geq (1-\eps)^{1/4}
\end{multline}
for every $p \geq p_1$.

Let $p_2=p_2(k,\eps)<1$ be defined by $1-p_2=(1-p_0)(1-p_1)$ and let $\sA$ be the event whose probability is estimated in \eqref{eq:C_v.exist}. Let $\omega_0$ and $\omega_1$ be independent copies of Bernoulli-$p_0$ and Bernoulli-$p_1$ percolation on $G$, so that $\omega_2:=\omega_0 \vee \omega_1$ is distributed as Bernoulli-$p_2$ percolation. Write $\P$ for the joint law of $\omega_0,\omega_1,$ and $\omega_2$.  Suppose that $\omega_1 \in \sA$ 
and for each $v \in V$ let $\tilde K(v)= \tilde K(v,\omega_1)$ be an $\omega_1$-open cluster such that $|\tilde K(v) \cap J_{3,v} \cap B(v,m_3)| \ge\frac{1}{16}(\log |V|)^{400}$. Write $\tilde K'(v)=J_{3,v}\cap B(v,m_3)\cap \tilde K(v)$ for each $v\in V$. 
 Note that for every pair $u,v$ of neighbours in $V$ we have that $\tilde K'(u), \tilde K'(v)\subset I_{3,u} \cap I_{3,v}$ by definition of $J_{3,v}$. Since $400 \cdot (10/12) \geq 2$, we may condition on $\omega_1$ and apply \eqref{eq:big_sets_high_prob} to $\omega_0$ to deduce that
 \[
\P(\text{$\omega_1 \in \sA$ and $\tilde K'(u)$ and $\tilde K'(v)$ are connected in $\omega_0$} \mid \omega_1 ) \geq 1 - \exp\left[-\frac{\eta}{16}(\log |V|)^{2}\right]
 \]
 for every two adjacent $u,v\in V$ on the event $\sA$. Since there are at most $|V|^2$ pairs of adjacent vertices, we may take a union bound to deduce that
 \begin{multline*}
\P(\text{$\omega_1\in \sA$ and $\tilde K'(u)$ and $\tilde K'(v)$ are connected in $\omega_0$ for every two adjacent vertices $u,v$} \mid \omega_1 ) 
\\\geq 1 - |V|^2\exp\left[-\frac{\eta}{16}(\log |V|)^{2}\right]
 \end{multline*}
 on the event $\sA$. Letting $\sB$ be the event that there exists an open cluster $K$ having non-trivial intersection with $I_{3,v}$ for every $v\in V$, it follows that
 \[
\P^G_{p_2}(\sB) \geq (1-\eps)^{1/4} \left(1 - |V|^2\exp\left[-\frac{\eta}{16}(\log |V|)^{2}\right]\right).
 \]
 Since the term in parentheses on the right hand side tends to $1$ as $|V|\to \infty$ and small values of $|V|$ can be handled by increasing $p$, it follows that there exists a constant $p_3 = p_3(k,\eps) \in [p_2,1)$ such that $\P^G_p(\sB) \geq (1-\eps)^{1/2}$ for every $p\geq p_3$. Finally, letting $p_4=p_4(k,\eps)<1$ be defined by $1-p_4=(1-p_3)(1-p_0)^2$, we have by a similar coupling argument to above that
\[
\P_{p_4}^G(x \leftrightarrow y) \geq \P^G_{p_3}(\sB) \cdot \min \{\P_{p_0}^G(x \leftrightarrow v) : v\in I_{3,x}\} \cdot \min \{\P_{p_0}^G(y \leftrightarrow v) : v\in I_{3,y}\} \geq 1-\eps
 \]
 for every $x,y\in V$ by \eqref{eq:G3twopoint}, completing the proof. \qedhere
\end{proof}

\subsection*{Acknowledgements} 
The first author was supported in part by ERC starting grant 804166 (SPRS) and thanks Gabor Pete for helpful discussions.
The second author was partially supported by the Stokes Research Fellowship at Pembroke College, Cambridge.
He is also grateful to Itai Benjamini and Ariel Yadin for introducing him to the problems considered in this paper, to Ben Green for a useful discussion on the additive-combinatorics literature, and to Romain Tessera for helpful conversations.

\phantomsection
\addcontentsline{toc}{part}{References}

 \setstretch{1}
 \footnotesize{
  \bibliographystyle{abbrv}
  \bibliography{unimodularthesis.bib}
  }

\end{document}